\PassOptionsToPackage{unicode}{hyperref}
\PassOptionsToPackage{hyphens}{url}
\PassOptionsToPackage{dvipsnames,svgnames,x11names}{xcolor}
\documentclass[
  12pt]{article}
\usepackage{tikz}
\usetikzlibrary{positioning}
 \usetikzlibrary{calc}
\usepackage{amsmath,amssymb,amsthm,mathrsfs}
\usepackage{iftex,bm}
\ifPDFTeX
  \usepackage[T1]{fontenc}
  \usepackage[utf8]{inputenc}
  \usepackage{textcomp} 
\else 
  \defaultfontfeatures{Scale=MatchLowercase}
  \defaultfontfeatures[\rmfamily]{Ligatures=TeX,Scale=1}
\fi
\usepackage{lmodern}
\ifPDFTeX\else  
\fi
\IfFileExists{upquote.sty}{\usepackage{upquote}}{}
\IfFileExists{microtype.sty}{
  \usepackage[]{microtype}
  \UseMicrotypeSet[protrusion]{basicmath} 
}{}
\makeatletter
\@ifundefined{KOMAClassName}{
  \IfFileExists{parskip.sty}{%
    \usepackage{parskip}
  }{
    \setlength{\parindent}{0pt}
    \setlength{\parskip}{6pt plus 2pt minus 1pt}}
}{
  \KOMAoptions{parskip=half}}
\makeatother
\usepackage{xcolor}
\makeatletter
\ifx\paragraph\undefined\else
  \let\oldparagraph\paragraph
  \renewcommand{\paragraph}{
    \@ifstar
      \xxxParagraphStar
      \xxxParagraphNoStar
  }
  \newcommand{\xxxParagraphStar}[1]{\oldparagraph*{#1}\mbox{}}
  \newcommand{\xxxParagraphNoStar}[1]{\oldparagraph{#1}\mbox{}}
\fi
\ifx\subparagraph\undefined\else
  \let\oldsubparagraph\subparagraph
  \renewcommand{\subparagraph}{
    \@ifstar
      \xxxSubParagraphStar
      \xxxSubParagraphNoStar
  }
  \newcommand{\xxxSubParagraphStar}[1]{\oldsubparagraph*{#1}\mbox{}}
  \newcommand{\xxxSubParagraphNoStar}[1]{\oldsubparagraph{#1}\mbox{}}
\fi
\makeatother

\usepackage{longtable,booktabs,array}
\usepackage{calc} 
\usepackage{etoolbox}
\makeatletter
\patchcmd\longtable{\par}{\if@noskipsec\mbox{}\fi\par}{}{}
\makeatother
\IfFileExists{footnotehyper.sty}{\usepackage{footnotehyper}}{\usepackage{footnote}}
\makesavenoteenv{longtable}
\usepackage{graphicx}
\makeatletter
\def\maxwidth{\ifdim\Gin@nat@width>\linewidth\linewidth\else\Gin@nat@width\fi}
\def\maxheight{\ifdim\Gin@nat@height>\textheight\textheight\else\Gin@nat@height\fi}
\makeatother
\setkeys{Gin}{width=\maxwidth,height=\maxheight,keepaspectratio}
\makeatletter
\def\fps@figure{htbp}
\makeatother

\makeatletter
\@ifpackageloaded{caption}{}{\usepackage{caption}}
\AtBeginDocument{%
\ifdefined\contentsname
  \renewcommand*\contentsname{Table of contents}
\else
  \newcommand\contentsname{Table of contents}
\fi
\ifdefined\listfigurename
  \renewcommand*\listfigurename{List of Figures}
\else
  \newcommand\listfigurename{List of Figures}
\fi
\ifdefined\listtablename
  \renewcommand*\listtablename{List of Tables}
\else
  \newcommand\listtablename{List of Tables}
\fi
\ifdefined\figurename
  \renewcommand*\figurename{Figure}
\else
  \newcommand\figurename{Figure}
\fi
\ifdefined\tablename
  \renewcommand*\tablename{Table}
\else
  \newcommand\tablename{Table}
\fi
}
\@ifpackageloaded{float}{}{\usepackage{float}}
\floatstyle{ruled}
\@ifundefined{c@chapter}{\newfloat{codelisting}{h}{lop}}{\newfloat{codelisting}{h}{lop}[chapter]}
\floatname{codelisting}{Listing}

\makeatother
\makeatletter
\@ifpackageloaded{caption}{}{\usepackage{caption}}
\@ifpackageloaded{subcaption}{}{\usepackage{subcaption}}
\makeatother

\ifLuaTeX
  \usepackage{selnolig}  
\fi
\usepackage[]{natbib}
\usepackage{bookmark}

\IfFileExists{xurl.sty}{\usepackage{xurl}}{} 
\hypersetup{
  pdftitle={Title},
  pdfauthor={Author 1; Author 2},
  pdfkeywords={3 to 6 keywords, that do not appear in the title},
  colorlinks=true,
  linkcolor={blue},
  filecolor={Maroon},
  citecolor={Blue},
  urlcolor={Blue},
  pdfcreator={LaTeX via pandoc}}

\newcommand{\anon}{1}

\DeclareMathOperator*{\argmax}{arg\,max}

\newcommand{\norm}[1]{\left\lVert #1 \right\rVert}

\newcommand{\bmu}{\boldsymbol{\mu}}
\newcommand{\tr}{\mathrm{tr}}
\newcommand{\calB}{\mathcal{B}}

\newcommand{\calD}{\mathcal{D}}
\newcommand{\calL}{\mathcal{L}}
\newcommand{\calP}{\mathcal{P}}
\newcommand{\calH}{\mathcal{H}}

\newcommand{\calQ}{\mathcal{Q}}
\newcommand{\calN}{\mathcal{N}}
\newcommand{\R}{\mathbb{R}}

\newcommand{\bigo}{\mathcal{O}}

\newcommand{\bE}{\mathbf{E}}
\newcommand{\bW}{\boldsymbol{W}}
\newcommand{\bX}{\boldsymbol{X}}
\newcommand{\bR}{\mathbf{R}}
\newcommand{\bpix}{\boldsymbol{\pi}_X}
\newcommand{\bpis}{\boldsymbol{\pi}_S}
\newcommand{\bPix}{\boldsymbol{\Pi}_X}
\newcommand{\bPis}{\boldsymbol{\Pi}_S}
\newcommand{\bPi}{\boldsymbol{\Pi}}
\newcommand{\bpi}{\boldsymbol{\pi}}
\newcommand{\hatpi}{\widehat{\boldsymbol{\pi}}}
\newcommand{\hatPi}{\widehat{\boldsymbol{\Pi}}}

\newcommand{\bY}{\boldsymbol{Y}}
\newcommand{\bM}{\mathbf{M}}
\newcommand{\bB}{\mathbf{B}}
\newcommand{\bK}{\mathbf{K}}
\newcommand{\by}{\boldsymbol{y}}
\newcommand{\bu}{\boldsymbol{u}}

\newcommand{\bGam}{\boldsymbol{\Gamma}}
\newcommand{\bQ}{\boldsymbol{Q}}

\newcommand{\bT}{\mathbf{T}}
\newcommand{\br}{\boldsymbol{r}}
\newcommand{\bV}{\mathbf{V}}
\newcommand{\bS}{\boldsymbol{S}}

\newcommand{\snr}{\mathrm{SNR}}
\newcommand{\bdiag}{\mathrm{bdiag}}
\newcommand{\bDel}{\boldsymbol{\Delta}}
\newcommand{\bA}{\mathbf{A}}
\newcommand{\boldW}{\mathbb{W}}
\newcommand{\boldY}{\mathbb{Y}}
\newcommand{\boldX}{\mathbb{X}}
\newcommand{\boldM}{\mathbb{M}}
\newcommand{\boldV}{\mathbb{V}}
\newcommand{\bZ}{\boldsymbol{Z}}

\newcommand{\bP}{\mathbf{P}}
\newcommand{\bLam}{\boldsymbol{\Lambda}}
\newcommand{\blam}{\boldsymbol{\lambda}}
\newcommand{\bx}{\boldsymbol{x}}
\newcommand{\bI}{\mathbf{I}}
\newcommand{\beps}{\boldsymbol{\epsilon}}
\newcommand{\bet}{\boldsymbol{\eta}}
\newcommand{\bSig}{\boldsymbol{\Sigma}}
\newcommand{\Siginv}{\boldsymbol{\Sigma}^{-1}}

\newcommand{\pr}{\mathbb{P}}
\newcommand{\E}{\mathbb{E}}

\newcommand{\bSigma}{{\boldsymbol{\Sigma}}}

\newcommand{\bTh}{\boldsymbol{\Theta}}
\newcommand{\bth}{\boldsymbol{\theta}}

\newcommand{\tilX}{\widetilde{\boldsymbol{X}}}
\newcommand{\tilY}{\widetilde{\boldsymbol{Y}}}

\newcommand{\bgam}{\boldsymbol{\gamma}}
\newcommand{\bzero}{\boldsymbol{0}}
\newcommand{\Prob}{\mathbb{P}}

\newcommand{\betahat}{\hat{\beta}}

\newtheorem{theorem}{Theorem}

\newtheorem{lemma}{Lemma}
\newtheorem{proposition}{Proposition}

\newtheorem{assumption}{Assumption}

\usepackage{mathtools}
\usepackage{nicefrac} 
\usepackage{hyperref}
\usepackage{xcolor}
\usepackage{algorithm}
\usepackage{algpseudocode}
\definecolor{darkgreen}{RGB}{0,100,0}

\usepackage{xr-hyper}

\begin{document}

\def\spacingset#1{\renewcommand{\baselinestretch}%
{#1}\small\normalsize} \spacingset{1}


\if1\anon
{
  \title{\bf Doubly-Unlinked Regression for Dependent Data}
  \author{Anik Burman$^{1}$, Sayantan Choudhury$^{2}$, Debangan Dey$^{3}$\thanks{
    Address for correspondence: debangan@tamu.edu}\vspace{3mm}\\
    $^{1}$Department of Biostatistics, Johns Hopkins University\\
    $^2$Department of Statistics and Data Science, MBZUAI \\
    $^3$Department of Statistics, Texas A\&M University}
    \date{}
  \maketitle
} \fi

\if0\anon
{
  \bigskip
  \bigskip
  \bigskip
  \begin{center}
    {\LARGE\bf Title}
\end{center}
  \medskip
} \fi
\bigskip
\date{}
\begin{abstract}
Shuffled regression concerns settings in which covariates and responses are observed without their correct pairing. In dependent-data problems, a second form of missing correspondence can arise when responses are also detached from the latent temporal, spatial, or geometric domain that induces their dependence structure. We study regression under this joint loss of correspondence and, to our knowledge, provide the first systematic treatment of this setting. Specifically, we consider a doubly-unlinked regression model in which both the covariate--response link and the response--domain link are unknown, represented by two latent permutation matrices, while dependence is induced by an unobserved stochastic process. This framework unifies shuffled regression and latent-domain permutation models within a common dependent-data setting. We characterize signal-to-noise regimes governing recovery of the regression parameter and the latent permutations, and show that consistent estimation of the regression coefficient can be achieved under strictly weaker conditions than exact permutation recovery. To address the combinatorial difficulty of inference, we develop REPAIR, a variational Bayes method based on a block-structured permutation model that captures localized scrambling while substantially reducing computational complexity. Simulations and an applied example illustrate the empirical behavior of REPAIR and support the theoretical results.
\end{abstract}

\noindent%
\vfill

\newpage
\spacingset{1.8} 

\section{Introduction}\label{sec-intro}

Regression with unknown correspondence, also referred to as shuffled regression, unlinked regression, or unlabeled sensing, concerns settings in which the usual pairing between covariates and responses is unavailable. Rather than observing matched pairs, the analyst observes the covariate and response collections up to an unknown permutation. The inferential task therefore, differs fundamentally from classical regression: estimation of the regression parameters must be carried out simultaneously with recovery of latent combinatorial structure. The problem has attracted substantial recent attention because correspondence errors arise naturally in modern data integration tasks such as record linkage across databases and distributed sensing systems \citep{unnikrishnan2015unlabeled, pananjady2017linear, pananjady2018, hsu2017, zhang2020optimal, slawski2019, beuthner2021data}.

A closely related but distinct source of missing correspondence arises in dependent-data settings, where responses are indexed by an underlying temporal, spatial, or latent geometric domain, but their alignment with that domain is unknown. Problems of this type appear in seriation problems in archaeology and anthropology \citep{banning2020seriation}, pseudo-time analysis in single-cell transcriptomics \citep{gu2022bayesian}, disease progression modeling \citep{wijeratne2024unscrambling}, spike sorting in dense neural arrays \citep{rossant2016spike, pachitariu2024spike}, sensor localization \citep{el2023localizing}, and privacy-sensitive environmental epidemiology \citep{lin2023geo, cassa2008re}. 

Although the current literature has largely addressed broken covariate–response and response–domain links separately, modern datasets often exhibit simultaneous failures of both. In privacy-sensitive digital mental health studies, record-linkage errors may disrupt covariate–response pairings while locations are masked for privacy, breaking the connection to the spatial domain. Spatial transcriptomics provides an analogous example: sample mixing can scramble cell-to-profile correspondence while positional information may be lost during batch processing. Motivated by these situations, we study the theoretical implications of the doubly-unlinked regime (Fig. \ref{fig:unlinked-reg-illustration}) and propose a novel method to handle it. Let $\bX=(X(s_1),\dots,X(s_n))^\top$ denote the covariate vector, let $\bY=(Y(s_1),\dots,Y(s_n))^\top$ denote the response vector, and let $\bW=(W(s_1),\dots,W(s_n))^\top$ denote a latent stochastic process at the sampled domain points $s_1,\dots,s_n$. We consider the Data Generating Process (DGP)
\begin{equation}
\label{eq:dgp}
\bY=\bPix \bX \beta+\bPis \bW+\beps,
\end{equation}
where $\bPix$ and $\bPis$ are unknown permutation matrices. The permutation $\bPix$ breaks the correspondence between $\bX$ and $\bY$, while $\bPis$ breaks the correspondence between $\bY$ and the latent domain through which dependence is induced. The DGP \eqref{eq:dgp} contains the usual shuffled regression model and latent-domain permutation model as special cases, while also allowing the two sources of missing correspondence to operate jointly.

\begin{figure}[!ht]
    \centering
    \includegraphics[width=\textwidth]{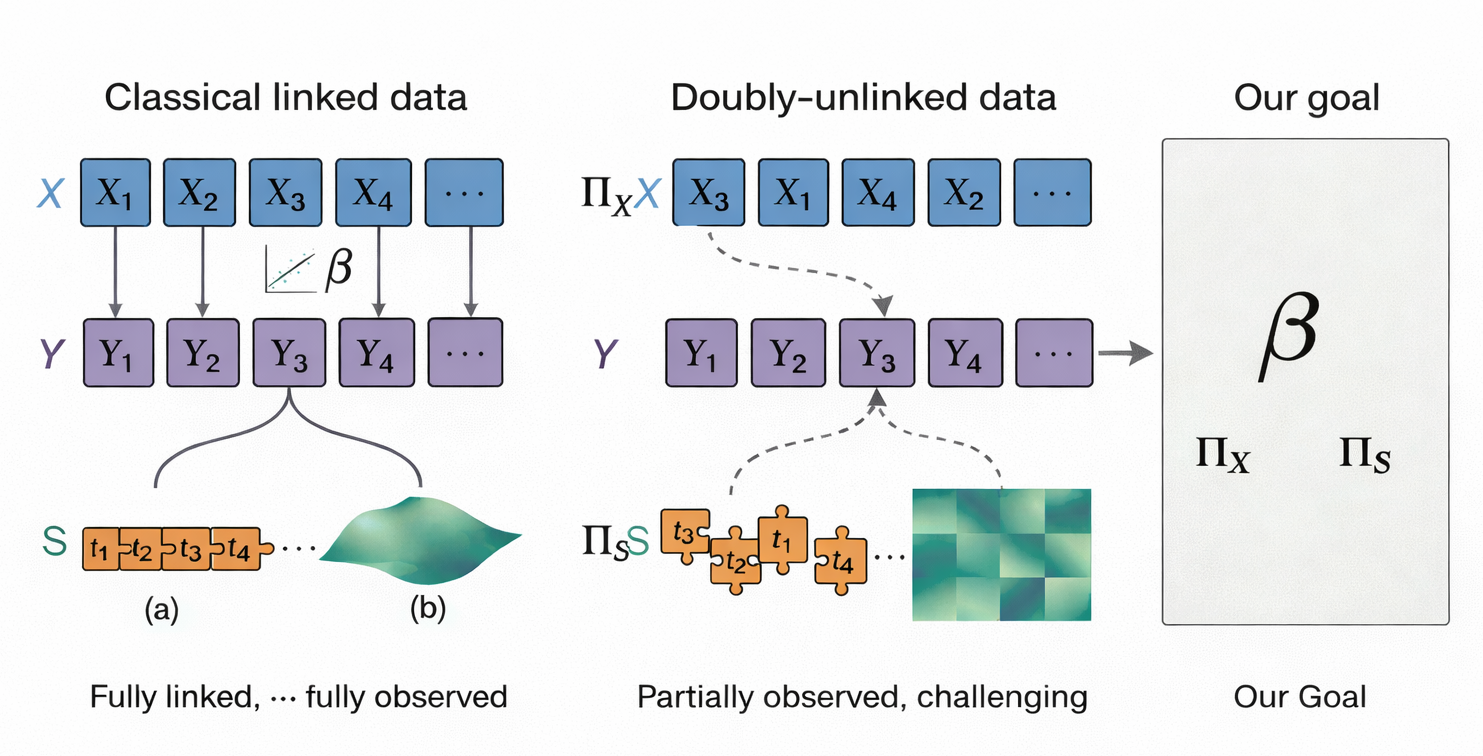}
    \caption{Schematic illustration of the doubly-unlinked regression setting. The left panel shows the classical linked-data setup with aligned exposures, outcomes, and domain variables. The middle panel illustrates the doubly-unlinked case where exposures and domain variables are observed after unknown permutations. The goal is to recover the regression effect $\beta$ and the permutation matrices $\bpix$ and $\bpis$.}
    \label{fig:unlinked-reg-illustration}
\end{figure}

The statistical and computational structure of \eqref{eq:dgp} differs qualitatively from either single-link problem, and to our knowledge has not been systematically studied. Computationally, optimization over two unknown permutations yields a search space of size $(n!)^2$: while the covariate permutation $\bPix$
is polynomial-time solvable in isolation, the latent-domain permutation $\bPis$ leads to a quadratic assignment problem, which
is NP-hard \citep{loiola2007}, and their coupling makes the combined problem substantially more difficult. Statistically, the joint problem is not a simple combination of the two single-link settings: $\beta$ must be inferred simultaneously against covariate mismatch and a permuted dependence structure, latent dependence alters the information geometry, and the recoverability of $\beta$
and its relationship to exact permutation recovery requires new analysis.

The theory on unlinked regression has primarily focused on settings where only the covariate--response correspondence is unknown. Foundational work in unlabeled sensing and shuffled linear regression established identifiability and recovery guarantees under independent and identically distributed (iid) models, characterizing both statistical and computational limits \citep{unnikrishnan2015unlabeled, pananjady2017linear, pananjady2018}. Building on these results, subsequent research developed polynomial-time relaxations and optimal estimators, and more recent methods leverage spectral or graph-based approaches to improve computational efficiency and robustness \citep{hsu2017, zhang2020optimal, liu2025shuffled, liu2024blind}.

This paper makes three contributions. First, we formulate a doubly-unlinked regression model that unifies shuffled regression and latent-domain permutation problems within a common dependent-data framework. This formulation makes it possible to study how regression and dependence interact under simultaneous loss of correspondence.

Second, we characterize signal-to-noise regimes governing recovery of the regression parameter and the permutation matrices. Our analysis shows that the condition required for consistent estimation of $\beta$ is strictly weaker than that required for exact permutation recovery. Thus, regression inference can remain feasible even when exact reconstruction of the latent correspondences is information-theoretically unattainable.

Third, we develop a computationally tractable variational Bayes procedure for joint inference on the regression parameter, the latent permutations, and the covariance parameters of the latent process. Because unrestricted inference under \eqref{eq:dgp} is combinatorially prohibitive, we work under a block-structured permutation model that captures localized scrambling while substantially reducing complexity. We study the proposed method through theory, simulation, and an applied analysis.

The remainder of the paper is organized as follows. Section~\ref{sec:dgp} introduces the data-generating process and formal assumptions. Section~\ref{sec:theory} presents the main theoretical results, including signal-to-noise conditions for recovery of the regression parameter and the latent permutations. Section~\ref{sec:vb} develops REPAIR, our variational Bayes method for computationally tractable joint inference. Section~\ref{sec:simulation} investigates the finite-sample performance of REPAIR in a range of simulation settings, Section~\ref{sec:data} presents an applied example, and Section~\ref{sec:discussion} concludes with a discussion of the broader implications of the proposed framework and possible future directions.

\section{Background and Problem Setup}\label{sec:dgp}

\textbf{Notation.}
We use the standard order notation $\bigo(\cdot)$ and $\Omega(\cdot)$ throughout. For two nonnegative sequences $a_n$ and $b_n$, we write $a_n=\bigo(b_n)$ if there exist constants $C>0$ and $N$ such that $a_n\le C b_n$ for all $n\ge N$, and $a_n=\Omega(b_n)$ if there exist constants $c>0$ and $N$ such that $a_n\ge c b_n$ for all $n\ge N$. For a vector $\bx\in\R^d$, $\|\bx\|_2$ denotes the Euclidean norm. For a matrix $\bA$, $\|\bA\|_2$ denotes the spectral norm unless stated otherwise.

\textbf{Formal Problem Setup.} We consider the data generating process \eqref{eq:dgp} over a latent domain $\calD\subset\R^k$. Let $\{s_i\}_{i=1}^n\subset\calD$ denote observed domain points, and let $\bW$ 
be a mean-zero latent process with covariance matrix $\bSig^\ast$. It represents structured dependence induced by the domain $\calD$. Such dependence is ubiquitous in time-series analysis, spatial statistics, longitudinal studies, sensor networks, and manifold-based biological data. Even if individual-level correspondences are scrambled, the covariance structure induced by $W(s_i)$ reflects the geometry or ordering of the underlying domain. Our primary inferential objective is estimation of $\beta$, with recovery of the permutations treated as a secondary objective when feasible. We begin by stating the assumptions imposed on the DGP in \eqref{eq:dgp}. Additional structural restrictions used to make inference tractable will be introduced subsequently.

\begin{assumption}[Data Generation Assumption]\label{assume:dgp}
We assume that the data generation process satisfies:
\begin{itemize}
    \item The exposure vector $\bX \sim \calN_n\left(\bzero, \sigma_X^2 \bI_n \right)$.
    \item $W(s_i)$ is a realization from $\mathcal{GP}(0, C(\cdot,\cdot))$ where the covariance function $C(\cdot,\cdot)$ belongs to a known isotropic class indexed by parameters $\bGam = \left(\gamma_1, \cdots, \gamma_c\right)^\top$.
    \item $\bX$ is independent of $\bW$, i.e., there is no latent endogeneity.
    \item The random error vector $\beps \sim \calN_n \left(\bzero, \tau^2 \bI_n\right)$ is independent of $\bX$ and $\bW$.
\end{itemize}
\end{assumption}

Under Assumption~\ref{assume:dgp}, the random effects vector observed at the $n$ domain indices $\bS=(s_1,\dots,s_n)^{\top}$ satisfies $\bW \sim \calN_n(\bzero, \bSig^\ast)$ where $\bSig^\ast = \sigma^2\bR_{\bth}$ and $\bR_{\bth}$ is the corresponding correlation matrix induced by $C(\cdot,\cdot)$ evaluated at $\bS$. Marginalizing over $\bW$ yields $\bY|\bX \sim  \calN_n\left(\bPix\bX\beta, \bPis\left(\bSig^\ast + \tau^2 \bI_n\right)\bPis^\top\right)$. Defining $\bSig = \bSig^\ast + \tau^2 \bI_n$, the conditional covariance of $\bY$ is $\bPis\bSig\bPis^\top$.

Because the latent process $\bW$ is not observed, direct inference on $\bPis$ is generally infeasible. For theoretical purposes, it is convenient to reparameterize $(\bPis,\bPix)$ through the one-to-one transformation $(\bPis,\bPix)\mapsto (\bPi_1,\bPi_2)$ defined by $\bPi_1=\bPis^\top \bPix$ and $\bPi_2=\bPis^\top$. We emphasize that this reparameterization is introduced only to facilitate analysis; the methodological development will be presented in the original parametrization. Left-multiplying \eqref{eq:dgp} by $\bPi_2$ gives $\bPi_2 \bY = \bPi_1 \bX \beta + \bW^\ast$, where $\bW^\ast \sim \calN_n(\bzero,\bSig)$. This representation recasts the problem as a Generalized Least Squares (GLS) version of shuffled regression under correlated errors.
Assuming that $\bSig$ is known, define $\bTh \coloneqq \left(\bPi_1, \bPi_2, \beta \right)$ and consider the GLS loss function
\begin{equation}\label{eq:loss-fn}
    \begin{split}
            \mathcal{L}(\bTh)  \coloneqq & \left|\left|\widetilde{\bY}_{\bPi_2} - \widetilde{\bX}_{\bPi_1}\beta\right|\right|^2_2 \\
     = & \underbrace{\left|\left|\bP_{\bPi_1,X}^\perp\widetilde{\bY}_{\bPi_2}\right|\right|^2_2}_{\calL_1(\bTh)} + \underbrace{\left|\left|\bP_{\bPi_1,X}\widetilde{\bY}_{\bPi_2} - \widetilde{\bX}_{\bPi_1}\beta\right|\right|^2_2}_{\calL_2(\bTh)}, 
    \end{split}
\end{equation}
where $\widetilde{\bX}_{\bPi_1} \coloneqq {\bSig}^{- \nicefrac{1}{2}}\bPi_1 \bX$, $\widetilde{\bY}_{\bPi_2} \coloneqq {\bSig}^{-\nicefrac{1}{2}}\bPi_2 \bY$, and $\bP_{\bPi_1, X} \coloneq \widetilde{\bX}_{\bPi_1}\left(\widetilde{\bX}_{\bPi_1}^\top \widetilde{\bX}_{\bPi_1}\right)^{-1}\widetilde{\bX}_{\bPi_1}^\top$. We need to minimize the loss function \eqref{eq:loss-fn} in order to estimate $\bTh$.

\textbf{Block Permutation Structure.}
The fully unrestricted permutation model allows $(n!)^2$ possible pairs $(\bPi_1,\bPi_2)$, rendering both theoretical characterization and computation intractable for even moderate $n$. Moreover, in many practical settings, scrambling mechanisms are not globally arbitrary but instead exhibit localized structure. To reflect this and to obtain a statistically meaningful intermediate regime between complete correspondence and fully adversarial permutations, we impose a structured restriction on $\bPi_1$ and $\bPi_2$.

For $n$ sampled domain indices, we partition them into $B$ blocks of size $K$, so that $n = KB$. Within each block, exposure values and domain indices may be permuted, but no permutation occurs across blocks. This restriction models localized scrambling mechanisms such as batching, windowed anonymization, temporal shuffling within short time segments, or partial reordering within contiguous domain neighborhoods. In such settings, coarse-scale ordering across blocks is preserved, while fine-scale correspondences within blocks are obscured.

From a statistical perspective, the block structure creates an intermediate inferential regime: the latent dependence induced by $W(s_i)$ continues to operate across blocks, preserving large-scale covariance structure, while local permutations disrupt short-range alignment between $\bX$ and $\bY$. This separation allows us to study whether regression recovery can exploit global dependence even when local correspondences are scrambled. The block model, therefore, isolates the interaction between dependence and permutation noise without requiring recovery over the full combinatorial space.

We further assume that, for the two $K \times K$ permutation matrices $\bpi_1 \coloneqq \bpix^\top \bpis$ and $\bpi_2 \coloneqq \bpis^\top$, the corresponding full permutation matrices of size $n = KB$ satisfy
$\bPi_u = \operatorname{bdiag}(\bpi_u,\dots,\bpi_u)$ for $ u = 1,2$ where $\operatorname{bdiag}()$ denotes block diagonal matrix. Thus, the same within-block permutation is repeated across all $B$ blocks. This assumption is convenient both analytically and computationally, as it imposes a simple structure while still allowing for local scrambling within each block. It is also natural in settings where observations are perturbed according to a common mechanism across blocks; for example, when data are masked prior to release, the same permutation rule may be applied repeatedly within each block. Under this specification, there is a correspondence between $\bPi_X$ and $\bPi_S$ with the corresponding block matrices $\bpix$ and $\bpis$. Similarly any estimate $\hatPi_u$ can be represented as $\bdiag(\hatpi_u)$. For Hamming distance $d_H$ between $\bI_K$ and $\bpi_u$ equal to $d_H(\bI_K,\bpi_u) = k_u$, we have $d_H(\bI_n,\bPi_u) \coloneq h_u = k_uB$, and analogous definitions apply to $\hatpi_u$ and $\hatPi_u$ with corresponding $\hat{k}_u$. We will use these notations throughout.

Under the Gaussian assumptions as seen in Assumption \ref{assume:dgp}, minimizing the GLS loss \eqref{eq:loss-fn} over $\bTh = (\bPi_1, \bPi_2, \beta)$ within the block permutation class coincides with Maximum Likelihood Estimation (MLE). We therefore study the consistency of the maximum likelihood estimators of $\beta$, $\bPi_1$, and $\bPi_2$ under the block model and the assumed DGP. In particular, we ask whether consistent estimation of $\beta$ requires exact recovery of the permutations, or whether the latent dependence structure induced by $W$ permits regression recovery under weaker signal-to-noise conditions. The next section establishes theoretical results characterizing the regimes governing permutation and regression consistency.

\section{Theoretical Guarantees for MLE}\label{sec:theory}
\textbf{Signal-to-Noise Ratio.} The recoverability of the regression parameter $\beta$ and the permutation matrices $(\bPi_X,\bPi_S)$ under DGP \eqref{eq:dgp} depends on the relative strength of the linear signal compared to the latent-domain variability. To quantify this balance, we define the Signal-to-Noise Ratio (SNR) as
\begin{eqnarray*}
\snr \coloneq \dfrac{\beta^2}{\lambda_{\max}(\bSig)\kappa(\bSig)},
\end{eqnarray*}
where $\lambda_{\max}(\bSig)$ is the largest eigenvalue of $\bSig = \sigma^2\bR_\theta +\tau^2\bI_n$ and $\kappa(\bSig)$ is its condition number. This definition reflects both the magnitude of the regression signal $\beta$ and the complexity induced by dependence in the latent process.

This SNR is well-defined whenever the eigenvalues of $\bSig$ are bounded away from zero and infinity, ensuring that both $\lambda_{\max}(\bSig)$ and $\kappa(\bSig)$ are finite. Such conditions hold in many dependent-data settings of interest. In the spatial setting, \citet{zhan2024neural} establishes bounded minimum and maximum eigenvalues for covariance matrices arising from the Mat\'ern family under growing-domain asymptotics, which in turn implies a finite condition number. In the time-series setting, if the latent process is weakly stationary, then its covariance matrix is Toeplitz, since its entries depend only on temporal lag. Classical Toeplitz theory then links the eigenvalues of such covariance matrices to the spectral density, implying that bounded spectral densities yield finite and uniformly controlled eigenvalues; see, for example, \citet{gray2006toeplitz} and \citet{xiao2012covariance}.

In the special case $\bSig = \tau^2 \bI_n$, corresponding to iid errors with no latent-domain dependence, we have $\lambda_{\max}(\bSig)=\tau^2$ and $\kappa(\bSig)=1$, so that $\snr$ reduces to $\beta^2/\tau^2$, recovering the classical definition used in iid shuffled regression settings \citep{pananjady2017linear,slawski2019}. Thus, our definition extends the standard SNR to dependent-data settings.

\textbf{Rates for Recovery.} Here we study recovery guarantees for maximum likelihood estimation under the block permutation model. The most general inferential objective is joint estimation of the block diagonal components of the permutation matrices $(\bpi_1,\bpi_2)$ in the space $\calP_K \times \calP_K$ where $\calP_K$ denotes the space of all $K$ dimensional permutation matrix,  along with the regression parameter $\beta$. Under Gaussian assumptions, minimizing the loss $\mathcal{L}(\bTh)$ defined in Equation \eqref{eq:loss-fn} yields the maximum likelihood estimator (MLE). Notice that $\calL(\bTh)$ can be factored into two parts $\calL(\bTh) = \calL_1(\bTh) + \calL_2(\bTh),$ where $\calL_1(\bTh) = \left|\left|\bP_{\bPi_1,X}^\perp\widetilde{\bY}_{\bPi_2}\right|\right|^2_2$ which governs permutation alignment and $\calL_2(\bTh) = \left|\left|\bP_{\bPi_1,X}\widetilde{\bY}_{\bPi_2} - \widetilde{\bX}_{\bPi_1}\beta\right|\right|^2_2$ governs regression estimation. For a given value of $(\hatpi_1,\hatpi_2)$ obtained by optimizing the loss function $\calL_1(\bTh)$, $\calL_2$ can be exactly set to 0, which will give us the closed form estimate of $\beta$ given by $\hat{\beta}  = \left(\widetilde{\bX}_{\hatPi_1}^\top  \widetilde{\bX}_{\hatPi_1}\right)^{-1} \widetilde{\bX}_{\hatPi_1}^\top \widetilde{\bY}_{\hatPi_2}.$  We first characterize conditions under which the permutation matrices are exactly recoverable which will directly imply $\beta$ is recoverable.

\begin{theorem}\label{thm:error-pi1-pi2}
        For $\snr = \Omega(K^\alpha)$ with $\alpha >1$ and $B \geq B_\ast(K,\alpha) \coloneq \frac{\alpha\log K}{K^\alpha - K}$, with the MLE estimator $\left( \hatPi_{1, \texttt{ML}}, \hatPi_{2, \texttt{ML}} \right) = \underset{(\bpi_1,\bpi_2)\in \calP_K\times \calP_K}{\argmax} \left|\left|\bP_{\bPi_1,X}^\perp\widetilde{\bY}_{\bPi_2}\right|\right|^2_2$, we have
    \begin{eqnarray*}
        \Prob \left( \left( \hatPi_{1, \texttt{ML}}, \hatPi_{2, \texttt{ML}} \right) \neq \left( \bPi_1, \bPi_2 \right) \right) = \bigo\left(K^4 \snr ^{-c^\ast_1}e^{-2c^\ast_1B}\right) + \bigo\left(K^2 e^{-c^\ast_2 B}\right),
    \end{eqnarray*}
    where the total sample size $n = KB$ for some universal constant $c^\ast_1,c^\ast_2 >0$.
\end{theorem}

The proof is provided in Appendix \ref{thm:error-pi1-pi2-proof}. The theorem suggests that under sufficiently strong signal-to-noise conditions that grow polynomially in the block size $K$, the permutation matrices can be recovered with exponentially decaying error probability in the number of blocks $B$. Throughout our analysis, the block size $K$ is treated as fixed, while the number of blocks $B$ is allowed to grow, so that the total sample size $n = KB$ increases through $B$. In this regime, the lower bound on the required SNR depends only on $K$ and does not scale with the total sample size $n$. This contrasts with the fully unstructured permutation setting studied in \cite{pananjady2017linear}, where the SNR requirement scales with $n^\alpha$. By restricting permutations to local neighborhoods of fixed size, the block structure yields substantially weaker signal requirements while still allowing consistent recovery as $B \to \infty$.

However, exact recovery of permutations may not be necessary for consistent estimation of the regression parameter. The next result establishes that $\beta$ can be consistently estimated under strictly weaker conditions.

\begin{theorem}\label{thm:error-betahat}
    Assuming $\log \snr > 1$ and the number of blocks $B = \Omega\left(\nicefrac{(1+\gamma)\log K}{\phi(\snr)}\right)$  with $\gamma >0$, for the estimate $\betahat = \left(\widetilde{\bX}_{\hatPi_1}^\top  \widetilde{\bX}_{\hatPi_1}\right)^{-1} \widetilde{\bX}_{\hatPi_1}^\top  \widetilde{\bY}_{{\hatPi}_2}$ where the estimates $\left(\hatPi_1,\hatPi_2\right) = \underset{(\bpi_1,\bpi_2)\in \calP_K\times \calP_K}{\argmax} \left|\left|\bP_{\bPi_1,X}^\perp\widetilde{\bY}_{\bPi_2}\right|\right|^2_2$, for some $\delta < 1$, we have:
    $$
    \left|\betahat - \beta\right| \leq \sqrt{5\lambda_{\max}(\bSig)\cdot\kappa(\bSig)}\cdot \dfrac{n^{-(1 - \delta)/2}}{1 - n^{-(1 - \delta)/2}}
    $$
    with probability 
    $$
    p = 1 - c_1 K^2 e^{-2B \phi(\snr)} - 2\exp\left(- n^\delta\right),
    $$
    where $n = KB$ and $\phi(\snr) \coloneq \frac{(\log \snr-1)^2}{\log \snr}$ and $c_1 >0$ is a universal constant.
\end{theorem}

The proof is given in Appendix \ref{thm:error-betahat-proof}. This theorem shows that it is possible to consistently estimate the effect size $\beta$, without necessarily correctly estimating the permutation matrices. This provides a motivation for masking or anonymization procedures that preserve population-level association between sensitive variables without requiring individual-level alignment.

\textbf{Limitation of MLE.} Despite the theoretical guarantees established above, directly solving the maximum likelihood problem is computationally infeasible in practice. The optimization over $(\bPi_1,\bPi_2)$ requires searching over discrete permutation spaces and can be formulated as a variant of the Quadratic Assignment Problem (QAP), which is known to be NP-hard. Consequently, exact optimization becomes computationally prohibitive even for moderately sized problems. In particular, under the block restriction described earlier, the number of possible permutation pairs grows as $(K!)^2$ within each block, leading to an exponential increase in the search space as $K$ increases.

Beyond the combinatorial complexity, the likelihood surface itself poses additional challenges. The objective function is highly nonconvex due to the interaction between the permutation matrices and the covariance structure $\bSig$. In particular, the permutations enter the likelihood through quadratic forms involving $\bSig^{-1}$, which couples the effects of $(\bPi_1,\bPi_2)$ in a nonlinear manner. As a result, naive optimization strategies such as greedy search or local swap-based updates can easily become trapped in suboptimal configurations and may exhibit substantial sensitivity to initialization. These issues make direct likelihood-based estimation unreliable and computationally unstable for moderate-to-large problem sizes.

To address these challenges, we adopt a scalable approximation strategy. Rather than attempting to solve the discrete optimization problem directly, we introduce a Bayesian formulation in which the unknown permutation matrices and model parameters are treated as latent random variables. This perspective allows us to replace the combinatorial search with inference over a continuous relaxation of the permutation space. In the next section, we develop a variational inference framework that approximates the posterior distribution of the parameters while remaining computationally tractable. The resulting algorithm provides a practical approximation to the maximum likelihood solution and scales efficiently to moderate-to-large samples.

\section{Variational Inference for Joint Permutation and Parameter Estimation}\label{sec:vb}
In this section, we present our \textbf{RE}gression with \textbf{P}ermutation \textbf{A}lignment via Variational \textbf{I}nfe\textbf{R}ence (REPAIR) method for parameter estimation. As discussed in the previous section, direct maximization of the joint likelihood is computationally infeasible due to the combinatorial structure induced by the permutation matrices. To address this challenge, we adopt a Bayesian formulation of the problem. In order to obtain a computationally scalable solution, we employ a Variational Bayes approach to approximate the posterior distribution of the model parameters given the observed data.
We begin by revisiting the data-generating process (DGP) introduced in Section \ref{sec:dgp}. Throughout this section, we assume the same block structure for the permutation matrices, where the permutations are partitioned into $B$ blocks, each consisting of $K$ elements. 

For $\bS_i = (s_{i1},\ldots,s_{iK})$ denoting the latent domain indices for block $i$, let $\bY_i$ denote the corresponding outcome vector, $\bX_i$ the exposure vector, $\bW_i$ the residual latent-domain error vector, and $\beps_i$ a vector of iid errors with distribution $\calN(0,\tau^2)$. Let $\bpix$ and $\bpis$ denote the common block-wise permutation matrices that respectively reorder the exposure vector and the latent domain indices. Under this setup, the data-generating process (DGP) for block $i$ can be written as $\bY_i = \bpix \bX_i \beta + \bpis \bW_i + \beps_i$, for $i=1,\ldots,B$. 
We interpret $\bW_i$ as a latent process capturing dependence over the sampling domain. Let $\boldY_n \coloneq  \left(\bY_1^\top,\cdots, \bY_B^\top\right)^\top$ and $\boldX_n \coloneq  \left(\bX_1^\top,\cdots, \bX_B^\top\right)^\top$ denote the stacked outcome and exposure vectors across all blocks. Similarly, we define $\mathbb{W}_n \coloneqq \left(\bW_1^\top,\cdots, \bW_B^\top\right)^\top$ as the stacked latent error process and the iid error process vector as $\boldsymbol{\varepsilon}_n = \{\beps_1^\top,\cdots\beps_B^\top\}^\top$. We place a Gaussian process prior on $\mathbb{W}_n$ with mean function zero and stationary covariance kernel parameterized by $\bgam \coloneqq (\phi, \sigma^2)$. where $\phi$ controls the range of the correlation between two points in the domain and $\sigma^2$ controls the common variance for each component of the latent process. 

Under this specification, the conditional likelihood (given $\boldX_n$) for the parameter vector 
$\bth \coloneqq \left(\bpix, \bpis, \beta, \sigma^2, \tau^2, \phi, \mathbb{W}_n\right)^\top$ can be written as
{
\small
$$
L(\boldY_n\mid\bth, \boldX_n) = \left(2 \pi \tau^2\right)^{-\nicefrac{KB}{2}}
\prod_{i=1}^B 
\exp \left\{-\dfrac{1}{2 \tau^2}
\left(\bY_i-\bpix \bX_i \beta-\bpis \bW_i\right)^\top
\left(\bY_i-\bpix \bX_i \beta-\bpis \bW_i\right)\right\}.
$$}
Assuming independent priors across the components of $\bth$, the joint prior distribution can be factorized as
{
\small
$$
\boldsymbol{p}(\bth) \coloneqq p\left(\mathbb{W}_n \mid \sigma^2, \phi\right) \cdot p(\beta) \cdot p\left(\sigma^2\right)\cdot p(\phi)\cdot p\left(\tau^2\right) \cdot p(\bpix)\cdot p(\bpis).
$$}
Directly imposing a discrete prior on the permutation matrices $\bpix$ and $\bpis$ is computationally challenging due to the combinatorial nature of the permutation space. In particular, each permutation matrix can take one of $K!$ possible configurations, making posterior inference over this discrete space intractable even for moderately sized blocks. To obtain a computationally tractable formulation, we adopt a continuous relaxation of the permutation space.

This relaxation requires redefining the prior distribution over the permutation matrices. Since the variational approximation operates over continuous densities, it is natural to introduce a continuous prior over matrices that approximate permutation matrices. Ideally, such a prior should favor matrices that lie close to the set of valid permutations while still allowing efficient optimization in a continuous parameter space. Motivated by this consideration, and following the continuous relaxation ideas developed in \cite{maddison2016concrete} and later extended for permutation inference by \cite{linderman2018reparameterizing}, we place independent coordinate-wise Gaussian mixture priors on the entries of $\bpix$ and $\bpis$.

The mixture components are centered at $0$ and $1$, encouraging the matrix entries to concentrate near binary values while retaining differentiability and enabling gradient-based optimization. After sampling from this relaxed prior, the resulting matrices can be projected onto the Birkhoff polytope, the convex hull of all doubly stochastic matrices, and subsequently mapped to the nearest permutation matrix. This construction provides a principled continuous approximation to the discrete permutation space while remaining compatible with variational inference methods.

Based on this construction, we impose the following priors on the model parameters for suitable choices of hyperparameters. Let $\bSig^\ast_n$ denote the covariance matrix of the dependent process defined over the $n=KB$ domain points using the isotropic kernel with parameters $\bgam$ introduced earlier. The individual prior distributions are given by
\begin{equation}
    \begin{split}
        \mathbb{W}_n\mid\sigma^2, \phi &\sim \calN_n(\bzero, \bSig^\ast_n(\sigma^2,\phi))\\
        \beta & \sim \calN(0, \sigma^2_\beta)\\
        \sigma^2 & \sim \mathrm{Inv-Gamma}(a_1,b_1)\\
        \tau^2 & \sim \mathrm{Inv-Gamma}(a_2,b_2)\\
        \phi & \sim Unif(0, \sqrt{2})\\
        ((\bpix))_{mk} \coloneq x_{mk} & \sim \dfrac{1}{2}\calN(0,\eta^2_x) + \dfrac{1}{2}\calN(1,\eta^2_x) \\
        ((\bpis))_{mk} \coloneq s_{mk} & \sim \dfrac{1}{2}\calN(0,\eta^2_s) + \dfrac{1}{2}\calN(1,\eta^2_s). \\
    \end{split}
\end{equation}
Let $\bet$ denote the vector of hyperparameters used in the prior distribution of $\bth = (\mathbb{W}_n,\beta,\sigma^2,\phi,\tau^2,\bpix,\bpis)^\top$. We approximate the posterior distribution of the parameter vector $\bth$ using variational inference over a tractable family of distributions $q(\bth \mid \blam)$, where $\blam$ denotes the parameters governing the variational density. The approximate posterior is obtained by optimizing $\blam$ to minimize the Kullback–Leibler (KL) divergence between the variational density $q(\cdot \mid \blam)$ and the true posterior conditional on $\boldX_n$. Equivalently, this corresponds to maximizing the Evidence Lower BOund (ELBO):
\begin{equation}\label{eq:elbo-def}
    \mathscr{L}(\blam) \coloneq \E_{q(\cdot|\blam)}\left[\log L(\boldY_n,\bth\mid\boldX_n) - \log q(\bth\mid\blam,\boldX_n)\right].
\end{equation}

We assume a mean-field variational approximation for the target posterior distribution of $\bth$. Under this factorization, the results of \cite{ren2011variational} allow closed-form expressions for most of the variational densities corresponding to the model parameters. For parameters whose variational densities do not admit closed-form solutions, the corresponding variational densities are estimated by directly maximizing the ELBO.

Based on the prior structure and the assumed likelihood model, Gaussian variational families arise for the dependent random effect vector $\mathbb{W}_n$, parameterized by $(\bmu_W,\bSig_W)$, and for $\beta$, parameterized by $(\mu_{\lambda,\beta},\sigma^2_{\lambda,\beta})$. The variance parameters $\sigma^2$ and $\tau^2$ follow inverse-gamma variational families parameterized by $(\lambda_{a_1},\lambda_{b_1})$ and $(\lambda_{a_2},\lambda_{b_2})$, respectively. The scale parameter $\phi$ is associated with a more complicated variational density whose form, along with closed form solution of the parameters defining the above mentioned variational distributions are derived in Appendix \ref{appendix-vd-parameters}.

For the permutation matrices $\bpix$ and $\bpis$, we adopt the continuous relaxation variational family proposed by \cite{linderman2018reparameterizing}. Instead of placing a distribution directly on the discrete set of permutation matrices, this approach introduces a smooth transformation that generates matrices near the Birkhoff polytope, the space of all doubly stochastic matrices, and subsequently rounds them toward valid permutations.
Specifically, a random matrix $\bZ \in \R^{n \times n}$ is first generated with independent standard normal entries. A mean parameter matrix $\bM$ is then mapped to the Birkhoff polytope using the Sinkhorn--Knopp algorithm \citep{sinkhorn1967concerning}, which iteratively rescales the rows and columns of a matrix until convergence to a doubly stochastic matrix, producing $\widetilde{\bM}$. Next, a perturbed matrix is constructed as $\mathbf{\Psi} = \widetilde{\bM} + \bV \odot \bZ$, where $\bV$ controls the elementwise scale of the perturbation and $\odot$ denotes elementwise multiplication. The nearest permutation matrix $\text{round}(\mathbf{\Psi})$ is obtained using the Hungarian algorithm \citep{kuhn1955hungarian}. Finally, a temperature-controlled interpolation
$\bpi^\ast = \tau \mathbf{\Psi} + (1-\tau)\,\text{round}(\mathbf{\Psi})$
produces a relaxed permutation matrix. As $\tau \to 0$, the distribution concentrates on the vertices of the Birkhoff polytope, thereby recovering discrete permutation matrices.

For the parameters $\zeta \coloneq (\bM,\bV)$, this transformation induces a variational density $q_{\tau}(\bpi^\ast \mid \zeta)$ over matrices near the Birkhoff polytope while remaining amenable to gradient-based optimization. Letting $\bZ = g_{\tau}^{-1}(\bpi^\ast;\zeta)$ denote the inverse transformation, the corresponding density may be expressed as
\begin{equation}\label{eq:perm-vd}
    q_{\tau}(\bpi^\ast\mid\zeta)
=
\prod_{m=1}^n \prod_{n'=1}^n
\frac{1}{\tau v_{mn'}}
\,
\mathcal{N}\!\left(z_{mn'};0,1\right)
\,\mathbb{I}\{\bpi^\ast \in \mathcal{G}_{\tau}\},
\end{equation}
where $z_{mn'} = [g_{\tau}^{-1}(\bpi^\ast;\zeta)]_{mn'}$, $v_{mn'}$ denotes the $(m,n')$-th entry of $\bV$, and $\mathcal{G}_{\tau}$ denotes the image of the transformation $g_\tau(\cdot\mid \zeta)$. For the two permutation matrices $\bpix$ and $\bpis$, we assign the above variational density as $q_{\tau_X}(\cdot\mid \zeta_X)$ and $q_{\tau_S}(\cdot\mid \zeta_S)$, respectively, where $\zeta_Q = (\bM_Q,\bV_Q)$ for $Q\in \{X,S\}$.

Let $\blam \coloneq (\mu_{\lambda,\beta},\sigma^2_{\lambda,\beta},\bmu_W,\bSig_W, \lambda_{a_1},\lambda_{b_1},\lambda_{a_2},\lambda_{b_2},\,\zeta_X,\zeta_S)^\top$ 
denote the vector parameterizing the joint variational density of $\bth$. Under the mean-field assumption, the joint variational density decomposes as
\begin{equation}\label{eq:variational-decomposition}
\begin{split}
    q(\bth \mid\blam) \coloneq &\, q_W\left(\mathbb{W}_n \mid \bmu_W, \bSig_W\right) \, \cdot\,  q_\beta(\beta\mid\mu_{\lambda,\beta},\sigma^2_{\lambda,\beta}) \, \cdot\,  q_{\sigma^2}\left(\sigma^2\mid\lambda_{a_1},\lambda_{b_1}\right)\,\cdot\, q_{\tau^2}\left(\tau^2\mid\lambda_{a_2},\lambda_{b_2}\right) \\
    &   \,\cdot\, q_\phi(\phi)\,\cdot\, q_{\tau_X}(\bpix\mid\zeta_X)\,\cdot\, q_{\tau_S}(\bpis\mid\zeta_S).
\end{split}
\end{equation}

Except for the parameters associated with the permutation matrices, closed-form solutions are available for all other variational parameters (see Appendix \ref{appendix-vd-parameters}). The variational parameters corresponding to the permutation matrices do not admit closed-form updates and are therefore estimated through numerical optimization that maximizes the ELBO with respect to these parameters.

In addition to the parameters $\blam$, several auxiliary quantities arise in computing the closed-form updates and evaluating the ELBO for the permutation matrices. One such quantity is $\E_{q_\phi}\left[\bR(\phi)^{-1}\right]$, where the expectation is taken with respect to the variational density of $\phi$. Another set of quantities are 
$\blam_{\bpix} = (\bM_X^\ast,\bV^\ast_X)$ and $\blam_{\bpis} = (\bM_S^\ast,\bV^\ast_S)$, where $\bM_X^\ast \coloneq \E_{q_{\tau_X}}[\bpix]$, $\bV_X^\ast \coloneq \E_{q_{\tau_X}}[\bpix^\top \bpix]$, and similarly for $\bpis$. These expectations appear in multiple steps of the parameter updates and must therefore be estimated as part of the optimization procedure, and thus we include them within the parameter vector $\blam$ defined earlier.

The resulting maximization of the ELBO leads to a coupled system of equations, since the update of each parameter depends on the current values of the others. This motivates an iterative estimation procedure in which each parameter in $\blam$, together with the auxiliary quantities described above, is updated using the most recent values of the remaining parameters.

After each iteration, the global ELBO (whose expression is provided in Appendix \ref{appendix-global-elbo}) is evaluated, and the procedure is terminated once the incremental change in the ELBO falls below a user-specified threshold. The temperature parameters $\tau_X$ and $\tau_S$ are initialized at user-specified starting values and subsequently annealed according to an exponential decay schedule with lower bound $\tau_{\min} = 0.05$ and decay rate $\alpha = 0.995$. The decay is governed by a shared global step counter that increases after every inner optimization step associated with either permutation model, ensuring that the temperature decreases gradually throughout the optimization procedure rather than resetting within each outer iteration. Algorithm \ref{alg-parameter-est} summarizes the resulting estimation procedure.

\begin{algorithm}[htbp]
\caption{VB Algorithm to estimate the parameters of the variational distribution of $\bth$}
\label{alg-parameter-est}
\begin{algorithmic}
\small
\Require the hyperparameters $\bet$, learning rates $(l_X,l_S)$, threshold $\epsilon$ for the ELBO, initial values of the variational density parameters $\blam^{(0)}$. 
\Ensure Variational parameter estimates:
\newline
{
\footnotesize
$\blam^{(T)} = \left\{\mu^{(T)}_{\lambda,\beta}, \sigma^{2^{(T)}}_{\lambda,\beta},\bmu_W^{(T)} ,\bSig^{(T)}_W,\lambda^{(T)}_{a_1}, \lambda^{(T)}_{b_1}, \lambda^{(T)}_{a_2}, \lambda^{(T)}_{b_2}, \E^{(T)}[\bR(\phi)^{-1}],\zeta_X^{(T)},\zeta_S^{(T)},\bM_X^{\ast^{(T)}}, \bV_X^{\ast^{(T)}}, \bM_S^{\ast^{(T)}}, \bV_S^{\ast^{(T)}} \right\}$ where $T \coloneq T(\epsilon)$.} 
\While{$\text{ELBO}^{(t+1)} - \text{ELBO}^{(t)}>\epsilon$}
    \State \textbf{Step 1:} Update the distribution of $\beta \sim \calN\left(\mu^{(t)}_{\lambda,\beta},\sigma^{2^{(t)}}_{\lambda,\beta}\right)$ where\\
    \hspace{5mm} 
    $
    \sigma^{2^{(t)}}_{\lambda, \beta} = \left( \frac{\lambda^{(t-1)}_{a_2}}{\lambda^{(t-1)}_{b_2}} \sum_{i=1}^B\bX_i^\top \bV^{\ast^{(t-1)}}_X \bX_i + \frac{1}{\sigma_\beta^{2^{(t-1)}}} \right)^{-1}
    $ and \\
    \hspace{5mm}
    $
    \mu^{(t)}_{\lambda, \beta} = \sigma^{2^{(t)}}_{\lambda, \beta} \left( \frac{\lambda^{(t-1)}_{a_2}}{\lambda^{(t-1)}_{b_2}}\sum_{i=1}^B \bX_i^\top {\bM^{\ast^{(t-1)}}_X}^\top \left(\bY_i - \bM^{\ast^{(t-1)}}_S \bmu^{(t-1)}_{W_i}\right) \right).
    $
    \State \textbf{Step 2:} Update the distribution of $\mathbb{W}_n \sim \calN_n\left(\bmu_W^{(t)},\bSig^{(t)}_W\right)$ where\\ 
    \hspace{5mm}
    $\bSig^{(t)}_{W} = \left({\dfrac{\lambda^{(t-1)}_{a_2}}{\lambda^{(t-1)}_{b_2}}} \mathrm{diag}\left(\bV^{\ast^{(t-1)}}_S\right) + \dfrac{\lambda^{(t-1)}_{a_1}}{\lambda^{(t-1)}_{b_1}} \E^{(t-1)}[\bR(\phi)]^{-1} \right)^{-1}$ and \\
\hspace{5mm}
$\bmu^{(t)}_{W} = \bSig^{(t)}_{W} \left( {{\dfrac{\lambda^{(t-1)}_{a_2}}{\lambda^{(t-1)}_{b_2}}} \mathrm{diag}\left(\bM^{\ast^{(t-1)}}_S\right)^\top}\left(\mathbb{Y}_n - \mu^{(t)}_{\lambda,\beta}\mathrm{diag}\left(\bM^{\ast^{(t-1)}}_X \right)\mathbb{X}_B\right) \right)$
    \State \textbf{Step 3:} Update the distribution of $\sigma^2 \sim Inv-Gamma\left(\lambda^{(t)}_{a_1},\lambda^{(t)}_{b_1}\right)$
    where\\
    \hspace{5mm}
    $
    \lambda^{(t)}_{a_1} = \frac{KB}{2} + a_1$ and \\
    \hspace{5mm}
    $\lambda^{(t)}_{b_1} = \frac{1}{2} \left[ \tr\left( \E^{(t-1)}[\bR(\phi)^{-1}] \bSig_W^{(t)} \right) + \bmu_W^{(t)^\top} \E^{(t-1)}[\bR(\phi)^{-1}] \bmu_W \right] + b_1$
    \State \textbf{Step 4:} Update the distribution of $\tau^2 \sim Inv-Gamma\left(\lambda^{(t)}_{a_2},\lambda^{(t)}_{b_2}\right)$
    where\\
    \hspace{5mm}
    $
    \lambda^{(t)}_{a_2} = \frac{KB}{2} + a_2$ and \\
    \hspace{5mm} 
    $\lambda^{(t)}_{b_2} = \sum_{i = 1}^B \left(\bY_i - \mu^{(t)}_{\lambda, \beta} M^{\ast^{(t-1)}}_X \bX_i - {M^{\ast^{(t-1)}}_S}\bmu^{(t)}_{W_i}\right)^\top \left(\bY_i - \mu^{(t)}_{\lambda, \beta} M^{\ast^{(t-1)}}_X \bX_i - {M^{\ast^{(t-1)}}_S}\bmu^{(t)}_{W_i} \right)$\\
    \hspace{16mm}$
    + \mu^{2^{(t)}}_{\lambda, \beta} \sum_{i = 1}^B \bX_i^\top \left(\bV^{*^{(t-1)}}_X - {\bM^{\ast^{(t-1)}}_X}^\top \bM^{\ast^{(t-1)}}_X\right) \bX_i + \sigma^{2^{(t)}}_{\lambda, \beta} \sum_{i = 1}^B \bX_i^\top \bV^{\ast^{(t-1)}}_X \bX_i$\\
    \hspace{15mm}
    $ + \sum_{i = 1}^B \left[ \bmu^{(t)^\top}_{W_i} \left(\bV^{\ast^{(t-1)}}_S - {\bM^{\ast^{(t-1)}}_S}^\top \bM^{\ast^{(t-1)}}_S\right)\bmu^{(t)}_{W_i} + \tr\left(\bV^{\ast^{(t-1)}}_S \bSig^{(t)}_{W_i}\right) \right] + b_2$
    \State \textbf{Step 5:} Update the distribution on $\phi$ which is proportional to \\
    \hspace{5mm}
    $\left|\bR(\phi)\right|^{-\frac{1}{2}}\exp\left(-\frac{\lambda^{(t)}_{a_1}}{2\lambda^{(t)}_{b_1}} \left[\tr\left(\bR(\phi)^{-1} \bSig^{(t)}_W\right) + \bmu_W^{{(t)}^\top} \bR(\phi)^{-1} \bmu^{(t)}_W\right]\right)$ and compute\\
    \hspace{5mm} 
    $\E_{q_\phi}^{(t)}[\bR(\phi)^{-1}]$ using importance sampling.
    \State \textbf{Step 6:} Using the last 5 steps and the learning rates $l_X$ and $l_S$, minimize the ELBO\\
    \hspace{5mm} for $\bpix$ and $\bpis$ to get the parameters $\zeta_X^{(t)},\zeta_S^{(t)},\bM_X^{\ast^{(t)}}, \bV_X^{\ast^{(t)}}, \bM_S^{\ast^{(t)}}, \bV_S^{\ast^{(t)}}$.
\EndWhile
\end{algorithmic}
\end{algorithm}
Algorithm \ref{alg:perm-est} describes the procedure used to recover the permutation matrices from the estimated variational parameters. The matrices $\widehat{\bM}^\ast_X$ and $\widehat{\bM}^\ast_S$, which correspond to the posterior expectations of $\bpix$ and $\bpis$, may not themselves be valid permutation matrices. Therefore, we first project these matrices onto the Birkhoff polytope $\calB_K$. This projection is performed using the Sinkhorn–Knopp algorithm.
Once a doubly stochastic approximation is obtained, we recover the closest permutation matrix by solving a linear assignment problem. This step is carried out using the Hungarian algorithm. The same procedure is applied to both $\widehat{\bM}^\ast_X$ and $\widehat{\bM}^\ast_S$ to obtain the final estimates $\hatpi_X$ and $\hatpi_S$.
\begin{algorithm}[H]
\small
\caption{Algorithm for estimating the permutation matrices}
\label{alg:perm-est}
    \begin{algorithmic}
        \Require The estimated parameters $\widehat{\bM}^\ast_X = \widehat{\E}_{q_{\tau_X}}[\bpix]$ and $\widehat{\bM}^\ast_S = \widehat{\E}_{q_{\tau_S}}[\bpis]$
        \Ensure $\hatpi_X,\hatpi_S$
        \State \textbf{Step 1:} Project $\widehat{\bM}^\ast_X$ onto Birkhoff Polytope $\calB_n$ using Sinkhorn-Knopp algorithm to get $\widehat{\calB}_X$.
        \State \textbf{Step 2:} Estimate the nearest permutation matrix $\hatpi_X$ from $\widehat{\calB}_X$ using the Hungarian Algorithm. 
        \State \textbf{Step 3:} Repeat steps 1 and 2 similarly as above to get $\hatpi_S$.
    \end{algorithmic}
\end{algorithm}
The following section presents a detailed simulation study designed to evaluate the finite-sample performance of the proposed REPAIR method.
\section{Simulation Studies}\label{sec:simulation}

\textbf{Simulation Regimes.} We conducted an extensive simulation study to assess the finite-sample performance of the proposed methodology REPAIR under varying levels of domain complexity and permutation misalignment. The simulation design varies two key factors that govern the difficulty of the estimation problem: (i) the granularity of the sampling latent domain and (ii) SNR of the exposure effect. By systematically varying these quantities while holding all other components of the data-generating mechanism fixed, the experiments allow us to evaluate both the statistical accuracy and robustness of the proposed variational Bayes based REPAIR estimator across a range of realistic scenarios. In addition, the design facilitates comparisons with competing methods under identical data-generating conditions. For each experiment, we considered several values of the $(K,B)$ pairs where the number of block regions $B \in \{49, 81, 100, 121\}$, with each region containing $k \in \{6,8,10,12,20\}$ latent locations. This yields total sample sizes ranging from $KB = 294$ to $2420$. The domain values were generated by partitioning the domain into a $\sqrt{B} \times \sqrt{B}$ grid and then randomly sampling $K$ locations uniformly within each region. The 1-D exposure vector $\mathbb{X}_n$ was independently sampled from a standard normal distribution, i.e., $X_j \sim \calN(0,1)$ for each $j = 1,\ldots,n$.

The true domain covariance structure $\bSig^\ast_n = \sigma^2 \bR(\phi)$ followed an exponential covariance kernel with variance parameter $\sigma^2 = 5$, range parameter $\phi = 0.5$. The iid noise variance was set at $\tau^2 = 0.5$. The spatial random effect $\mathbb{W}_n$ was then simulated from $\calN_n(\bzero, \bSig^\ast_n)$, and the iid noise vector $\boldsymbol{\varepsilon}_n$ was generated from $\calN(0,\tau^2)$.

For every $(K,B)$ pair, the Hamming distances $d_H(\bpix,\bI_K)$ and $d_H(\bpis,\bI_K)$ were randomly selected from $\{1,\cdots,K\}$ and kept fixed across replicates. These distances control the amount of misalignment induced by the permutation matrices in the data-generating process. To examine the effect of signal strength, we considered two values of $\beta \in \{2,8\}$ corresponding to different SNR regimes. Finally, the outcome vector $\mathbb{Y}_n$ was generated according to the structural equation $\mathbb{Y}_n = \bPix \mathbb{X}_n \beta + \bPis \mathbb{W}_n + \boldsymbol{\varepsilon}_n$ where $\bPi_X = \text{bdiag}(\bpix)$ and $\bPi_S = \text{bdiag}(\bpis)$. For each configuration of parameters, we generated 100 independent replicates. The objective of this simulation study is to examine how the proposed estimation procedure performs across varying values of $(K, B)$ and SNR levels.

\textbf{Methods Considered.} For model fitting, we assume that the covariance kernel is correctly specified; in particular, we fit the data using the exponential kernel, and estimate the parameters of it accordingly. We compare three estimation strategies on the simulated datasets:  
1) \textbf{FullGP}, an oracle method in which the true permutation matrices are assumed known and a standard spatial Gaussian process regression with an exponential covariance kernel is fitted. This benchmark provides a gold-standard reference for evaluating the performance of our approach.  
2) \textbf{ArealGP}, where both the exposure and the outcome are aggregated at the block level, and a GLS regression is performed on the aggregated values to estimate the parameters,  
3) Our proposed \textbf{REPAIR} method, implemented as described in Algorithms ~\ref{alg-parameter-est} and ~\ref{alg:perm-est}. For the FullGP and ArealGP methods, Likelihood based optimization was used and the MLEs have been computed.

\textbf{Simulation Results.} 
Here, we are mostly interested in how the methods have successfully recovered the effect size parameter $\beta$ and the unknown permutation matrices $\bpix,\bpis$ under several choices of $(K,B)$ pairs and the different SNR values. Although we have reported the estimates of the variance and covariance parameters in the supplementary materials. 
\begin{figure}[ht]
    \centering
    \includegraphics[width=\textwidth]{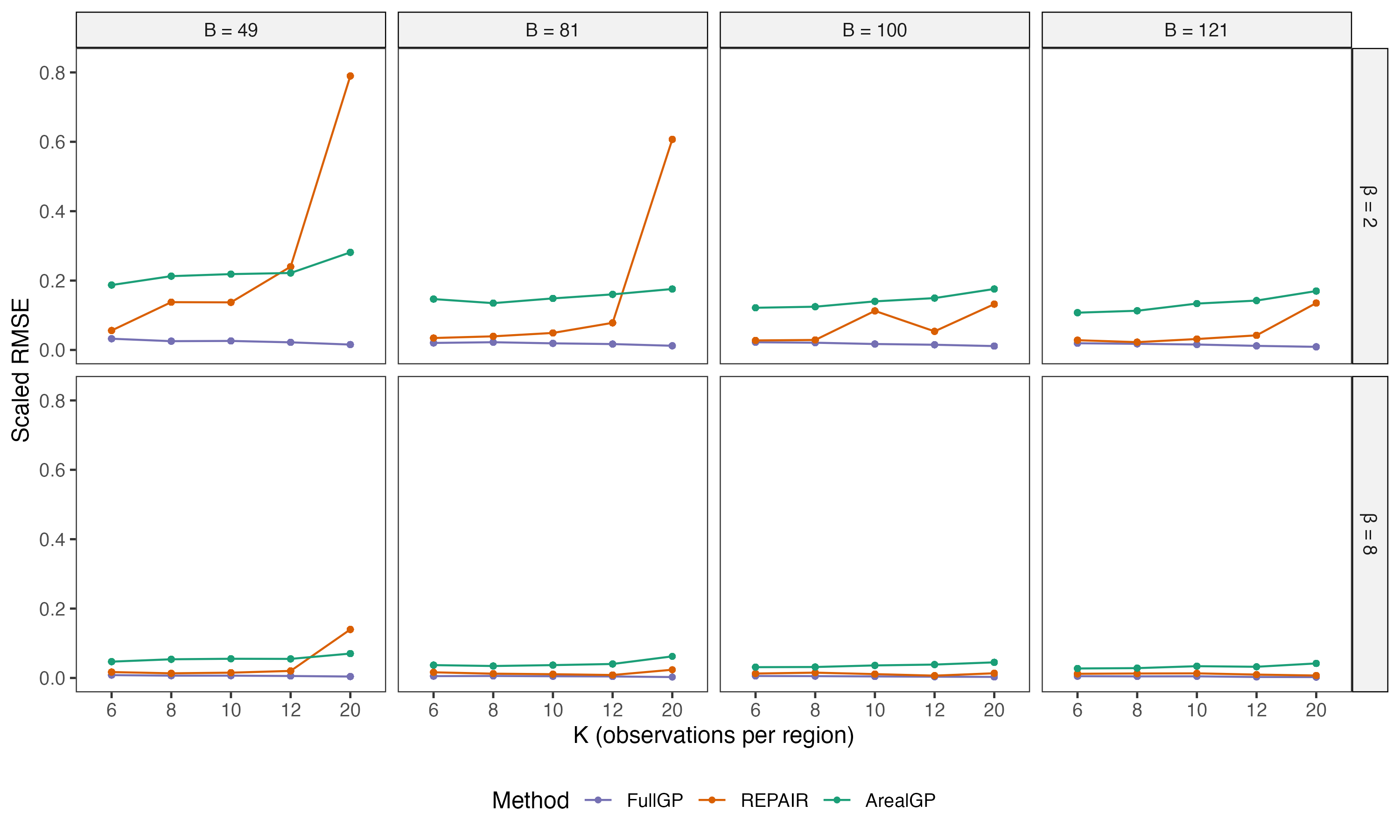}
    \caption{Scaled RMSE of $\hat{\beta}$ across simulation settings. Columns correspond to the number of regions $B$ and the horizontal axis shows the number of observations per region $K$. Results are reported for two SNR regimes, $\beta=2$ and $\beta=8$, comparing FullGP, ArealGP, and REPAIR.}
    \label{fig:beta_rmse_lines}
\end{figure}

Figure \ref{fig:beta_rmse_lines} summarizes the estimation accuracy for the regression coefficient $\beta$ across the different simulation configurations. Compared with the oracle FullGP method, the proposed REPAIR approach consistently performs better than the competing ArealGP method in terms of estimation accuracy, with its performance improving as the number of blocks $B$ increases. 
For the REPAIR method, however, we observe that the RMSE tends to increase as the block size $K$ becomes larger. This behavior reflects the increasing difficulty of recovering $\beta$ when the permutation structure becomes more complex for larger values of $K$. This phenomenon is consistent with the theoretical insight established in Theorem \ref{thm:error-betahat}, which shows that the probability of obtaining small estimation error for $\beta$ can deteriorate when $K$ grows relative to $B$. 
Nevertheless, when the number of blocks $B$ is sufficiently large, REPAIR is able to recover the effect size accurately even for larger values of $K$, and continues to outperform the ArealGP method in most settings. Furthermore, the scaled RMSE values clearly indicate that the SNR plays an important role in determining estimation accuracy. For high SNR ($\beta = 8$), the performance of REPAIR is nearly indistinguishable from the oracle FullGP method. In contrast, when the signal is weaker ($\beta = 2$), the estimation problem becomes more challenging and the gap between REPAIR and the oracle method becomes more pronounced. These empirical observations align closely with the theoretical results developed earlier. One area where REPAIR appears somewhat less accurate is in estimation of the variance parameters, for which it has higher RMSE than the competing methods; see Supplement \ref{suppsec:variance-estimates}. This likely reflects the limitations of the mean-field variational approximation, which seems to be more effective for point estimation of $\beta$ than for recovering nuisance variance components.

\begin{figure}[ht]
\centering
\includegraphics[width=\textwidth]{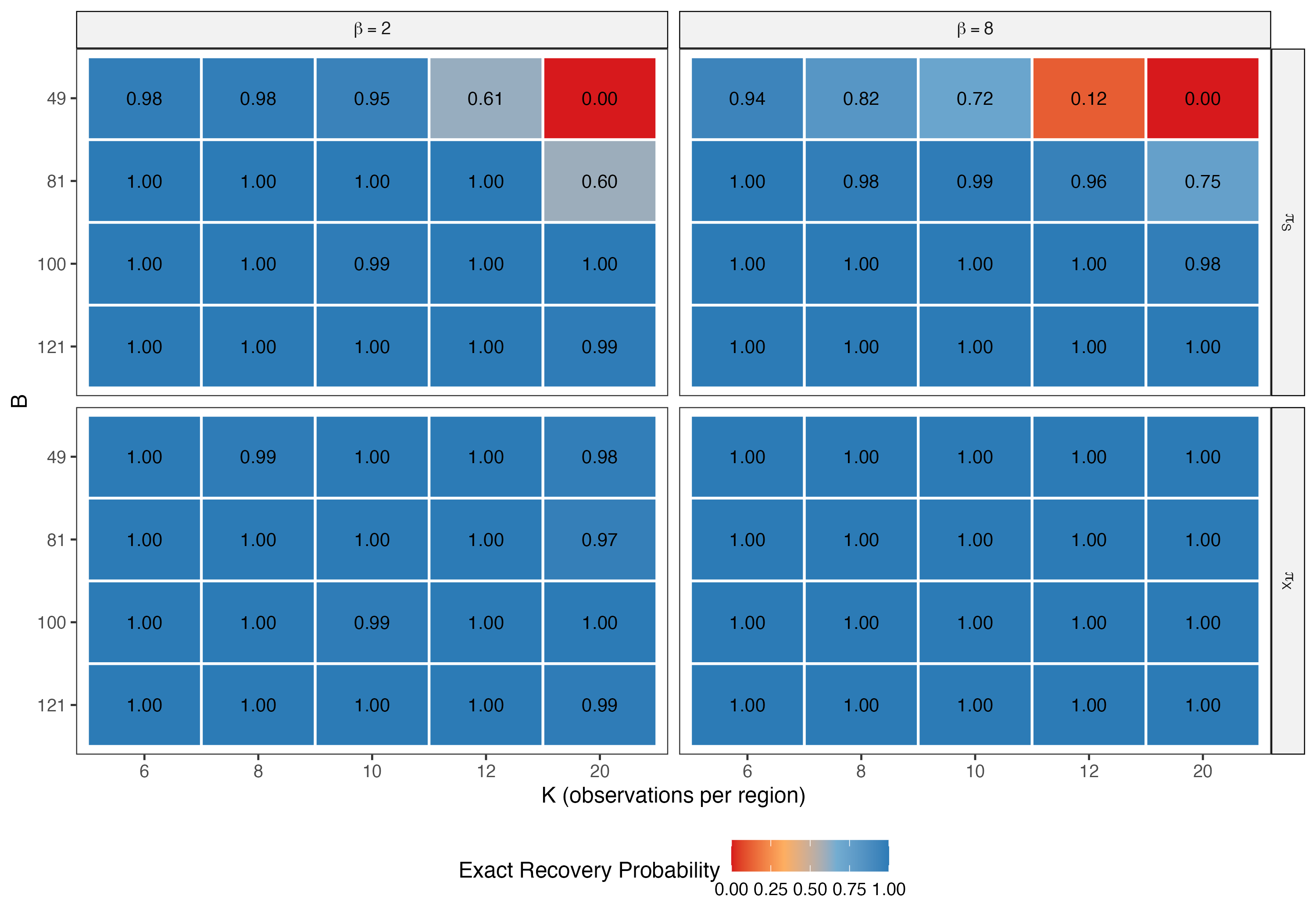}
\caption{Permutation recovery performance of the REPAIR method across simulation settings. 
Cells show the estimated recovery probability for the permutation matrices $\pi_S$ and $\pi_X$ across different values of the number of blocks $B$ and observations per block $K$, under two signal regimes ($\beta = 2$ and $\beta = 8$).}
\label{fig:perm_recovery}
\end{figure}
Figure \ref{fig:perm_recovery} summarizes the permutation recovery performance of the REPAIR method for the two latent permutation matrices $\bpix$ and $\bpis$ across the different simulation settings. Overall, the recovery probabilities are broadly similar across the two signal regimes and across most combinations of $(K,B)$. While the theoretical results characterize recovery probabilities averaged over all possible values of the Hamming distance between the true and identity permutations, it is computationally infeasible in practice to enumerate all such configurations in a simulation study. Consequently, the empirical results shown here should be interpreted as representative realizations of the underlying theoretical behavior rather than a direct numerical validation of the theoretical bounds.

An interesting pattern that emerges from Figure \ref{fig:perm_recovery} is that the exposure permutation matrix $\bpix$ is typically recovered with higher probability than the latent-domain permutation matrix $\bpis$. In several regimes, particularly when the block size $K$ is relatively large, the recovery probability for $\bpis$ decreases noticeably even when $\bpix$ continues to be identified reliably. Despite this partial recovery of the latent permutation structure, the regression coefficient $\beta$ is still estimated accurately, as shown in Figure \ref{fig:beta_rmse_lines}. 

This observation highlights an important strength of the proposed REPAIR method: accurate estimation of the effect size does not require perfect recovery of both permutation matrices. In particular, reliable identification of the exposure permutation $\bpix$ appears sufficient for stable estimation of $\beta$, even when the latent-domain permutation $\bpis$ is only partially recovered. From an applied perspective, this property can be interpreted as providing a degree of privacy preservation, since the underlying permutation structure may remain partially obscured while still allowing accurate inference on the effect size of interest.

\section{Real Data Analysis}\label{sec:data}

We illustrate the proposed methodology using the \textit{Meuse} soil dataset from the \textrm{sp} package in \textbf{R}; see \citet{pebesma2012package}. The dataset records locations of topsoil heavy metal concentrations, together with soil and landscape variables, collected on the floodplain of the river Meuse near the village of Stein in the Netherlands. The measured heavy metals include cadmium (Cd), copper (Cu), lead (Pb), and zinc (Zn). A distinguishing feature of this dataset is the presence of the river Meuse itself as a naturally occurring geographic boundary, which induces substantial spatial heterogeneity across the floodplain. From an environmental standpoint, this is of direct interest because the floodplain soils are used for agriculture, so elevated heavy metal concentrations may affect crops consumed by humans and livestock. In our analysis, we focus on \textit{zinc concentration} as the outcome and \textit{elevation} as the exposure, with the goal of studying how the latent spatial structure of the floodplain is reflected in the exposure--outcome relationship.

We analyze the transformed response $Y=\log(1+\text{zinc})$, center both the exposure and the outcome, and rescale the spatial coordinates to the unit square for numerical stability. To construct a doubly-unlinked version of the data, we randomly remove 5 observations from the original 155 locations, leaving 150 sites. The retained locations are then ordered by the first spatial coordinate and partitioned into contiguous blocks. We consider two blocking schemes: (a) a coarser partition with $15$ blocks of size $10$, and (b) a finer partition with $30$ blocks of size $5$. Within each block, the exposure values and spatial coordinates are permuted, thereby inducing the doubly-unlinked structure. Figure~\ref{fig:meuse_partition} displays the resulting partitions.

\begin{figure}[t]
    \centering
    \includegraphics[width=\textwidth]{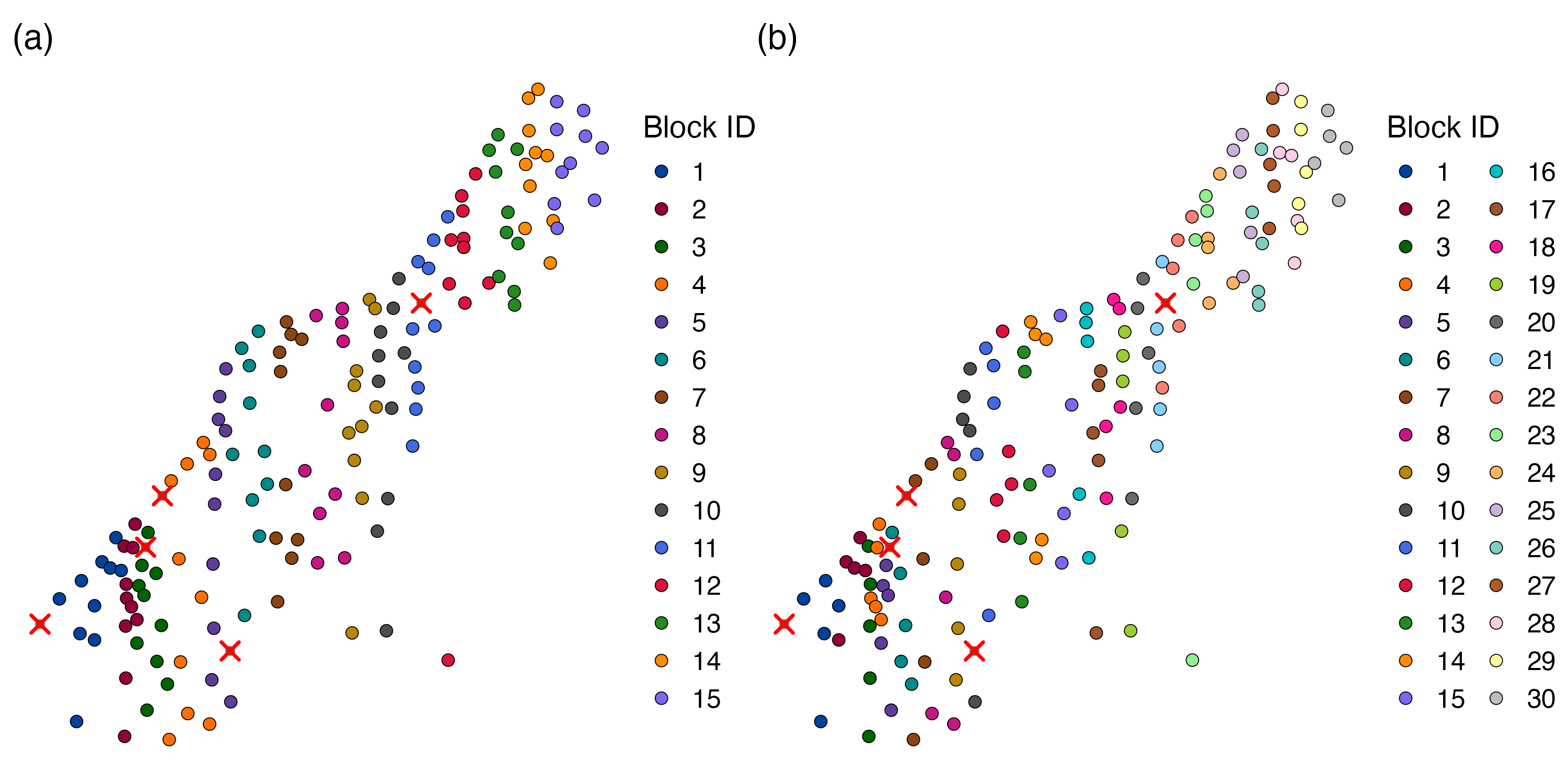}
    \caption{Partitions for the Meuse analysis after dropping five observations: (a) $15\times 10$ and (b) $30\times 5$. Colored points denote block memberships, and red crosses indicate the removed observations.}
    \label{fig:meuse_partition}
\end{figure}

We compare {ArealGP}, {REPAIR}, and the fully linked benchmark {FullGP} in effect size estimation. Across both blocking schemes, all three methods recover a negative regression effect of elevation on zinc concentration. Under the $15\times 10$ partition, the estimated regression coefficients are $-0.2926$ for FullGP, $-0.2317$ for ArealGP, and $-0.2102$ for REPAIR. Under the finer $30\times 5$ partition, the corresponding estimates are $-0.2926$, $-0.3965$, and $-0.2831$. Thus, the direction of the estimated exposure effect is stable across methods and across blocking schemes: higher elevation is associated with lower zinc concentration.

A more informative comparison is given by the proximity of each unlinked-data estimator to the benchmark. Under the finer $30\times 5$ partition, the {REPAIR} estimate $-0.2831$ is quite close to the FullGP estimate $-0.2926$, whereas under the coarser $15\times 10$ partition the estimate $-0.2102$ is farther away. By contrast, the {ArealGP} benchmark is more sensitive to the imposed partition, moving from $-0.2317$ under the coarser scheme to $-0.3965$ under the finer one.

\begin{figure}[ht]
    \centering
    \includegraphics[width=\textwidth]{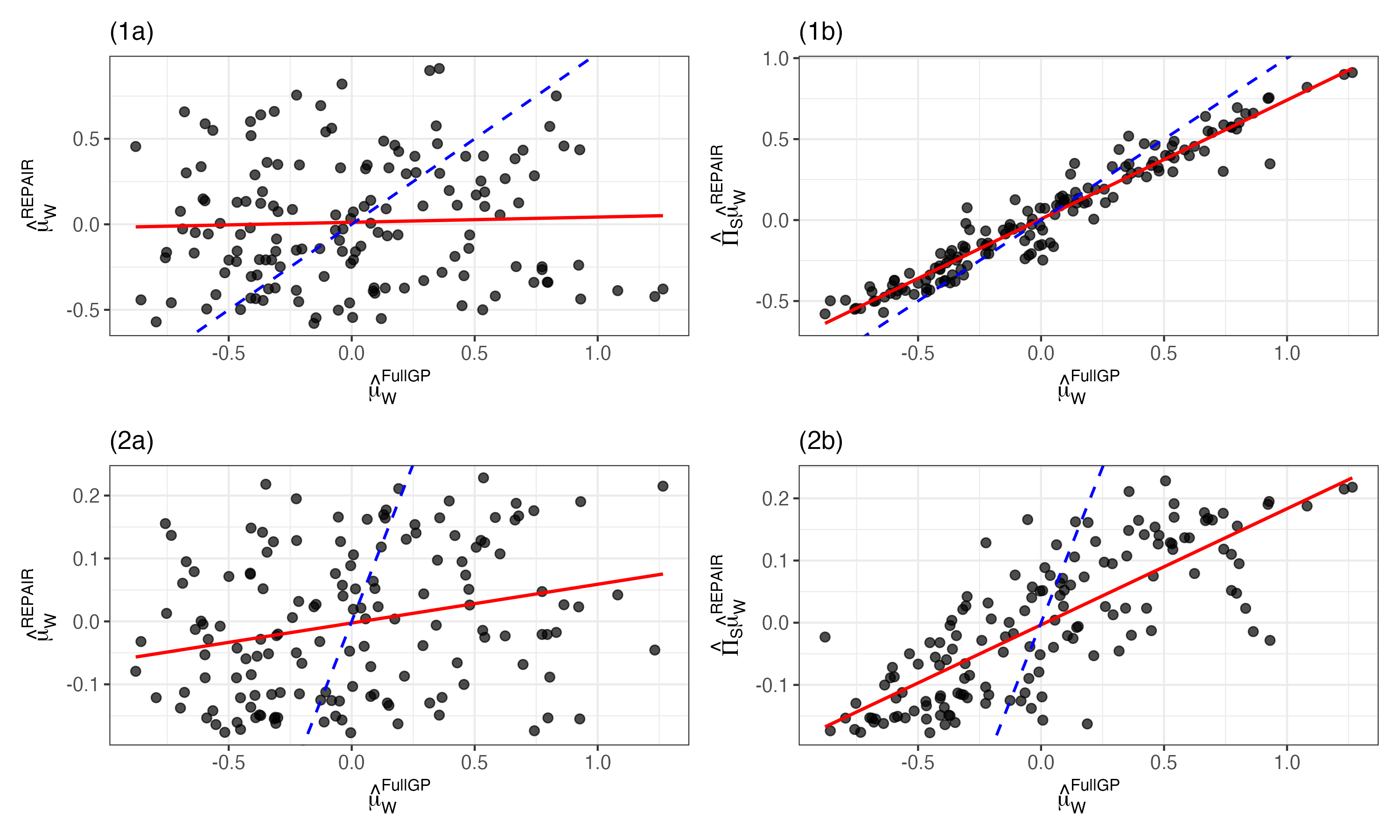}
    \caption{Comparison of latent surface $W$ estimates from REPAIR and the FullGP benchmark under the two blocking schemes. Panels (1a) and (2a) plot $\hat{\bmu}_W^{\mathrm{REPAIR}}$ against $\hat{\bmu}_W^{\mathrm{FullGP}}$ for the $30\times 5$ and $15\times 10$ partitions, respectively. Panels (1b) and (2b) plot $\widehat{\bPi}_S \hat{\bmu}_W^{\mathrm{REPAIR}}$ against $\hat{\bmu}_W^{\mathrm{FullGP}}$. The dashed blue line is the $45^\circ$ reference line, and the solid red line is the least-squares fit.}
    \label{fig:latent_surface_compare}
\end{figure}

In the $30\times 5$ setup, REPAIR recovers the within-block permutations essentially perfectly for both the exposure and spatial coordinates, whereas in the $15\times 10$ setup the recovery is only partial by the end of the optimization. This is consistent with the regression results: the finer partition appears to make the latent alignment problem more tractable, allowing REPAIR to recover an exposure effect much closer to the FullGP benchmark. Figure~\ref{fig:latent_surface_compare} reinforces this point at the level of the latent spatial surface. In the $30\times 5$ case, the permutation-corrected REPAIR estimate, $\widehat{\bPi}_S \hat{\bmu}_W^{\mathrm{REPAIR}}$, aligns very closely with $\hat{\bmu}_W^{\mathrm{FullGP}}$, with the points lying near the $45^\circ$ line. Here $\hat{\bmu}_W$ denotes the mean of the variational distribution for $\mathbb{W}_n$. This indicates that once the estimated spatial permutation is accounted for, REPAIR is able to recover the latent spatial field with high accuracy in the finer blocking regime. By contrast, in the $15\times 10$ case the agreement is visibly weaker: the scatter is more diffuse and departs more substantially from the identity line. This deterioration mirrors the weaker recovery of the regression effect under the coarser partition and suggests that the latent alignment problem becomes substantially harder as the within-block size grows.

Taken together, these results highlight the practical importance of keeping the block size sufficiently small, or equivalently the number of blocks $B$ sufficiently large relative to the within-block permutation complexity. When this condition is favorable, as in the $30\times 5$ setting, REPAIR can recover both the regression signal and the latent spatial surface quite accurately. When the blocks are too coarse, as in the $15\times 10$ setting, permutation recovery degrades, and this loss propagates to both surface recovery and effect estimation.

Overall, the Meuse analysis illustrates two key points. First, despite the presence of latent within-block shuffling, the exposure effect remains recoverable and can be estimated close to the oracle benchmark when the number of blocks is sufficiently large relative to the block size. Second, this setting naturally reflects a privacy-preserving regime, since the original alignment of exposure and spatial information is masked, yet scientifically meaningful inference is still possible.
\section{Discussion}\label{sec:discussion}

This paper introduces and studies the doubly-unlinked regression regime in dependent data, in which both the covariate–response correspondence and the response–domain correspondence are simultaneously unknown. We establish theoretical conditions governing the feasibility of permutation and regression recovery, show that consistent estimation of the effect size $\beta$ is achievable under strictly weaker signal-to-noise conditions than those required for exact permutation recovery, and develop REPAIR, a computationally tractable variational Bayes algorithm for joint inference on the regression parameter, the latent permutations, and the covariance structure of the latent process.  These developments lead to foundational contributions in theory and methods of unlinked regression to dependent-data settings. Throughout, we have modeled the latent domain as a spatial or temporal continuum, but the same inferential challenge arises when observations are indexed by nodes of a network: the latent domain is then a graph, dependence is encoded by graph structure rather than a covariance kernel, and the two broken links correspond precisely to the node-alignment problems studied in the graph matching literature \citep{arroyo2021maximum, lyzinski2014seeded, fishkind2019seeded, dawn2025covariate}. Extending the theoretical and algorithmic framework developed here to graph-indexed latent domains is therefore a natural and principled next step.

Our applied analysis on the Meuse soil dataset was intended as a proof of concept rather than a substantive scientific claim. The dataset is well-understood, fully observed, and publicly available, making it well-suited for evaluating whether REPAIR can recover a known signal under artificially induced doubly-unlinked structure. The results confirm that it can, particularly when the number of blocks is large relative to the within-block size. A natural next step would be to apply this framework to the motivating settings described in the introduction, privacy-sensitive digital mental health studies and spatial transcriptomics, but these require access to sensitive data whose release is subject to broader institutional and regulatory considerations. We view this work as initiating that discussion, establishing the statistical feasibility of inference in the doubly-unlinked regime before the data infrastructure to support such analyses is fully in place.

Several future directions remain open. Throughout this paper, we restrict attention to a scalar covariate $X$; extending the framework to multivariate $\bX \in \R^p$
is an immediate priority. Another important direction is to relax the assumption that the same permutation matrix is repeated across all blocks. A third direction concerns the block structure itself: the current formulation preserves coarse-scale ordering while permuting within local neighborhoods, but an equally natural formulation inverts this, preserving fine-scale local structure while permuting the blocks themselves,  a regime that may better reflect certain anonymization mechanisms or distributed data collection pipelines. Characterizing identifiability, recoverability, and the appropriate SNR conditions under this alternative structure are important problems for future work.

\section{Software}
The \texttt{REPAIR} software and vignette are available at \url{https://github.com/isayantan/SpatialReg-Unlinked}.

\section{Acknowledgments}\label{sec:acknowledge}

The authors thank Dr. Debarghya Mukherjee (Department of Statistics, Boston University) for helpful discussions during the early stages of problem formulation. Portions of this research were conducted using the advanced computing resources of the Texas A\&M Department of Statistics Arseven Computing Cluster. The authors also acknowledge the use of Claude and GitHub Copilot for coding assistance and debugging, and Gemini Nano Banana for visualization aid.

\bibliography{bibliography.bib}

\newpage
\renewcommand{\thesection}{S\arabic{section}}
\setcounter{section}{0}

\renewcommand{\thesubsection}{\thesection.\arabic{subsection}}

\makeatletter
\@addtoreset{lemma}{subsection}
\makeatother
\renewcommand{\thelemma}{\thesubsection.\arabic{lemma}}
\setcounter{lemma}{0}

\phantomsection\label{supplementary-material}
\bigskip

\begin{center}
{\large\bf SUPPLEMENTARY MATERIAL}
\end{center}

\section{Details of Variational Inference}
Variational inference is a general approach for approximating a complicated target distribution by a simpler, tractable family of distributions. In Bayesian problems, the target is typically the posterior distribution of latent variables and model parameters given the observed data. When this posterior is analytically unavailable, direct computation is infeasible, and exact sampling-based methods may be computationally burdensome, variational inference replaces the original problem by an optimization problem. More specifically, one selects a family of candidate densities $\mathcal{Q}$ and then chooses the member $q^\ast \in \mathcal{Q}$ that is closest to the true posterior in Kullback--Leibler divergence. Equivalently, this can be formulated as maximizing the evidence lower bound (ELBO), which balances fidelity to the joint model with the entropy of the approximating distribution. The resulting approximation is often much faster to compute than Markov chain Monte Carlo, while still retaining a probabilistic characterization of uncertainty.

A particularly useful feature of variational inference is that it allows one to impose structural simplifications that make high-dimensional problems tractable. A common choice is the mean-field approximation, under which the variational density is factorized across groups of latent variables and parameters. This reduces posterior approximation to a sequence of lower-dimensional optimization steps, often yielding closed-form coordinate updates or efficient gradient-based optimization schemes. In models with latent Gaussian structure, conjugate priors, or exponential-family components, these updates can frequently be written explicitly, which leads to scalable inference even when the ambient dimension is large.

Variational inference is especially well suited to our framework because the doubly-unlinked regression problem introduces several latent objects simultaneously: the regression effect, the latent spatial process, covariance parameters, and two unknown permutation structures. The exact posterior over these quantities is highly coupled and combinatorial, since the permutations interact nonlinearly with both the regression and spatial dependence components. As a result, direct posterior computation is intractable, and naive enumeration over permutation configurations is impossible even for moderate block sizes. Variational inference provides a principled way to approximate this joint posterior while preserving the main dependence structure required for inference.

In our setting, variational inference serves two roles. First, it provides a computationally feasible mechanism for recovering the latent alignment structure induced by the unlinked observations. Second, it yields approximate posterior summaries for the scientific parameters of interest, including the regression coefficient and latent spatial surface, in a way that naturally propagates uncertainty from the unknown permutations. This is particularly important because the latent shuffling is not merely a nuisance computational feature; it is central to the statistical formulation of the problem. By combining a tractable variational family with blockwise structure, we obtain an inference procedure that is both scalable and well adapted to the privacy-preserving nature of the doubly-unlinked framework.
\subsection{Variational Density of Parameters}\label{appendix-vd-parameters}
As discussed in the Section \ref{sec:vb}, we assume a mean field approximation on the class of variational density of our target parameter vector. This provides a closed form solution to the densities which is given by the following theorem:
\begin{theorem}[\cite{ren2011variational}]\label{thm:optimal-vd}
     For a class of variational densities on $\bth = \left(\bth_1^\top,\cdots,\bth_k^\top\right)^\top$ given by the family $\calQ = \left\{q(\bth):q(\bth) = \prod_{l=1}^kq_l(\bth_i)\right\}$, the optimal $q^\ast_l(\bth_l)$ that maximizes the ELBO is given by:
    $$
    q^\ast_l(\bth_l) = \dfrac{\exp\{\E_{u \neq l}\log L(\mathbb{Y}_n,\bth|\mathbb{X}_n)\}}{\int \exp\{\E_{u \neq l}\log L(\mathbb{Y}_n,\bth|\mathbb{X}_n)\}\mathrm{d}\bth_l}
    $$
\end{theorem}

\textbf{Joint Likelihood.} The above theorem requires the computation of conditional joint likelihood $L(\mathbb{Y}_n,\bth\mid\mathbb{X}_n)$ for $n = KB$, which is given as:
\begin{equation}\label{eq:joint-likelihood}
    \begin{split}
&\log L(\mathbb{Y}_n,\bth\mid\mathbb{X}_n) \\
= & \log L(\mathbb{Y}_n\mid\bth,\mathbb{X}_n) + \log p(\bth)\\
= & - \frac{n}{2} \log2 \pi - \frac{n}{2}\log\tau^2-\frac{1}{2 \tau^2} \sum_{i=1}^B\left(\bY_i-\bpix \bX_i \beta -\bpis \bW_i\right)^{\top}\left(\bY_i-\bpix \bX_i \beta -\bpis \bW_i\right) \\
&- \frac{n}{2} \log 2 \pi-\frac{1}{2} \log |\bSig^\ast_n| - \frac{1}{2}\mathbb{W}_n^\top \bSig_n^{\ast-1} \mathbb{W}_n\\
& -\dfrac{1}{2}\log 2\pi\sigma^2_\beta-\dfrac{\beta^2}{2 \sigma_\beta^2}-C_\phi \\
&+a_1 \log b_1-\log \Gamma\left(a_1\right)-\left(a_1+1\right) \log \sigma^2 -\frac{b_1}{\sigma^2} \\
&+a_2 \log b_2-\log \Gamma\left(a_2\right)-\left(a_2+1\right) \log \tau^2 -\dfrac{b_2}{\tau^2}\\
&+\sum_{m=1}^{n}\sum_{k=1}^{n} \log \left[\frac{1}{\sqrt{2 \pi \eta^2_x}}\left\{\exp \left(-\frac{x_{m k}^2}{2 \eta_x^2}\right)+\exp \left(-\frac{\left(x_{m k}-1\right)^2}{2 \eta_x^2}\right)\right\}\right]\\
&+\sum_{m=1}^{n}\sum_{k=1}^{n}\log \left[\frac{1}{\sqrt{2 \pi \eta^2_s}}\left\{\exp \left(-\frac{s_{m k}^2}{2 \eta_s^2}\right)+\exp \left(-\frac{\left(s_{m k}-1\right)^2}{2 \eta_s^2}\right)\right\}\right]
    \end{split}
\end{equation}
Using Theorem~\ref{thm:optimal-vd}, we derive the optimal variational factors $q_l(\cdot)$ for the components of $\bth = (\mathbb{W}_n,\beta,\sigma^2,\phi,\tau^2,\bpix,\bpis)^\top$.
For each parameter \(\theta \in \bth\), the derivation proceeds by isolating all terms in the joint log-likelihood that depend on \(\theta\), while taking expectations with respect to the remaining variational factors. A key quantity in these updates is $\br_i \coloneq \bY_i - \bpix \bX_i \beta - \bpis \bW_i$, which denotes the residual term in the joint likelihood. In particular, for each \(\theta\), we compute the conditional expectation of \(\br_i^\top \br_i\) given \(\theta\), and combine it with the remaining $\theta$-dependent terms obtained after marginalizing over the other variables. This idea works only because of the mean filed approximation. Thus we start off by simplifying the quadratic term:
\begin{equation}\label{eq:residual-qf}
\begin{split}
    \br_i^\top \br_i &= (\bY_i - \bpix \bX_i \beta - \bpis \bW_i)^\top (\bY_i - \bpix \bX_i \beta - \bpis \bW_i) \\
&= \bY_i^\top \bY_i - 2 \bY_i^\top \bpix \bX_i \beta - 2 \bY_i^\top \bpis \bW_i \\
&\quad + \beta^2 \bX_i^\top \bpix^\top \bpix \bX_i \\
&\quad + 2 \beta \bX_i^\top \bpix^\top \bpis \bW_i \\
&\quad + \bW_i^\top \bpis^\top \bpis \bW_i 
\end{split}
\end{equation}
Recall from Section \ref{sec:vb} that $\E_{q_{tau_X}}[\bpix] = \bM^\ast_X$ and $\E_{q_{tau_X}}[\bpix^\top \bpix] = \bV^\ast_X$, while $\E_{q_{tau_S}}[\bpis] = \bM^\ast_S$ and $\E_{q_{tau_S}}[\bpis^\top \bpis] = \bV^\ast_S$. Another useful fact is that the dependence covariance matrix can always be written as $\bSig^\ast_n = \sigma^2 \bR(\phi)$. This decomposition will be used repeatedly in the derivations below, and in particular plays a central role in obtaining closed-form expressions for the parameters of the variational densities corresponding to our target quantities. 

\textbf{Variational density of $\beta$}
\begin{equation}\label{eq:residualsq-conditional-beta}
    \begin{split}
        \E_q[\br_i^\top \br_i \ |\ \beta]
        & = \bY_i^\top \bY_i - 2\beta \bX_i^\top {\bM^
        \ast_X}^\top \bY_i - 2 \bY_i^\top \bM^\ast_S \bmu_{W_i} \\
        &\quad + \beta^2 \bX_i^\top \bV^\ast_X \bX_i \\
        &\quad + 2 \beta \bX_i^\top {\bM^\ast_X}^\top \bM^\ast_S \bmu_{W_i} \\
        &\quad + \bmu_{W_i}^\top \bV^\ast_S \bmu_{W_i} + \text{tr}(\bV^\ast_S \bSig_{W_i})
    \end{split}
\end{equation}
Hence taking expectation over the joint likelihood in (\ref{eq:joint-likelihood}) except $\beta$, we have 
\begin{equation*}
    \begin{split}
        q_\beta(\beta) \propto & \exp\left(-\frac{1}{2} \E_{q_{\tau^2}}\left(\frac{1}{\tau^2}\right) \sum_{i=1}^B \E_q[\br_i^\top \br_i | \beta] - \frac{\beta^2}{2 \sigma^2_{\beta}}\right)\\
        = & \exp\left(-\frac{1}{2} \frac{\lambda_{a_2}}{\lambda_{b_2}} \sum_{i=1}^B \E_q[\br_i^\top \br_i | \beta] - \frac{\beta^2}{2 \sigma^2_{\beta}}\right)
    \end{split}
\end{equation*}
We will use this trick for the following parameters without explicitly mentioning it. This shows that the variational distribution of 
$\beta \sim \calN\left(\mu_{\lambda, \beta}, \sigma^2_{\lambda, \beta}\right)$ where the variational mean and variance are given by:
\begin{equation*}
    \begin{split}
    \sigma^2_{\lambda, \beta} & = \left( \frac{\lambda_{a_2}}{\lambda_{b_2}} \sum_{i=1}^B \bX_i^\top \bV^\ast_X \bX_i + \frac{1}{\sigma_\beta^2} \right)^{-1}\\
    \mu_{\lambda, \beta} & = \sigma^2_{\lambda, \beta} \left( \frac{\lambda_{a_2}}{\lambda_{b_2}}\sum_{i=1}^B \bX_i^\top {\bM^\ast_X}^\top \left(\bY_i - \bM^\ast_S \bmu_{W_i}\right) \right).
    \end{split}
\end{equation*}
\textbf{Variational density of $W$}
\begin{align*}
\E_q[\br_i^\top \br_i | \boldW_n] 
&= \bY_i^\top \bY_i- 2 \mu_{\lambda, \beta} {\bM^\ast_X}^\top \bX_i^\top \bY_i - 2 \bY_i^\top \bM^\ast_S \bW_i \\
&\quad + (\mu_{\lambda, \beta}^2 + \sigma^2_{\lambda, \beta}) \bX_i^\top \bV^\ast_X \bX_i \\
&\quad + 2 \mu_{\lambda, \beta} \bX_i^\top {\bM^\ast_X}^\top \bM^\ast_S \bW_i + \bW_i^\top \bV^\ast_S \bW_i
\end{align*}
Thus we have:
$$
q_W(\boldW_n) \propto \exp\left(-\frac{1}{2} \E_{q_{\tau_X}}\left(\frac{1}{\tau^2}\right) \sum_{i=1}^B \E[\br_i^\top \br_i | \boldW_n] - \frac{1}{2} \E_q\left(\frac{1}{\sigma^2}\boldW_n^\top \bR(\phi)^{-1} \boldW_n\right)\right)
$$
Since the expected log-likelihood is quadratic in \( \bW_i \), we have $\boldW_n \sim \calN(\bmu_{W}, \bSig_{W})$.
Let us denote 
$$
\boldV_S^\ast = \begin{pmatrix}
\bV_S^\ast & \bzero & \cdots & \bzero \\
\bzero & \bV_S^\ast & \cdots & \bzero \\
\vdots & \vdots & \ddots & \vdots \\
\bzero & \bzero & \cdots & \bV_S^\ast
\end{pmatrix}_{KB \times KB} = \bI_B \otimes \bV_S^\ast
$$
where $\otimes$ denotes the Kroneker product between two matrices. Similarly we can define $\boldM_S^\ast,\boldM_X^\ast, \boldV_X^\ast$. Thus the variational mean and variance of $\boldW_n$ are given by:
\begin{equation*}
\begin{split}
    \bSig_{W} = & {\left({\dfrac{\lambda_{a_2}}{\lambda_{b_2}}} (\boldV^\ast_S) + \dfrac{\lambda_{a_1}}{\lambda_{b_1}} \E[[\bR(\phi)]^{-1}] \right)^{-1}}\\
    \bmu_{W} = & \bSig_{W} \left( {{\dfrac{\lambda_{a_2}}{\lambda_{b_2}}} {\boldM^\ast_S}^\top}\left(\mathbb{Y}_B - \mu_{\lambda,\beta}\boldM^\ast_X \mathbb{X}_B\right) \right) = \dfrac{\lambda_{a_2}}{\lambda_{b_2}}\bSig_W\left\{\begin{bmatrix}
    {\bM_S^\ast}^\top\bY_1\\
    {\bM_S^\ast}^\top\bY_2\\
    \vdots\\
    {\bM_S^\ast}^\top\bY_B\\
\end{bmatrix} - \bmu_{\lambda,\beta}\begin{bmatrix}
    {\bM_X^\ast}^\top\bX_1\\
    {\bM_X^\ast}^\top\bX_2\\
    \vdots\\
    {\bM_X^\ast}^\top\bX_B\\
\end{bmatrix}\right\}
\end{split}
\end{equation*}

\textbf{Variational density of $\sigma^2$}

Using the conjugacy property of the Inverse Gamma distribution, it can be easily shown that $\sigma^2 \sim IG(\lambda_{a_1}, \lambda_{b_1})$ where:
\begin{align*}
\lambda_{a_1} &= \frac{KB}{2} + a_1 \\
\lambda_{b_1} &= \frac{1}{2} \left( \tr\left( \E[(\bR(\phi))^{-1}] \bSig_W \right) + \bmu_W^{\top} \E[(\bR(\phi))^{-1}] \bmu_W \right) + b_1 .
\end{align*}

\textbf{Variational density of $\tau^2$}

Using the conjugacy property of the Inverse Gamma distribution, it can be easily shown that $\tau^2 \sim IG(\lambda_{a_2}, \lambda_{b_2})$ where:
\begin{equation*}
    \begin{split}
        \lambda_{a_2} = & \frac{KB}{2} + a_2\\\
        \lambda_{b_2} = & \frac{1}{2} \sum_{i=1}^B \E_q[\br_i^\top \br_i ] + b_2
    \end{split}
\end{equation*}
Now, let us compute the marginal expectation of the quadratic form of $\br_i$:
\begin{equation}\label{eq:error-qf-vi-exp}
    \begin{split}
        \sum_{i = 1}^B \E_q[\br_i^\top \br_i] 
& = \sum_{i = 1}^B \E_{\beta}\left[ \E_q[\br_i^\top \br_i \mid \beta] \right] \\
& \overset{(\ref{eq:residualsq-conditional-beta})}{=} \sum_{i = 1}^B \Big[ 
    \bY_i^\top \bY_i 
    - 2 \mu_{\lambda,\beta} \bX_i^\top {\bM^\ast_X}^\top \bY_i 
    - 2 \bY_i^\top \bM^\ast_S \bmu_{W_i} \\
&\quad + (\mu^2_{\lambda, \beta} + \sigma^2_{\lambda, \beta}) \bX_i^\top \bV^\ast_X \bX_i 
    + 2 \mu_{\lambda, \beta} \bX_i^\top {\bM^\ast_X}^\top \bM^\ast_S \bmu_{W_i} \\
&\quad + \mu_{W_i}^\top \bV^\ast_S \mu_{W_i} 
    + \text{tr}(\bV^\ast_S \bSig_{W_i}) 
\Big] \\
& = \sum_{i = 1}^B (\bY_i - \mu_{\lambda, \beta} \bM^\ast_X \bX_i)^\top (\bY_i - \mu_{\lambda, \beta} \bM^\ast_X \bX_i) \\
&\quad + \mu^2_{\lambda, \beta} \sum_{i = 1}^B \bX_i^\top (\bV^\ast_X - {\bM^\ast_X}^\top \bM^\ast_X) \bX_i + \sigma^2_{\lambda, \beta} \sum_{i = 1}^B \bX_i^\top \bV^\ast_X \bX_i\\
&\quad - 2 \sum_{i = 1}^B \mu_{W_i}^\top {\bM^\ast_S}^\top (\bY_i - \mu_{\lambda, \beta} \bM^\ast_X \bX_i) \\
&\quad + \sum_{i = 1}^B \left[ \bmu_{W_i}^\top \bV^\ast_S \mu_{W_i} + \text{tr}(\bV^\ast_S \Sigma_{W_i}) \right]\\
& = \sum_{i = 1}^B (\bY_i - \mu_{\lambda, \beta} \bM^\ast_X \bX_i - {\bM^\ast_S}\mu_{W_i})^\top (\bY_i - \mu_{\lambda, \beta} \bM^\ast_X \bX_i - {\bM^\ast_S}\bmu_{W_i} ) \\
&\quad + \mu^2_{\lambda, \beta} \sum_{i = 1}^B \bX_i^\top (\bV^\ast_X - {\bM^\ast_X}^\top \bM^\ast_X) \bX_i + \sigma^2_{\lambda, \beta} \sum_{i = 1}^B \bX_i^\top \bV^\ast_X \bX_i\\
&\quad + \sum_{i = 1}^B \left[ \bmu_{W_i}^\top \left(\bV^\ast_S - {\bM^\ast_S}^\top \bM^\ast_S\right)\bmu_{W_i} + \tr(\bV^\ast_S \bSig_{W_i}) \right]
    \end{split}    
\end{equation}

\textbf{Variational density of $\phi$}

All the terms in the likelihood that involves $\phi$ are through the covariance matrix $\bSig_n^\ast$ of $\boldW_n$ which can be written as $\bSig_n^\ast = \sigma^2 \bR(\phi)$. Thus taking expectation of the log likelihood with all the other parameters except $\phi$, we get:
\begin{equation}\label{eq:vi-density-phi}
\begin{split}
q(\phi) \propto&  \exp\left(-\frac{1}{2} \log |\bR(\phi)| - \frac{\lambda_{a_1}}{2\lambda_{b_1}} \left[\tr(\bR(\phi)^{-1} \bSig_W) + \bmu_W^\top \bR(\phi)^{-1} \bmu_W\right]\right)\\
\eqqcolon & c(\phi) 
\end{split}
\end{equation}
This means that we can write $q(\phi) = \dfrac{c(\phi)}{\int_\phi c(\phi)}$ or equivalently 
\begin{equation}\label{eq:phi-vd-log}
 \log q(\phi) + \log \left(\int_\phi c(\phi)\right) = \log c(\phi)   
\end{equation}
\textbf{Variational density for $\bpix$}

Since it is not possible to get the closed form of the variational density of $\bpi_X$, we will compute the ELBO of $\bpix$ here. We will maximise it to get the variational density.
\begin{align*}
\E_{q_{\tau_X}}[\br_i^\top \br_i\mid\bpix] &=  \bY_i^\top \bY_i - 2 \bY_i^\top \bpix \bX_i \mu_{\lambda, \beta} - 2 \bY_i^\top {\bM^\ast_S} \bmu_{W_i} \\
&\quad + (\mu_{\lambda, \beta}^2 + \sigma^2_{\lambda, \beta}) \bX_i^\top \bpix^\top \bpix \bX_i \\
&\quad + 2 \mu_{\lambda, \beta} \bX_i^\top \bpix^\top \bM^\ast_S {\bmu_{W_i}} \\
&\quad + {\bmu_{W_i}^\top \bV_S^\ast \bmu_{W_i} + \tr(\bV_S^\ast\bSig_W)}.
\end{align*}
Thus the ELBO for $\bpix = ((x_{m,k}))_{n \times n}$ is given by:
\begin{align*}
& \text{ELBO}(\bpix)\\
= & \E_{q_{\tau_X}}[\log L(\boldY_n,\bth\mid\boldX_n) - \log q_{\tau_X}(\bpix \mid \zeta_X)] & \\
= & \E_{q_{\tau_X}}\{\E[\log L(\boldY_n,\bth\mid\boldX_n)\mid\bpix] - \log q_{\tau_X}(\bpix \mid \zeta_X)\}\\
=  & C + \E_{q_{\tau_X}}\Bigg(-\frac{\lambda_{a_2}}{2\lambda_{b_2}} \sum_{i = 1}^B (-2 \bY_i^\top \bpix \bX_i \mu_{\lambda, \beta} + (\mu_{\lambda, \beta}^2 + \sigma^2_{\lambda, \beta}) \bX_i^\top \bpix^\top \bpix \bX_i + 2 \mu_{\lambda, \beta} \bX_i^\top \bpix^\top \bM^\ast_S \bmu_{W_i}) +\\
& \sum_{m=1}^n\sum_{k=1}^n \log \left[\frac{1}{\sqrt{2 \pi \eta^2_x}}\frac{1}{2}\left\{\exp \left(-\frac{x_{m k}^2}{2 \eta_x^2}\right)+\exp \left(-\frac{\left(x_{m k} - 1\right)^2}{2 \eta_x^2}\right)\right\}\right]\Bigg)  \\
& + n^2 \log \tau_X + H(\bV_X) \quad \text{(this entropy term is discussed in \cite{linderman2018reparameterizing})}
\end{align*}
where $H(\bV_X = (v_{X,mn})) = 
\frac{1}{2} \sum_{m,n} \log \left(2\pi e v^2_{X,mn}\right)$.

\textbf{Variational density for $\bpis$}

Just Like $\bpix$, we compute the ELBO of $\bpis$ and maximize it to its variational density.
\begin{align*}
\E_q[\br_i^\top \br_i \mid \bpis] &= \bY_i^\top \bY_i 
- 2 \bY_i^\top \bM_X^\ast \bX_i \mu_{\lambda, \beta} 
- 2 \bY_i^\top \bpis \bmu_{W_i} \\
&\quad + (\mu_{\lambda, \beta}^2 + \sigma^2_{\lambda, \beta}) \bX_i^\top \bV^\ast_X \bX_i 
+ 2 \mu_{\lambda, \beta} \bX_i^\top {\bM^\ast_X}^\top \bpis \bmu_{W_i} \\
&\quad + \bmu_{W_i}^\top \bpis^\top \bpis \bmu_{W_i} 
+ \tr(\bSig_{W_i} \bpis^\top \bpis)
\end{align*}
Thus the ELBO for $\bpis = ((s_{m,k}))_{n \times n}$ is given by:
\begin{align*}
& \text{ELBO}(\bpis)\\
= & \E_{q_{\tau_S}}\left[\log L(\boldY_n, \bth \mid \boldX_n) - \log q_{\tau_S}(\bpis\mid \zeta_S)\right] \\
= & \E_{q_{\tau_S}}\Big[ \E\left[\log L(\boldY_n, \bth \mid \boldX_n) \mid \bpis\right] 
- \log q_{\tau_S}(\bpis\mid \zeta_S) \Big] \\
= & C 
+ \E_{q_{\tau_S}}\Bigg\{ -\frac{\lambda_{a_2}}{2\lambda_{b_2}} \sum_{i = 1}^B\Big(
- 2 \bY_i^\top \bpis \bmu_{W_i} 
+ 2 \mu_{\lambda, \beta} \bX_i^\top {\bM^\ast_X}^\top \bpis \bmu_{W_i} \\
&\quad + \bmu_{W_i}^\top \bpis^\top \bpis \bmu_{W_i} 
+ \tr\left(\bSigma_{W_i} \bpis^\top \bpis\right) \Big) \\
&\qquad + \sum_{m=1}^n \sum_{k=1}^n \log \left[ \frac{1}{\sqrt{2\pi \eta_s^2}} \frac{1}{2}
\left( \exp\left( -\frac{s_{mk}^2}{2 \eta_s^2} \right) 
+ \exp\left( -\frac{(s_{mk}-1)^2}{2 \eta_s^2} \right) \right) \right] \Bigg\} \\
&\qquad + n^2 \log \tau_S + H(\bV_S)
\end{align*}
where $H(\bV_S = (v_{S,mn})) = 
\frac{1}{2} \sum_{m,n} \log \left(2\pi e v^2_{S,mn}\right)$.
\subsection{Full ELBO}\label{appendix-global-elbo}
Here we compute the full ELBO, which will be used as a stopping criterion for our Algorithm \ref{alg-parameter-est}. There are two components to the ELBO calculation: the expectation of the joint log likelihood $\log L(\mathbb{Y}_n,\bth\mid\mathbb{X}_n)$ and the log of joint density $q(\cdot \mid \blam)$, both calculated with respect to the joint variational density $q(\cdot \mid \blam)$. Denote $N(x; \mu,\sigma^2)$ the density of a univariate Gaussian distribution. Also recall that $\br_i = \bY_i-\bpix \bX_i \beta -\bpis \bW_i$ and $\bSig^\ast_n = \sigma^2 \bR(\phi)$. Ignoring the constants, the joint likelihood (\ref{eq:joint-likelihood}) is proportional to:
\begin{equation}
    \begin{split}
        & \log L(\mathbb{Y}_n,\bth\mid\mathbb{X}_n) \\
        \propto & -\left(a_2 + \frac{n}{2} + 1\right)\log \tau^2 -\left(b_2 + \frac{1}{2}\sum_{i=1}^B \br_i^\top \br_i\right)\frac{1}{\tau^2} \\
        & -\left(a_1 + \frac{n}{2} + 1\right)\log \sigma^2 - \left(b_1 + \frac{1}{2}\mathbb{W}_n^\top \bR(\phi)^{-1} \mathbb{W}_n\right)\frac{1}{\sigma^2} \\
        & -\frac{1}{2} \log |\bR(\phi)| - \frac{\beta^2}{2\sigma^2_\beta}\\
        &+\sum_{m=1}^{n}\sum_{k=1}^{n} \log \left[\dfrac{1}{2}\left\{N(x_{m,k};0,\eta^2_x) + N(x_{m,k};1,\eta^2_x)\right\}\right]\\
&+\sum_{m=1}^{n}\sum_{k=1}^{n} \log \left[\dfrac{1}{2}\left\{N(s_{m,k};0,\eta^2_s) + N(s_{m,k};1,\eta^2_s)\right\}\right]
    \end{split}
\end{equation}
Let us note a few things here. For the Inverse Gamma target variational distributions for $\tau^2$ and $\sigma^2$, we have $\E_{q_{\tau^2}} \log \tau^2 = \log \lambda_{b_2} - \psi(\lambda_{a_2})$, $\E_{q_{\tau^2}} \tau^{-2} = \lambda_{a_2}/\lambda_{b_2}$ and similarly $\E_{q_{\sigma^2}} \log \sigma^2 = \log \lambda_{b_1} - \psi(\lambda_{a_1})$, $\E_{q_{\sigma^2}} \sigma^{-2} = \lambda_{a_1}/\lambda_{b_1}$ for the Digamma function $\psi(\cdot)$. Also note that:
\begin{equation*}
    \begin{split}
        \E_q \left[\mathbb{W}_n^\top \bR(\phi)^{-1} \mathbb{W}_n\right] = & \E_q \E_q \left[\mathbb{W}_n^\top \bR(\phi)^{-1} \mathbb{W}_n \mid \phi\right]\\
        = & \E_q\left[\bmu_W^\top \bR(\phi)^{-1} \bmu_W + \tr(\bR(\phi)^{-1}\bSig_W)\right]\\
        = & \bmu_W^\top \E_q\left[\bR(\phi)^{-1}\right] \bmu_W + \tr\left(\E_q\left[\bR(\phi)^{-1}\right]\bSig_W\right)
    \end{split}
\end{equation*}
Thus, we have:
\begin{equation}\label{eq:elbo-joint-lkhd}
    \begin{split}
        & \E_q \left[\log L(\mathbb{Y}_n,\bth\mid\mathbb{X}_n)\right] \\
        \propto & -\left(a_2 + \frac{n}{2} + 1\right)[\log \lambda_{b_2} - \psi(\lambda_{a_2})] -\left(b_2 + \frac{1}{2}\sum_{i=1}^B \E_q\left[\br_i^\top \br_i\right]\right)\frac{\lambda_{a_2}}{\lambda_{b_2}} \\
        & -\left(a_1 + \frac{n}{2} + 1\right)[\log \lambda_{b_1} - \psi(\lambda_{a_1})] - \frac{1}{2}\left(2b_1 + \bmu_W^\top \E_q\left[\bR(\phi)^{-1}\right] \bmu_W + \tr\left(\E_q\left[\bR(\phi)^{-1}\right]\bSig_W\right)\right)\frac{\lambda_{a_1}}{\lambda_{b_1}} \\
        & -\frac{1}{2} \E_q\left[\log |\bR(\phi)|\right] - \frac{\mu^2_{\lambda,\beta} +\sigma^2_{\lambda,\beta}}{2\sigma^2_\beta}\\
        &+\sum_{m=1}^{n}\sum_{k=1}^{n}\E_q \left\{\log \left[\dfrac{1}{2}\left\{N(x_{m,k};0,\eta^2_x) + N(x_{m,k};1,\eta^2_x)\right\}\right]\right\}\\
&+\sum_{m=1}^{n}\sum_{k=1}^{n} \E_q\left\{\log \left[\dfrac{1}{2}\left\{N(s_{m,k};0,\eta^2_s) + N(s_{m,k};1,\eta^2_s)\right\}\right]\right\}\\
= & -\left(a_2 + \frac{n}{2} + 1\right)[\log \lambda_{b_2} - \psi(\lambda_{a_2})] -\left(b_2 + \frac{1}{2}\sum_{i=1}^B \E_q\left[\br_i^\top \br_i\right]\right)\frac{\lambda_{a_2}}{\lambda_{b_2}} -\frac{b_1\lambda_{a_1}}{\lambda_{b_1}} - \frac{\mu^2_{\lambda,\beta} +\sigma^2_{\lambda,\beta}}{2\sigma^2_\beta}\\
& {+} \E_q\left[-\frac{1}{2} \log |\bR(\phi)| - \frac{\lambda_{a_1}}{2\lambda_{b_1}} \left[\tr(\bR(\phi)^{-1} \bSig_W) + \bmu_W^\top \bR(\phi)^{-1} \bmu_W\right]\right] \\
&+\sum_{m=1}^{n}\sum_{k=1}^{n}\E_q \left\{\log \left[\dfrac{1}{2}\left\{N(x_{m,k};0,\eta^2_x) + N(x_{m,k};1,\eta^2_x)\right\}\right]\right\}\\
&+\sum_{m=1}^{n}\sum_{k=1}^{n} \E_q\left\{\log \left[\dfrac{1}{2}\left\{N(s_{m,k};0,\eta^2_s) + N(s_{m,k};1,\eta^2_s)\right\}\right]\right\}\\
= & -\left(a_2 + \frac{n}{2} + 1\right)[\log \lambda_{b_2} - \psi(\lambda_{a_2})] -\left(b_2 + \frac{1}{2}\sum_{i=1}^B \E_q\left[\br_i^\top \br_i\right]\right)\frac{\lambda_{a_2}}{\lambda_{b_2}} -\frac{b_1\lambda_{a_1}}{\lambda_{b_1}} - \frac{\mu^2_{\lambda,\beta} +\sigma^2_{\lambda,\beta}}{2\sigma^2_\beta}\\
& + \E_{q_\phi}\left[\log q(\phi)\right] + \log \left(\int_\phi c(\phi)\right)\quad \text{(See Equation \ref{eq:phi-vd-log})} \\
&+\sum_{m=1}^{n}\sum_{k=1}^{n}\E_{q_{\tau_X}} \left\{\log \left[\dfrac{1}{2}\left\{N(x_{m,k};0,\eta^2_x) + N(x_{m,k};1,\eta^2_x)\right\}\right]\right\}\\
&+\sum_{m=1}^{n}\sum_{k=1}^{n} \E_{q_{\tau_S}}\left\{\log \left[\dfrac{1}{2}\left\{N(s_{m,k};0,\eta^2_s) + N(s_{m,k};1,\eta^2_s)\right\}\right]\right\}\\
    \end{split}
\end{equation}
Let us now shift to the second part of the ELBO, which is the expectation of the joint log-variational densities of our parameters of interest. For $\bth = (\mathbb{W}_n,\beta,\sigma^2,\phi,\tau^2,\bpix,\bpis)^\top$, recall that using mean field approximation, we have:
\begin{equation*}
\begin{split}
    q(\bth \mid\blam) \coloneq &\, q_W\left(\mathbb{W}_n \mid \bmu_W, \bSig_W\right) \, \cdot\,  q_\beta(\beta\mid\mu_{\lambda,\beta},\sigma^2_{\lambda,\beta}) \, \cdot\,  q_{\sigma^2}\left(\sigma^2\mid\lambda_{a_1},\lambda_{b_1}\right)\,\cdot\, q_{\tau^2}\left(\tau^2\mid\lambda_{a_2},\lambda_{b_2}\right) \\
    &   \,\cdot\, q_\phi(\phi)\,\cdot\, q_{\tau_X}(\bpix\mid\zeta_X)\,\cdot\, q_{\tau_S}(\bpis\mid\zeta_S).
\end{split}
\end{equation*}
Hence, we can say:
\begin{equation*}
    \begin{split}
        \calH\{q(\bth \mid \blam)\}
= & -\E_{q}\!\left[\log q(\bth \mid \blam)\right] \\
= & 
\calH\!\left(q(\mathbb{W}_n)\right)
+ \calH\!\left(q(\beta)\right)
+ \calH\!\left(q(\sigma^2)\right)\\
& + \calH\!\left(q(\phi)\right)
+ \calH\!\left(q(\tau^2)\right)
+ \calH\!\left(q(\bpix)\right)
+ \calH\!\left(q(\bpis)\right).
    \end{split}
\end{equation*}
Since $n=\dim(\mathbb{W}_n)$, for the Gaussian factors,
\begin{align*}
\calH\!\left(q(\mathbb{W}_n)\right)
&=
\frac{1}{2}\log\!\left( (2\pi e)^n \, \lvert \boldsymbol{\Sigma}_W \rvert \right),\\
\calH\!\left(q(\beta)\right)
&=
\frac{1}{2}\log\!\left( 2\pi e \, \sigma^2_{\lambda,\beta} \right)
\end{align*}
For the inverse-gamma factors, using the parameterization Inv-Gamma$(\alpha,\beta)$,
\begin{align*}
\calH\!\left(q(\sigma^2)\right)
&=
\lambda_{a_1}
+ \log(\lambda_{b_1})
+ \log\Gamma(\lambda_{a_1})
- (1+\lambda_{a_1})\psi(\lambda_{a_1}), \\
\calH\!\left(q(\tau^2)\right)
&=
\lambda_{a_2}
+ \log(\lambda_{b_2})
+ \log\Gamma(\lambda_{a_2})
- (1+\lambda_{a_2})\psi(\lambda_{a_2}),
\end{align*}
where $\psi(\cdot)$ denotes the digamma function. For the permutation matrices, using the results from \cite{linderman2018reparameterizing}, we have:
\begin{align*}
\calH\left(q(\boldsymbol{\pi}_X)\right)
&=  n^2 \log \tau_X + H(\bV_X), \\
\calH\!\left(q(\boldsymbol{\pi}_S)\right)
&=  n^2 \log \tau_S + H(\bV_S),
\end{align*}
where $H(\bV_U = (v_{U,mn})) = 
\frac{1}{2} \sum_{m,n} \log \left(2\pi e v^2_{U,mn}\right)$. The only parameter whose entropy does not have a closed form is $\phi$, and hence we have used importance sampling to compute $\calH\left(q(\phi)\right) = -\E_{q(\phi)}[\log q(\phi)]$, although its contribution in the total ELBO is 0 as there is a negative factor of this in the joint likelihood term of the ELBO (see Equation \ref{eq:elbo-joint-lkhd}).

\newpage
\section{Proofs}
\subsection{Propositions used in Theorems}
    \begin{proposition}[Conditional Concentration of $\bE_1$ (\ref{eq:error-partition})]\label{prop:E1-dist-given-X}
         For $\bT_{\bPi_1, \bPi_2} = \bP_{\hatPi_1, X}^\perp {\bSig}^{-1/2} \hatPi_2 \bPi_2^\top \bPi_1 \bX$, and $\bM = {\bSig}^{-1/2} \hatPi_2 \bPi_2^\top {\bSig}^{1/2}$, conditioned on the event $\left\|\bT_{\bPi_1, \bPi_2} \right\|^2 > t^\ast$ for some $t^\ast>0$, we have:
        \begin{equation}\label{eq:E1_dist_given_X}
            \Prob\left(\bE_1 < 2\beta^2\delta^\ast\right) \leq \exp\left(-c_1^\ast\dfrac{\beta^2}{\kappa(\bSig)}\cdot t^\ast\right)
        \end{equation} 
        for $\delta^\ast = 
        c\left\|\bT_{\bPi_1, \bPi_2} \right\|^2$ and some universal constants $c,c_1^\ast > 0$.
    \end{proposition}
    
\begin{proof}
           Recall that:
    \begin{eqnarray*}
        \bT_{\bPi_1, \bPi_2} & = & \bP_{\hatPi_1, X}^\perp {\bSig}^{-1/2} \hatPi_2 \bPi_2^\top \bPi_1 \bX = \bP_{\hatPi_1, X}^\perp \bM \widetilde{\bX}_{\bPi_1}\\
        \bV & = & \bP_{\hatPi_1,\bX}^\perp {\bSig}^{-1/2} \hatPi_2 \bPi_2^\top
    \end{eqnarray*}
    where $\bM = {\bSig}^{-1/2} \hatPi_2 \bPi_2^\top {\bSig}^{1/2}$ and $\bP_{\hatPi_1, X} = Proj\left(\widetilde{\bX}_{\hatPi_1}\right) = \frac{\widetilde{\bX}_{\hatPi_1}\widetilde{\bX}_{\hatPi_1}^\top}{\widetilde{\bX}_{\hatPi_1}^\top \widetilde{\bX}_{\hatPi_1}}$ for $\tilX_{\hatPi_1} = {\bSig}^{-1/2}\hatPi_1\bX$. Then we get:
    \begin{eqnarray*}
     \bE_1 & = & \left\| \bP_{\hatPi_1, X}^\perp {\bSig}^{-1/2} \hatPi_2 \bPi_2^\top \bPi_1 \bX \beta \right\|_2^2 
        + 2 \left\langle 
        \bP_{\hatPi_1, X}^\perp {\bSig}^{-1/2} \hatPi_2 \bPi_2^\top \bPi_1 \bX \beta,\ 
        \bP_{\hatPi_1, X}^\perp {\bSig}^{-1/2} \hatPi_2 \bPi_2^\top \bW^\ast)\right\rangle\\
        & = & \beta^2 \left\|\bT_{\bPi_1, \bPi_2} \right\|^2 + 2 \beta\langle\bT_{\bPi_1, \bPi_2},\bV\bW^\ast\rangle\\
        & = & \beta^2 \left\|\bT_{\bPi_1, \bPi_2} \right\|^2 + 2 \beta\bT_{\bPi_1, \bPi_2}^\top\bV\bW^\ast
    \end{eqnarray*}
    Thus conditioned on $\bX$ i.e. conditional on $\left\|\bT_{\bPi_1, \bPi_2} \right\|$, we have:
    \begin{eqnarray*}
        \bE_1\big|\left\|\bT_{\bPi_1, \bPi_2} \right\|^2  \sim \mathcal{N} \left( \beta^2 \left\|\bT_{\bPi_1, \bPi_2} \right\|^2, 4\beta^2 \bT_{\bPi_1, \bPi_2}^\top \bV {\bSig} \bV^\top \bT_{\bPi_1, \bPi_2} \right)
    \end{eqnarray*}
    Let us expand $\bT_{\bPi_1, \bPi_2}^\top \bV {\bSig} \bV^\top \bT_{\bPi_1, \bPi_2}$:
    \begin{eqnarray*}
        && \bT_{\bPi_1, \bPi_2}^\top \bV {\bSig} \bV^\top \bT_{\bPi_1, \bPi_2} \\
        & = & [(\bM \widetilde{\bX}_{\bPi_1})^\top\bP_{\hatPi_1, X}^\perp][ \bP_{\hatPi_1,\bX}^\perp {\bSig}^{-1/2} \hatPi_2 \bPi_2^\top] {\bSig}^{1/2} \bSig^{1/2}[({\bSig}^{-1/2} \hatPi_2 \bPi_2^\top)^\top\bP_{\hatPi_1,\bX}^\perp ]\bP_{\hatPi_1, X}^\perp \bM \widetilde{\bX}_{\bPi_1}\\
        & = & [(\bM \widetilde{\bX}_{\bPi_1})^\top\bP_{\hatPi_1, X}^\perp][ {\bSig}^{-1/2} \hatPi_2 \bPi_2^\top{\bSig}^{1/2}][  \bSig^{1/2}({\bSig}^{-1/2} \hatPi_2 \bPi_2^\top)^\top ]\bP_{\hatPi_1, X}^\perp \bM \widetilde{\bX}_{\bPi_1}\\
        & = & \bT_{\bPi_1, \bPi_2}^\top \bM\bM^\top \bT_{\bPi_1, \bPi_2}\\
        & = & \|\bM^\top \bT_{\bPi_1, \bPi_2}\|^2
    \end{eqnarray*}
    Thus, using the definition of $\bM$ we can equivalently write:
    \begin{eqnarray*}
        \bE_1\ \big|\ \left\|\bT_{\bPi_1, \bPi_2} \right\|^2  \sim \mathcal{N} \left(\beta^2 \left\|\bT_{\bPi_1, \bPi_2} \right\|^2, 4\beta^2 \left\|\bM^\top \bT_{\bPi_1, \bPi_2}\right\|^2 \right) 
    \end{eqnarray*}

 Thus for any $\delta > 0$ and conditioned on $\left\|\bT_{\bPi_1, \bPi_2} \right\|^2$, we have:
 \begin{eqnarray*}
     \Prob(\bE_1 <2\beta^2\delta^\ast) & \leq & \exp\left(-c_1\dfrac{\left(\beta^2\left\|\bT_{\bPi_1, \bPi_2} \right\|^2 - 2 \beta^2\delta^\ast\right)^2}{8\beta^2\left\|\bM^\top T_{\bPi_1, \bPi_2}\right\|^2}\right)\\
     & = & \exp\left(-c_1\beta^2\dfrac{(\left\|\bT_{\bPi_1, \bPi_2} \right\|^2 - 2 \delta^\ast)^2}{8\left\|\bM^\top T_{\bPi_1, \bPi_2}\right\|^2}\right)
 \end{eqnarray*}
Now observe the fact that $\|\bM\| = \|{\bSig}^{-1/2} \hatPi_2 \bPi_2^\top {\bSig}^{1/2}\| \leq \|{\bSig}^{-1/2} \|\cdot \|\hatPi_2\|\cdot \| \bPi_2^\top\| \cdot\|{\bSig}^{1/2}\|\leq \sqrt{\kappa(\bSig)}$ and also the fact that $\left\|\bM^\top T_{\bPi_1, \bPi_2}\right\|^2 \leq \left\|\bM\right\|^2 \left\| T_{\bPi_1, \bPi_2}\right\|^2 \leq \kappa(\bSig)\left\| T_{\bPi_1, \bPi_2}\right\|^2$. Thus for the choice of $\delta^\ast = c\left\|\bT_{\bPi_1, \bPi_2} \right\|^2$ gives us:
\begin{eqnarray*}
    \dfrac{(\left\|\bT_{\bPi_1, \bPi_2} \right\|^2 - 2 \delta^\ast)^2}{8\left\|\bM^\top \bT_{\bPi_1, \bPi_2}\right\|^2} =\dfrac{\left(2c-1\right)^2\left\|\bT_{\bPi_1, \bPi_2} \right\|^4}{8\left\|\bM^\top \bT_{\bPi_1, \bPi_2}\right\|^2}
    \geq  c_2\dfrac{\left\|\bT_{\bPi_1, \bPi_2} \right\|^4}{\kappa(\bSig)\left\| \bT_{\bPi_1, \bPi_2}\right\|^2} = c_2\dfrac{\left\|\bT_{\bPi_1, \bPi_2} \right\|^2}{\kappa(\bSig)}\\
\end{eqnarray*}
Thus conditioned on the event $\left\|\bT_{\bPi_1, \bPi_2} \right\|^2> t^\ast$, we have:
\begin{equation*}
    \Prob_W(\bE_1 <2\beta^2\delta^\ast) \leq \exp\left(-c_1c_2\beta^2 \dfrac{\left\|\bT_{\bPi_1, \bPi_2} \right\|^2}{\kappa(\bSig)}\right) \leq \exp\left(-c_1^\ast\dfrac{\beta^2}{\kappa(\bSig)}\cdot t^\ast\right)
\end{equation*}
for some universal constant $c_1^\ast >0$.
 \end{proof}

\begin{proposition}\label{prop:E21-dist-given-X}
    For $\bE_{21}$ defined in Equation \ref{eq:error-partition}, conditioned on the event $\left\|\bT_{\bPi_1, \bPi_2} \right\|^2 > t^\ast$, we have:
    \begin{eqnarray}\label{eq:E21-dist-given-X}
        \Prob\left(|\bE_{21}| > \beta^2\delta^\ast/2\right) & \leq & \left(-c^\ast_{21}\beta^2\cdot t^\ast\right)
    \end{eqnarray}
    for $\delta^\ast = c\left\|\bT_{\bPi_1, \bPi_2} \right\|^2$ and some universal constants $c,c^\ast_{21} > 0$.
\end{proposition}
\begin{proof}
Observe that the term: 
\begin{equation*}
\bE_{21} = \left\| \bP_{{\bPi}_1, X}^\perp {\bSig}^{-1/2} \bW^\ast \right\|_2^2 - \left\| \bP_{\hatPi_1, X}^\perp {\bSig}^{-1/2} \bW^\ast \right\|_2^2 = \left\| \bP_{\hatPi_1, X} {\bSig}^{-1/2} \bW^\ast \right\|_2^2 - \left\| \bP_{{\bPi}_1, X} {\bSig}^{-1/2} \bW^\ast \right\|_2^2
\end{equation*}
As $\bW^\ast \sim(\bzero, {\bSig})$, we have $\bet = {\bSig}^{-1/2}\bW^\ast\sim N(\bzero,\bI_n) $, and since $\bP_{{\bPi}_1, X}$ and $\bP_{\hatPi_1, X}$ are rank 1 projection matrices almost surely, we can write $\bE_{21} = Z_1 - \tilde{Z}_1$ where both $Z_1$ and $\tilde{Z}_1$ are $\chi^2_1$ random variables almost surely (not necessarily independent). Thus we have:
$$
\begin{aligned}
    \Prob(|\bE_{21}| >t) = & \E_{X}\E_{W|X}[\mathbb{I}(|\bE_{21}| >t)|X]\\
    = &\E_{X}\E_{W|X}[\mathbb{I}(|Z_1 - \tilde{Z}_1| >t)|X]\\
    = & \E_{X}\Prob[|Z_1 - \tilde{Z}_1| >t|X]\\
    = & \Prob(|Z_1 - \tilde{Z}_1| >t)\\
    \leq & \Prob(|Z_1 - 1|> t/2) + \Prob(|\tilde{Z}_1 - 1|> t/2)
    \leq & c^\prime \exp\left(- \dfrac{c_{21}}{\sqrt{2}}t\right)
\end{aligned}
$$
This follows from the fact that $\chi^2$ random variables are sub-exponential. Hence, conditional on $\left\|\bT_{\bPi_1, \bPi_2} \right\|^2>t^\ast$, we have:
\begin{eqnarray*}
    \Prob\left(|\bE_{21}| > \beta^2\delta^\ast/2\right) & = &  \Prob\left(|\bE_{21}| > \beta^2 \cdot c\left\|\bT_{\bPi_1, \bPi_2} \right\|^2\right)\\
    & \leq & c'\exp\left(-{c_{21}c}\ \beta^2 \left\|\bT_{\bPi_1, \bPi_2} \right\|^2\right)\\
    & \leq & c'\exp\left(-c_{21}c\cdot\beta^2\cdot t^\ast\right)\\
    & \leq & \left(-c^\ast_{21}\beta^2\cdot t^\ast\right)
\end{eqnarray*}

\end{proof}

\begin{proposition}\label{prop:E22-dist-cond-X}
    Conditioned on the event $\left\|\bT_{\bPi_1, \bPi_2} \right\|^2 > t^\ast$, for $\bE_{22}$ defined in Equation \ref{eq:error-partition} and any $ \lambda_{\max}(\bSig) t^\ast \geq \dfrac{h_2}{\snr}$ where $h_2 \coloneq d_H(\hatPi_2,\bPi_2)$, we have:
    {\small
    \begin{equation*}\label{eq:E22-dist-cond-X}
            \Prob\left(|\bE_{22}| > \beta^2\delta^\ast/2\right)\leq  \exp \left(-c^\ast_{22}\snr \min \left\{\frac{\snr}{h_2}\left(\lambda_{\max}(\bSig) t^\ast - \dfrac{h_2}{\snr}\right)^2,\left(\lambda_{\max}(\bSig) t^\ast - \dfrac{h_2}{\snr}\right)\right\}\right)
    \end{equation*}}
    for $\delta^\ast = c\left\|\bT_{\bPi_1, \bPi_2} \right\|^2$ and some universal constant $c,c_{22} > 0$, and $\snr$ defined in Section \ref{sec:theory}.
\end{proposition}
\begin{proof}
Observe that:
    \begin{equation}
\begin{split}
\bE_{22} & = \left\| \bP_{\hatPi_1, X}^\perp {\bSig}^{-1/2} \bW^\ast \right\|_2^2
- \left\| \bP_{\hatPi_1, X}^\perp {\bSig}^{-1/2} \hatPi_2 {\bPi}_2^\top \bW^\ast \right\|_2^2\\
& =  \bZ^\top \bP_{\hatPi_1, X}^\perp \bZ 
- \bZ^\top {\bSig}^{1/2} {\bPi}_2 \hatPi_2^\top {\bSig}^{-1/2} \bP_{\hatPi_1, X}^\perp {\bSig}^{-1/2} \hatPi_2 {\bPi}_2^\top {\bSig}^{1/2} \bZ \\
& = \bZ^\top (\bP_{\hatPi_1, X}^\perp - \bM^\top \bP_{\hatPi_1, X}^\perp \bM)\bZ
\end{split}
\end{equation}
where $\bM = {\bSig}^{-1/2} \hatPi_2 {\bPi}_2^\top {\bSig}^{1/2}$, and $ \bZ= {\bSig}^{-1/2}\bW^\ast \sim N(0, \bI_n)$. We analyze the quadratic form:
\begin{equation}\label{eq:E22}
\bZ^\top \bA \bZ \coloneqq \bZ^\top (\bP_{\hatPi_1, X}^\perp - \bM^\top \bP_{\hatPi_1, X}^\perp \bM)\bZ
\end{equation}

The Hanson-Wright inequality (\ref{lem:hanson-wright}) gives the following tail bound for any $t^\ast \geq 0$:
\begin{equation}
\Prob \left( \left| \bZ^\top \bA \bZ - \tr(\bA) \right| > t^\ast \right)
\leq \exp \left( - c_{22} \, \min \left( \frac{{t^\ast}^2}{\|\bA\|_F^2}, \, \frac{t^\ast}{\|\bA\|_2} \right) \right),
\end{equation}
Then note that we have the following, where the second inequality follows from triangle inequality and third one from Hanson-Wright (as long as $t^\ast > |\tr(\bA)|$):
\begin{eqnarray*}
    \Prob \left(|\bZ^{\top} \bA \bZ|  > t^\ast \right) & = & \Prob \left(|\bZ^{\top} \bA \bZ| - \left| \tr(\bA) \right|  >t^\ast - \left| \tr (\bA) \right|\right) \\
    & \leq & \Prob \left( \left| \bZ^{\top} \bA Z - \tr(\bA) \right| > t^\ast - \left| \tr (\bA) \right| \right) \\
    & \leq &  \exp \left( - c_{22} \, \min \left\{ \frac{(t^\ast - |tr(\bA)|)^2}{\|\bA\|_F^2}, \, \frac{t^\ast - |tr(\bA)|}{\|\bA\|_2} \right\} \right).
\end{eqnarray*}
By Lemma \ref{lem:specA}, we have 
$\|\bA\|_{2} \leq \kappa(\bSig) \leq \sqrt{2}\kappa(\bSig)$, $\|\bA\|_{F} \leq \sqrt{\textrm{rank}(A)}\|\bA\|_{2} \leq \sqrt{2h_2} \kappa({\bSig})$ and $|\tr(\bA)| \leq \sqrt{2}h_2\kappa(\bSig)$. Using the above results:
\begin{eqnarray*}
    \Prob \left(  |\bZ^{\top} \bA \bZ|  > t^\ast \right)
    & \leq &  \exp \left( - c_{22} \, \min \left\{ \frac{(t^\ast-\sqrt{2}h_2\kappa(\bSig))^2}{2h_2\kappa^2({\bSig})}, \, \frac{t^\ast -\sqrt{2}h_2\kappa(\bSig)}{\sqrt{2}\kappa({\bSig})} \right\} \right).
\end{eqnarray*}
Thus plugging in $t^\ast = {\beta^2\delta^\ast}/{2}$ for $\delta^\ast = c\left\|\bT_{\bPi_1, \bPi_2} \right\|^2$, and conditioning on $\left\|\bT_{\bPi_1, \bPi_2} \right\|^2> t^\ast$ for some $t^\ast>0$, we have:
\begin{eqnarray*}
    t^\ast - \sqrt{2}h_2\kappa(\bSig)
    & = & \beta^2 \cdot \frac{c}{2}\left\|\bT_{\bPi_1, \bPi_2} \right\|^2 - \sqrt{2}h_2\kappa(\bSig)\\
    & > & \sqrt{2}c^\ast\left[\beta^2 t^\ast - h_2\kappa(\bSig)\right]\\
    & = & \sqrt{2}c^\ast\left[{\kappa(\bSig)\ \snr}\cdot \lambda_{\max}(\bSig) t^\ast - h_2\kappa(\bSig)\right]\\
    & = & \sqrt{2}c^\ast\kappa(\bSig)\ \snr\left[\lambda_{\max}(\bSig) t^\ast - \dfrac{h_2}{\snr}\right]
\end{eqnarray*}
Hence for $\lambda_{\max}(\bSig) t^\ast \geq \frac{h_2}{\snr}$, and conditioned on $\left\|\bT_{\bPi_1, \bPi_2} \right\|^2> t^\ast$, we have:
{
\small
\begin{eqnarray*}
    \Prob\left(|\bE_{22}| > \beta^2\delta^\ast/2\right)\leq  \exp \left( - c^\ast_{22} \, \snr \, \min \left\{\frac{\snr}{h_2}\left(\lambda_{\max}(\bSig) t^\ast - \dfrac{h_2}{\snr}\right)^2, \, \left(\lambda_{\max}(\bSig) t^\ast - \dfrac{h_2}{\snr}\right) \right\}\right)
\end{eqnarray*}}
\end{proof}

\begin{proposition}\label{prop:prob-T-less-t}
    For $\bT_{\bPi_1, \bPi_2} = \bP_{\hatPi_1, X}^\perp {\bSig}^{-1/2} \hatPi_2 \bPi_2^\top \bPi_1 \bX$, under the assumption $\bX \sim \calN(\bzero,\bI_n)$ and $d_H\left(\hatPi_2\bPi_2^\top\bPi_1 \hatPi_1^\top,\bI_n\right) = h_{12}$, for any $0 \leq  \lambda_{\max}(\bSig) t^\ast \leq h_{12}$ , we have:
    \begin{eqnarray}\label{eq:prob-T-less-t}
        \Prob\left(\left\|\bT_{\bPi_1, \bPi_2} \right\|^2 < t^\ast\right) 
        & \leq &  6 \exp\left(-\frac{h_{12}}{10}\left[\log \frac{h_{12}}{\lambda_{\max}(\bSig)t^\ast} + \dfrac{\lambda_{\max}(\bSig)t^\ast}{h_{12}} -1\right]\right)
    \end{eqnarray}
\end{proposition}
\begin{proof}
Before simplifying $\left\|\bT_{\bPi_1, \bPi_2} \right\|^2$, we define two quantities and a norm:
$$
\begin{aligned}
    \bX_{1}\coloneqq & \hatPi_1\bX = {\bSig}^{1/2}\tilX_{\hatPi_1}\\
 \bX_{12}\coloneqq&\hatPi_2\bPi_2^\top\bPi_1\bX =\hatPi_2\bPi_2^\top\bPi_1 \hatPi_1^\top\bX_{1} \coloneqq \hatPi_{12}\bX_1\\
    \langle a,b\rangle_{\Siginv} \coloneqq& a^\top\Siginv b
\end{aligned}
$$
So notice one fact that $\|\bX_{12}\|^2_2 = \|\bX_1\|^2_2 = \|\bX\|^2_2$.
Using these, and the fact that $\bM = {\bSig}^{-1/2} \hatPi_2 \bPi_2^\top {\bSig}^{1/2}$ we make few simplifications:
$$
\begin{aligned}
    \tilX^\top_{\bPi_1}\bM^\top \bM\tilX_{\bPi_1} 
    = & (\bX^\top\bPi_1^\top{\bSig}^{-1/2})({\bSig}^{1/2} \bPi_2 \hatPi_2^\top{\bSig}^{-1/2}) ({\bSig}^{-1/2} \hatPi_2\bPi_2^\top{\bSig}^{1/2}){\bSig}^{-1/2}\bPi_1\bX\\
    = & \bX^\top\bPi_1^\top\bPi_2 \hatPi_2^\top\Siginv \hatPi_2\bPi_2^\top\bPi_1\bX
    =  ||\bX_{12}||^2_{\Siginv}\\
    \tilX_{\bPi_1}^\top \bM^\top\tilX_{\hatPi_1}
    = & (\bX^\top\bPi_1^\top{\bSig}^{-1/2})({\bSig}^{1/2} \bPi_2 \hatPi_2^\top{\bSig}^{-1/2}){\bSig}^{-1/2}\hatPi_1\bX\\
    = & \bX^\top\bPi_1^\top\bPi_2 \hatPi_2^\top\Siginv\hatPi_1\bX
    = \langle \bX_{1}, \bX_{12}\rangle_{\Siginv}\\
    \tilX_{\hatPi_1}^\top\tilX_{\hatPi_1}
    = & \bX^\top\hatPi_1^\top{\bSig}^{-1/2}{\bSig}^{-1/2}\hatPi_1\bX =  ||\bX_1||^2_{\Siginv}
\end{aligned}
$$
Thus, using these notations, we have:
\begin{equation}
    \begin{split}
        \left\|\bT_{\bPi_1, \bPi_2} \right\|^2 = & \bT_{\bPi_1, \bPi_2}^\top \bT_{\bPi_1, \bPi_2}\\
        = & \tilX^\top_{\bPi_1}\bM^\top\left(\bI_n - \bP_{\hatPi_1}\right)\bM\tilX_{\bPi_1}\\
        = & \tilX^\top_{\bPi_1}\bM^\top \bM\tilX_{\bPi_1} - \tilX^\top_{\bPi_1}\bM^\top \frac{\widetilde{\bX}_{\hatPi_1}\widetilde{\bX}_{\hatPi_1}^\top}{\widetilde{\bX}_{\hatPi_1}^\top \widetilde{\bX}_{\hatPi_1}}M\tilX_{\bPi_1}\\
        = &||\bX_{12}||^2_{\Siginv} - \frac{\langle \bX_{1}, \bX_{12}\rangle^2_{\Siginv}}{||\bX_1||^2_{\Siginv}}\\
    \end{split}
\end{equation}
Now, using Lemma \ref{lem:T-bound} and equality of norms, we have:
\begin{eqnarray*}
    \left\|\bT_{\bPi_1, \bPi_2} \right\|^2 & \geq & \lambda_{\min}(\bSig^{-1})\left(\|\bX_{12}\|_2^2 - \frac{\langle \bX_{1}, \bX_{12}\rangle^2}{||\bX_1||^2}\right)\\
    & \geq & \dfrac{1}{\lambda_{\max}(\bSig)}\left(\|\bX_{12}\|_2^2 - \frac{\left|\langle \bX_{1}, \bX_{12}\rangle\right| \cdot \|\bX_{12}\|\cdot \|\bX_1\|}{||\bX_1||^2}\right)\\
    & = & \dfrac{1}{2\lambda_{\max}(\bSig)}\left(2\|\bX_{12}\|_2^2 -2\left|\langle \bX_{1}, \bX_{12}\rangle\right|\right)\\
    & = & \dfrac{1}{2\lambda_{\max}(\bSig)}\min \left\{\|\bX_{12} - \bX_1\|^2_2,\|\bX_{12} + \bX_1\|^2_2 \right\}
\end{eqnarray*}
Thus for some $t^\ast > 0$, we have:
\begin{eqnarray*}
    \Prob\left(\left\|\bT_{\bPi_1, \bPi_2} \right\|^2 < t^\ast\right) & \leq & \Prob\left(\dfrac{1}{2\lambda_{\max}(\bSig)}\min \left\{\|\bX_{12} - \bX_1\|^2_2,\|\bX_{12} + \bX_1\|^2_2 \right\} < t^\ast\right)\\
    & \leq & \Prob\left(\|\bX_{12} - \bX_1\|^2_2 < 2\lambda_{\max}(\bSig)t^\ast\right) + \Prob\left(\|\bX_{12} + \bX_1\|^2_2 < 2\lambda_{\max}(\bSig)t^\ast\right)
\end{eqnarray*}
As $\bX \sim N(\bzero,\bI_n)$, thus $\bX_1, \bX_{12} \sim N(\bzero,\bI_n)$, with $d_H(\hatPi_{12},\bI_n) = h_{12}$ and $\bX_{12} = \hatPi_{12}\bX_1$, following Lemma 4 of  \cite{pananjady2017linear}, and defining $t_0 \coloneqq \lambda_{\max}(\bSig)t^\ast$, we have for any $0 < t_0 < h_{12} \iff 0 < t^\ast < \nicefrac{h_{12}}{\lambda_{\max}(\bSig)}$, assuming $h_{12} \geq 3$,
\begin{eqnarray*}
    \Prob\left(\left\|\bT_{\bPi_1, \bPi_2} \right\|^2 < t^\ast\right) \leq 6 \exp\left(-\frac{h_{12}}{10}\left[\log \frac{h_{12}}{\lambda_{\max}(\bSig)t^\ast} + \dfrac{\lambda_{\max}(\bSig)t^\ast}{h_{12}} -1\right]\right)
\end{eqnarray*}
\end{proof}

\begin{proposition}\label{prop:choice-t}
    For $\snr = \Omega(K^\alpha)$ with $\alpha >1$ and $k_2,k_{12} \geq 2$, with $B \geq B_\ast(K,\alpha) = \frac{\alpha\log K}{K^\alpha - K}$, we have
    $$
    t = k_{12}\dfrac{h_2 + \log \snr}{\snr} \in \left[\dfrac{h_2}{\snr},h_{12}\right]
    $$
    With this choice of $t$, we have the following bounds for $K \geq 4$:
    \begin{itemize}
        \item $ P_i \coloneq \sum \exp(-c_i \ \snr \ t) = \bigo\left(K^4 \snr ^{-c_i}e^{-2c_iB}\right)$
        \item $ P_{22} \coloneq \sum\exp \left( - c_{22} \snr \, \min \left\{\dfrac{\snr}{h_2}\left(t^\ast - \dfrac{h_2}{\snr}\right)^2, \, \left(t^\ast - \dfrac{h_2}{\snr}\right) \right\}\right) = \bigo\left(K^4 \snr ^{-c_i}e^{-2c_iB}\right)$
        \item $P_0 \coloneq \sum\exp\left(-\frac{h_{12}}{10}\left[\log \frac{h_{12}}{t} + \dfrac{t}{h_{12}} -1\right]\right) = \bigo\left(K^2 e^{-c_0 B}\right)$
    \end{itemize}
    \vspace{2mm}
    where the sum is over all pairs $(\hatpi_1, \hatpi_2 ) \in \calQ_K$.
\end{proposition}
\begin{proof}
    Let us start off by mentioning the fact that the block size $K \geq 2$ and the number of blocks $B \geq 1$, and we have $h_2 = k_2B$ and $h_{12} = k_{12}B$ with $2 \leq k_2,k_{12}\leq K$.

    \textbf{Lower bound:} Let us first see why the choice of $t$ lies in the given interval. Notice that as $k_{12}\geq 2$ with $\frac{\log \snr}{\snr} > 0$ since $\snr > 1$, we have the following chain of inequality:
    \begin{eqnarray*}
        \dfrac{h_2}{\snr} \leq \dfrac{h_2 + \log \snr}{\snr} \leq k_{12}\dfrac{h_2 + \log \snr}{\snr}=t
    \end{eqnarray*}
\textbf{Upper bound:} Since $h_2=k_2B\le KB$ and the expression $\frac{k_{12}\big(k_2B+\log\snr\big)}{\snr}$ is increasing in $k_2$, the worst case is $k_2=K$. Thus, it suffices to show
$$
\frac{k_{12}\big(KB+\log\snr\big)}{\snr}\leq k_{12}B
\quad\Longleftrightarrow\quad
\frac{\log\snr}{\snr-K}\leq B,
$$
This is a valid lower bound for $B$ as  $\snr>K$ (true because $\alpha>1$ implies $K^{\alpha}>K$). Define the function 
$f(x)\coloneqq\frac{\log x}{x-K}$ for $x>K$.
A direct derivative computation gives
$$
f'(x)=\frac{1-\tfrac{K}{x}-\log x}{(x-K)^2}.
$$
Let $g(x)\coloneq\log x+\tfrac{K}{x}-1$. Then $g'(x)=\tfrac{1}{x}-\tfrac{K}{x^2}=\tfrac{x-K}{x^2}>0$ for $x>K$, and 
$g(k)=\log K>0$. Hence $g(x)>0$ for all $x>K$, i.e.\ $1-\tfrac{K}{x}-\log x<0$, so $f'(x)<0$. 
Therefore $f$ is strictly decreasing on $(K,\infty)$.

By monotonicity and the assumption $\snr\geq K^{\alpha}$,
$$
\frac{\log\snr}{\snr-K}=f(\snr)\le f(K^{\alpha})
=\frac{\log(K^{\alpha})}{K^{\alpha}-K}
=\frac{\alpha\log K}{K^{\alpha}-K}
\leq B.
$$
This proves $\frac{\log\snr}{\snr-K}\leq B$ and hence 
$\frac{k_{12}(h_2+\log\snr)}{\snr}\leq k_{12}B$ for all $2\leq k_{12},k_2\leq K$.

Because $f$ is decreasing, 
$$
\sup_{x\ge K^{\alpha}} f(x) = f(K^{\alpha})=\frac{\alpha\log K}{K^{\alpha}-K}\eqqcolon B_\ast(K,\alpha)
$$
Thus $B\geq B_\ast(K,\alpha)$ is also \emph{necessary} to guarantee the inequality uniformly for all $\snr\geq K^{\alpha}$.

\textbf{Bound on $P_i$:}
\begin{eqnarray*}
    P_i = \sum_{(\hatpi_1, \hatpi_2 ) \in \calQ_k}\exp\left(-c_i\snr \,t\right) 
    & \leq &\sum_{k_2=2}^K\sum_{k_{12}=2}^K K^{k_2+k_{12}}\exp\left(-c_i k_{12}  (k_2 B + \log \snr)\right) \\
    & = &\sum_{k_{12}=2}^{K} K^{k_{12}} e^{-c_i k_{12}\log\snr}
    \left(\sum_{k_2=2}^{K} e^{-k_2(c_i k_{12}B-\log K)}\right)\\
    & = & \sum_{k_{12}=2}^{K} K^{k_{12}} e^{-c_i k_{12}\log\snr}
\left(\sum_{k_2=2}^{K} a_{k_{12}}^{\,k_2}\right),
\quad a_{k_{12}}\coloneqq Ke^{-c_i k_{12}B}. 
\end{eqnarray*}
Assuming  {$B>\tfrac{\log K}{2c_i}$} so that $a_{k_{12}} \leq K e^{-2c_i B}<1$ for all $k_{12}\geq 2$. Then the inner finite geometric sum admits $\sum_{k_2=2}^{K} a_{k_{12}}^{k_2}
= \frac{a_{k_{12}}^{2}\,(1-a_{k_{12}}^{K-1})}{1-a_{k_{12}}}
\leq \frac{a_{k_{12}}^{2}}{1-a_{k_{12}}}$.
This implies that
$$
\sum_{k_{12}=2}^{K} K^{k_{12}} e^{-c_i k_{12}\log\snr}
\left(\sum_{k_2=2}^{K} e^{-k_2(c_i k_{12}B-\log K)}\right)
\leq
\sum_{k_{12}=2}^{K}
\frac{K^{k_{12}} e^{-c_i k_{12}\log\mathrm{SNR}} a_{k_{12}}^{2}}{1-a_{k_{12}}}.
$$
Since $a_{k_{12}}=K e^{-c_i k_{12}B}$ decreases in $k_{12}$,
$\frac{1}{1-a_{k_{12}}}\le \frac{1}{1-K e^{-2c_i B}}$. Also $a_{k_{12}}^{2}=K^{2}e^{-2c_i k_{12}B}$ implies that 
$$
\sum_{k_{12}=2}^{K}
\frac{K^{k_{12}} e^{-c_i k_{12}\log\mathrm{SNR}} a_{k_{12}}^{2}}{1-a_{k_{12}}}
\leq \frac{K^{2}}{1-K e^{-2c_i B}}
\sum_{k_{12}=2}^{K}
K^{k_{12}} e^{-c_i k_{12}\log\mathrm{SNR}}\, e^{-2c_i k_{12}B}.
$$
Set $q\coloneqq Ke^{-(c_i\log\snr+2c_i B)}
=\frac{K}{\snr^{c_i} e^{2c_i B}}$. Then the remaining sum is geometric: $\sum_{k_{12}=2}^{K} q^{k_{12}} \le \sum_{j=2}^{\infty} q^{j}=\frac{q^{2}}{1-q}$, provided $q<1 \iff c_i(\log\snr+2B)>\log K$. This implies that we have the final bound on $P_i$ as:
$$
P_i \leq \frac{K^2}{1- K e^{-2c_i B}}\cdot
\frac{\big(\tfrac{K}{\mathrm{SNR}^{c_i} e^{2c_i B}}\big)^{\!2}}
{1-\tfrac{K}{\mathrm{SNR}^{c_i} e^{2c_i B}}} = \bigo\left(K^4 \snr ^{-c_i}e^{-2c_iB}\right)
$$

\textbf{Bound on $P_{22}$:}
First observe that:
\begin{eqnarray*}
    R & \coloneq & \snr\left(t - \dfrac{h_2}{\snr}\right) \\
    & = & \snr\left(k_{12}\dfrac{h_2 + \log \snr}{\snr} - \dfrac{h_2}{\snr}\right)\\
    & = & (k_{12} - 1)k_2B + k_{12}\log \snr\\
    & \geq & (k_{12} - 1)(k_2B + \log \snr)\\
    & \geq & k_2B \geq 1\\
\end{eqnarray*}
This implies that: $\snr \cdot \min \left\{\frac{\snr}{h_2}\left(t - \dfrac{h_2}{\snr}\right)^2, \, \left(t - \dfrac{h_2}{\snr}\right) \right\} =  \min \left\{\dfrac{R^2}{k_2B}, R\right\} = R$. Hence, we have:
\begin{eqnarray*}
    P_{22} & = & \sum_{(\hatpi_1, \hatpi_2 ) \in \calQ_K}\exp \left( - c^\ast_{22} \snr \, \min \left\{\dfrac{\snr}{h_2}\left(t - \dfrac{h_2}{\snr}\right)^2, \, \left(t - \dfrac{h_2}{\snr}\right) \right\}\right)\\
    & \leq & \sum_{k_2 = 2}^K\sum_{k_{12}=2}^K K^{k_2 + k_{12}}\exp\left(-c^\ast_{22}R\right) \notag\\
    & = &\sum_{k_2 = 2}^K \sum_{k_{12}=2}^K K^{k_2 + k_{12}}
    \exp\!\left(-c^\ast_{22}\big[(k_{12}-1)k_2 B + k_{12}\log \snr\big]\right)\\
    & = & \sum_{k_{12}=2}^K \sum_{k_2 = 2}^K
    \Big(K^{k_{12}} e^{-c^\ast_{22} k_{12}\log\snr}\Big)
    \Big(K^{k_2} e^{-c^\ast_{22} (k_{12}-1)k_2 B}\Big)\\
    & = & \sum_{k_{12}=2}^K
    \Big(K^{k_{12}} e^{-c^\ast_{22} k_{12}\log\snr}\Big)
    \sum_{k_2=2}^K \Big(Ke^{-c^\ast_{22} (k_{12}-1)B}\Big)^{k_2}
\end{eqnarray*}

Assume $Ke^{-c^\ast_{22}B}<1 \iff B > \nicefrac{\log K}{c^\ast_{22}}$ so that for every \(k_{12}\ge2\) the ratio $a_{k_{12}}\coloneqq Ke^{-c^\ast_{22} (k_{12}-1)B}$ satisfies $0<a_{k_{12}}<1$. Then the inner finite geometric sum obeys
$$
\sum_{k_2=2}^K a_{k_{12}}^{k_2}
= \frac{a_{k_{12}}^{2}\bigl(1-a_{k_{12}}^{k-1}\bigr)}{1-a_{k_{12}}}
\leq \frac{a_{k_{12}}^{2}}{1-a_{k_{12}}}
\leq \frac{K^2 e^{-2c^\ast_{22} (k_{12}-1)B}}{1- K e^{-c^\ast_{22}B}}.
$$
Hence
$$
\begin{aligned}
P_{22}
&\le \frac{1}{1- Ke^{-c^\ast_{22}B}}
\sum_{k_{12}=2}^K
K^{k_{12}} e^{-c^\ast_{22} k_{12}\log\snr}\, \big(K^2 e^{-2c^\ast_{22} (k_{12}-1)B}\big)\\
&= \frac{K^2 e^{2c^\ast_{22}B}}{1- K e^{-c^\ast_{22}B}}
\sum_{k_{12}=2}^K
\Big(K\cdot\snr^{-c^\ast_{22}} e^{-2c^\ast_{22}B}\Big)^{k_{12}}.
\end{aligned}
$$
Set $q\coloneqq K\cdot\snr^{-c^\ast_{22}} e^{-2c^\ast_{22}B}$. Assuming {$B > \nicefrac{\log K}{c_{22}^\ast}$} implies $q< 1$.
Then $\sum_{k_{12}=2}^K q^{k_{12}} \leq\sum_{j=2}^{\infty} q^{j}
= \frac{q^2}{1-q}$, so
{
\small
$$
P_{22} \leq 
\frac{K^2 e^{2c^\ast_{22}B}}{1- K e^{-c^\ast_{22}B}}\cdot
\frac{\big(K\snr^{-c^\ast_{22}} e^{-2c^\ast_{22}B}\big)^{2}}{1- K\snr^{-c^\ast_{22}} e^{-2c^\ast_{22}B}}
=
\frac{K^{4}e^{-2c^\ast_{22}B}\snr^{-2c^\ast_{22}}}
{\bigl(1 - K e^{-c^\ast_{22}B}\bigr)\Bigl(1 - \tfrac{K}{\snr^{c^\ast_{22}}e^{2c^\ast_{22}B}}\Bigr)}
= \bigo\left(K^4 \snr ^{-c^\ast_{22}}e^{-2c^\ast_{22}B}\right)
$$}
\textbf{Bound on $P_0$:}
$$
P_0 \leq \sum_{k_2=2}^{K}\sum_{k_{12}=2}^{K}
K^{k_2+k_{12}}
\exp\left(
-\frac{k_{12}B}{10}\Big[\log\frac{k_{12}B}{t}+\frac{t}{k_{12}B}-1\Big]
\right)
$$
Let us define 
$$
y_{k_2}\coloneqq \frac{h_{12}}{t} = \frac{k_{12}B}{k_{12}(k_2B+\log\snr)/\snr}
=\frac{B\snr}{k_2B+\log\snr}
$$
Using this function $\phi(y) \coloneqq \log y+\frac{1}{y}-1\ge 0$ and the notation $r_{k_2} \coloneqq Ke^{-\frac{B}{10}\phi(y_{k_2})}$, we can then simplify the upper bound on $P_0$ as:
$$
P_0 \leq \sum_{k_2=2}^{K}K^{k_2}\sum_{k_{12}=2}^{K}
\Big(Ke^{-\frac{B}{10}\phi(y_{k_2})}\Big)^{k_{12}}
=\sum_{k_2=2}^{K}K^{k_2}
\frac{r_{k_2}^{2}\big(1-r_{k_2}^{K-1}\big)}{1-r_{k_2}}
$$
Next we can assume that there exists $\delta \in (0,1)$ such that for all $k_2 \in [2,K]$, we have $y_{k_2} \geq 1 + \delta$, which tells us that $\phi(y_{k_2})\ge c_\delta\coloneqq \phi(1+\delta)
=\log(1+\delta)+\frac{1}{1+\delta}-1>0$. This implies that $r_{k_2}\leq K e^{-\frac{B}{10}c_\delta}$ and thus $K^{k_2}r_{k_2}^2\leq \big(Ke^{-\frac{B}{5}c_\delta}\big)^{k_2}$. Using this chain on inequalities, we have:
\begin{equation}
    \begin{split}
        P_0 \leq  \sum_{k_2=2}^{K}K^{k_2}\,
        \frac{r_{k_2}^{\,2}\big(1-r_{k_2}^{\,k-1}\big)}{1-r_{k_2}}
        \leq & \dfrac{1}{1 - Ke^{-\frac{B}{10}c_\delta}}\sum_{k_2=2}^K K^{k_2}r_{k_2}^2\\ 
        \leq & \dfrac{1}{1 - Ke^{-\frac{B}{10}c_\delta}}\sum_{k_2=2}^k \big(Ke^{-\frac{B}{5}c_\delta}\big)^{k_2}\\
        = & \dfrac{1}{1 - Ke^{-\frac{B}{10}c_\delta}}\frac{\big(Ke^{-\tfrac{B}{5}c_\delta}\big)^2\big(1-\big(Ke^{-\tfrac{B}{5}c_\delta}\big)^{K-1}\big)}
        {1-Ke^{-\tfrac{B}{5}c_\delta}}\\
        \leq &  \dfrac{\big(Ke^{-\tfrac{B}{5}c_\delta}\big)^2}{\left(1 - Ke^{-\frac{B}{10}c_\delta}\right)\left(1-Ke^{-\tfrac{B}{5}c_\delta}\right)}\\
        = & \bigo\left(K^2 e^{-c_0 B}\right)
    \end{split}
\end{equation}
The last line follows from the fact that if
$B\geq(1+\varepsilon)\max\!\Big\{\frac{10}{c_\delta},\frac{5}{c_\delta}\Big\}\log K$, then $Ke^{-\frac{B}{10}c_\delta}\leq K^{-\varepsilon}, Ke^{-\frac{B}{5}c_\delta}\leq K^{-\varepsilon}$, and so the denominators are bounded away from zero. This completes the proof of this proposition.
\end{proof}

\subsection{Proof of Theorem \ref{thm:error-pi1-pi2}}\label{thm:error-pi1-pi2-proof}

\begin{proof}
    Here we follow the proof technique of \cite{pananjady2017linear}. Recall that the transformed DGP that we are working on is $\bPi_2 \bY  =  \bPi_1 \bX \beta + \bW^\ast$, where $\bW^\ast \sim \calN_n(\bzero, \bSig)$ for $\bSig = \sigma^2\bR(\phi) + \tau^2\bI_n$. Similar to thier Theorem 1, we define
    \begin{equation*}
        \left\{ \left( \hatPi_{1, \texttt{ML}}, \hatPi_{2, \texttt{ML}} \right) \neq \left( \bPi_1, \bPi_2 \right) \right\} = \bigcup_{(\hatPi_1, \hatPi_2 ) \in \calQ_n } \left\{ \Delta \left( (\hatPi_1, \hatPi_2 ), \left( \bPi_1, \bPi_2 \right) \right) \leq 0 \right\}
    \end{equation*}
    where $\calQ_n \coloneqq \mathcal{P}_n \times \mathcal{P}_n \setminus \left( \bPi_1, \bPi_2 \right)$ and
    \begin{equation*}
        \Delta \left( (\hatPi_1, \hatPi_2 ), \left( \bPi_1, \bPi_2 \right) \right) \coloneqq \left\|\bP_{\hatPi_1, X}^\perp \widetilde{\bY}_{\hatPi_2}\right\|_2^2 - \left\|\bP_{\bPi_1, X}^\perp \widetilde{\bY}_{\bPi_2}\right\|_2^2.
    \end{equation*}    
    One thing to mention here is that $\mathcal{P}_n$ is not a space of $n!$ many permutation matrix, but it is essentially a subspace which contains block diagonal permutation matrices with the same block diagonal elements od dimension $K \times K$, and thus we safely replace $\mathcal{P}_n$ by $\mathcal{P}_K$ where $K$ is the fixed size of each block and similarly $\calQ_n = \mathcal{P}_n \times \mathcal{P}_n \setminus \left( \bPi_1, \bPi_2 \right)$ by $\calQ_k \coloneqq \mathcal{P}_K \times \mathcal{P}_k \setminus \left( \bpi_1, \bpi_2 \right)$ since we have the one-to-one relationship $\bPi_u = \bdiag(\bpi_u)$ and $\hatPi_u = \bdiag(\hatpi_u)$ for $u = 1,2$.
    Therefore, we get the following bound on the probability
    \begin{eqnarray}\label{eq:hatpi_neq_pi_ub}
        \mathbb{P} \left( \left( \hatPi_{1, \texttt{ML}}, \hatPi_{2, \texttt{ML}} \right) \neq \left( \bPi_1, \bPi_2 \right) \right) & \leq & \sum_{(\hatpi_1, \hatpi_2 ) \in \calQ_k} \mathbb{P} \left( \Delta \left( (\hatPi_1, \hatPi_2 ), \left( \bPi_1, \bPi_2 \right) \right) \leq 0  \right) \notag \\
        & = & \sum_{(\hatpi_1, \hatpi_2 ) \in \calQ_k} \mathbb{P} \left( \left\|\bP_{\hatPi_1, X}^\perp \widetilde{\bY}_{\hatPi_2}\right\|_2^2 - \left\|\bP_{\bPi_1, X}^\perp \widetilde{\bY}_{\bPi_2}\right\|_2^2 \leq 0  \right).
    \end{eqnarray}
    To get a bound on RHS of \eqref{eq:hatpi_neq_pi_ub} we decompose $\left\|\bP_{\hatPi_1, X}^\perp \widetilde{\bY}_{\hatPi_2}\right\|_2^2 - \left\|\bP_{\bPi_1, X}^\perp \widetilde{\bY}_{\bPi_2}\right\|_2^2$ as follows:
    {\footnotesize
    \begin{equation}\label{eq:error-partition}
        \begin{split}
            & \left\|\bP_{\hatPi_1, X}^\perp \widetilde{\bY}_{\hatPi_2}\right\|_2^2 - \left\|\bP_{\bPi_1, X}^\perp \widetilde{\bY}_{\bPi_2}\right\|_2^2\\
            = & \underbrace{\left(\left\|\bP_{\hatPi_1, X}^\perp \widetilde{\bY}_{\hatPi_2}\right\|_2^2 - \left\|\bP_{\hatPi_1, X}^\perp {\bSig}^{-1/2} \hatPi_2 {\bPi}_2^\top \bW^\ast\right\|_2^2\right)}_{\bE_1} 
         - \underbrace{\left( \left\|\bP_{\bPi_1, X}^\perp \widetilde{\bY}_{\bPi_2}\right\|_2^2 - \left\|\bP_{\hatPi_1, X}^\perp {\bSig}^{-1/2} \hatPi_2 {\bPi}_2^\top \bW^\ast\right\|_2^2\right)}_{\bE_{2}}\\
         = & \bE_1 - \underbrace{\left[\underbrace{
        \left( \left\| \bP_{{\bPi}_1, X}^\perp {\bSig}^{-1/2} \bW^\ast \right\|_2^2 
        - \left\| \bP_{\hatPi_1, X}^\perp {\bSig}^{-1/2} \bW^\ast \right\|_2^2 \right)}_{\bE_{21}} 
        - \underbrace{
        \left( \left\| \bP_{\hatPi_1, X}^\perp {\bSig}^{-1/2} \bW^\ast \right\|_2^2 
        - \left\| \bP_{\hatPi_1, X}^\perp {\bSig}^{-1/2} \hatPi_2 {\bPi}_2^\top \bW^\ast \right\|_2^2 \right)}_{\bE_{22}}\right]}_{\bE_2}
        \end{split}
    \end{equation}}
Note that the last error partition in Equation \ref{eq:error-partition} is possible because:
$$
\bP_{\bPi_1, X}^\perp \widetilde{\bY}_{\bPi_2} = \bP_{\bPi_1, X}^\perp \left(\bPi_1 \bX \beta + \bW^\ast\right) =  \bP_{\bPi_1, X}^\perp\bW^\ast.
$$
Let us define a set $\calH_K := \left\{ \left( \hatpi_1, \hatpi_2 \right) \in \calQ_K \mid  \hatpi_1 \bpi_1^\top =  \hatpi_2 \bpi_2^\top \right\}$. We see that, $\left|\calH_K\right| = K!$, since, for any choice of $\hatpi_2$, we will get a unique $\hatpi_1$, such that  $(\hatpi_1, \hatpi_2 ) \in \calH_K$, and for such an element, we have $\hatPi_1 \bPi_1^\top =  \hatPi_2 \bPi_2^\top$, which implies $d_H\left(\hatPi_{12},\bI_n\right) = 0$ where $\hatPi_{12}\coloneq \hatPi_2\bPi_2^\top\bPi_1 \hatPi_1^\top = \bdiag\left(\hatpi_2\bpi_2^\top\bpi_1 \hatpi_1^\top\right)$ and so we can define $\hatpi_{12} \coloneq \hatpi_2\bpi_2^\top\bpi_1 \hatpi_1^\top$. 

Next, note the fact that there is a bijection $(\hatpi_1,\hatpi_2) \mapsto (\hatpi_{12},\hatpi_2)$ as $\bpi_1,\bpi_2$ are assumed to be fixed constant permutation matrices, and thus $(\hatpi_{1},\hatpi_2) \in \calQ_K$ is equivalent to $(\hatpi_{12},\hatpi_2) \in \mathcal{P}_K \times \mathcal{P}_K \setminus \left( \bI_K, \bpi_2 \right) \eqqcolon \calQ^\ast_K$. Also, note that under $(\hatpi_{1},\hatpi_2) \in \calH_K$, we have $\hatpi_{12} = \bI_K$, and thus $(\hatpi_{1},\hatpi_2) \in \calH_K$ is equivalent to $(\hatpi_{12}, \hatpi_2) \in \calH^\ast_K \coloneq \left\{(\hatpi_{12},\hatpi_2) \in \calQ_K^\ast : \hatpi_{12} = \bI_K\right\}$ with $\left|\calH^\ast_K\right| = K!$. Now we break down the analysis into two parts for (1) $ \left( \hatpi_1, \hatpi_2 \right) \in \calH_K^c$ and (2) $ \left( \hatpi_1, \hatpi_2 \right) \in \calH_K$.

\textbf{Case 1 :} $ \left( \hatpi_1, \hatpi_2 \right) \in {\calH^\ast_K}^c$. First we find a bound on $\Prob\left[\Delta \left( (\hatPi_1, \hatPi_2 ), \left( \bPi_1, \bPi_2 \right) \right) \leq 0\mid (\hatpi_1, \hatpi_2 ) \in \calH_K^c \right]$. One thing to note here is that, we are not really considering the $(\hatpi_1, \hatpi_2 )$ as random variables, and just considering them as elements in $\calQ_K$. For $t_0 \geq 0$ we obtain:
    {
    \footnotesize
    \begin{eqnarray*}
        \Prob\left[\Delta \left( (\hatPi_1, \hatPi_2), \left( \bPi_1, \bPi_2 \right) \right) \leq 0 \mid (\hatpi_1, \hatpi_2 ) \in \calH_K^c \right]
        & = & \Prob\left(\bE_1 - (\bE_{21} - \bE_{22}) \leq 0 \right) \\
        & \leq & \Prob\left(\bE_1 - (\bE_{21} - \bE_{22}) < \beta^2\delta^\ast \right) \\
        & = & \Prob\left(\bE_1 - (\bE_{21} - \bE_{22}) < \beta^2\delta^\ast , \left\|\bT_{\bPi_1, \bPi_2} \right\|^2 >  t^\ast\right) \\
        && + \Prob\left(\bE_1 - (\bE_{21} - \bE_{22}) < \beta^2\delta^\ast , \left\|\bT_{\bPi_1, \bPi_2} \right\|^2 <
         t^\ast\right)
    \end{eqnarray*}}
    where $\bT_{\bPi_1, \bPi_2} \coloneqq \bP_{\hatPi_1, X}^\perp {\bSig}^{-1/2} \hatPi_2 \bPi_2^\top \bPi_1 \bX$ and $\delta^\ast = c\left\|\bT_{\bPi_1, \bPi_2} \right\|^2$ for some constant $c>0$. Notice that we can simplify  $\bT_{\bPi_1, \bPi_2} = \bP_{\hatPi_1, X}^\perp {\bSig}^{-1/2}\hatPi_{12}\hatPi_1\bX$. Note that for $\left( \hatpi_1, \hatpi_2 \right) \in \calH_K^c$, we have $d_H\left(\hatPi_{12},\bI_n\right)>0$ and thus $\bT_{\bPi_1, \bPi_2} \neq \bzero$ because if the Hamming distance is 0 which implies $\hatPi_{12} = \bI_n$, then since $\bP_{\hatPi_1, X}^\perp$ denotes the projection matrix on the column space orthogonal to $\bSig^{-1/2}\hatPi_1\bX$, we would have $\bT_{\bPi_1, \bPi_2}=\bzero$. Thus for $\left( \hatpi_1, \hatpi_2 \right) \in \calH_K^c$, the choice of $\delta^\ast$ is valid as it will be positive a.s. Hence, the RHS of the above inequality can be further bounded as follows:
    {
    \footnotesize
    \begin{eqnarray*}
        \Prob\left(\Delta \left( (\hatPi_1, \hatPi_2 ), \left( \bPi_1, \bPi_2 \right) \right) \leq 0 \mid (\hatpi_1, \hatpi_2 ) \in \calH_K^c \right)
        & \leq & \Prob\left(\bE_1 - (\bE_{21} - \bE_{22}) < \beta^2\delta^\ast \ {\Big|} \ \left\|\bT_{\bPi_1, \bPi_2} \right\|^2 > t^\ast\right) \\
        && + \Prob\left(\left\|\bT_{\bPi_1, \bPi_2} \right\|^2 <
        \ t^\ast\right)\\
        & \leq & \Prob\left(\bE_1 < 2\beta^2\delta^\ast\ {\Big|} \ \left\|\bT_{\bPi_1, \bPi_2} \right\|^2 > t^\ast\right)\\
        && + \Prob\left(\bE_{21} - \bE_{22} > \beta^2\delta^\ast \ {\Big|} \ \left\|\bT_{\bPi_1, \bPi_2} \right\|^2 > t^\ast\right)\\
        && + \Prob\left(\left\|\bT_{\bPi_1, \bPi_2} \right\|^2 <
        \ t^\ast\right) \\
        & \leq & \Prob\left(\bE_1 < 2\beta^2\delta^\ast\ {\Big|} \ \left\|\bT_{\bPi_1, \bPi_2} \right\|^2 > t^\ast\right)\\
        && + \Prob\left(|\bE_{21}| > {\beta^2\delta^\ast}/{2} \ {\Big|} \ \left\|\bT_{\bPi_1, \bPi_2} \right\|^2 > t^\ast\right)\\
        && + \Prob\left(|\bE_{22}| > {\beta^2\delta^\ast}/{2} \ {\Big|} \ \left\|\bT_{\bPi_1, \bPi_2} \right\|^2 > t^\ast\right)\\
        && + \Prob\left(\left\|\bT_{\bPi_1, \bPi_2} \right\|^2 <
        t^\ast\right).
    \end{eqnarray*}}
    Let us denote $d_H(\hatpi_2,\bpi_2) = k_2$ which implies $d_H\left(\hatPi_2,\bPi_2\right)=d_H\left(\hatPi_2\bPi_2^\top,\bI_n\right) \coloneqq h_2 = k_2B$.
    Similarly, for $d_H\left(\hatpi_1 \bpi_1^\top ,  \hatpi_2 \bpi_2^\top\right) = k_{12}$, we have $d_H\left(\hatPi_{12},\bI_n\right) \coloneq h_{12}= k_{12}B$. Observe the fact that $2 \leq k_2 \leq K$, and under this case we have $2 \leq k_{12}\leq K$.
    
    Now we set $t^\ast = \nicefrac{t}{\lambda_{\max}(\bSig)}$, then for a choice of $t \in \left[\dfrac{h_2}{\snr},h_{12}\right]$, from Propositions \ref{prop:E1-dist-given-X}, \ref{prop:E21-dist-given-X}, \ref{prop:E22-dist-cond-X},  and \ref{prop:prob-T-less-t},  with the above inequality we get
    \begin{eqnarray*}
        &&\Prob\left[\Delta \left( (\hatPi_1, \hatPi_2 ), \left( \bPi_1, \bPi_2 \right) \right) \leq 0 \mid (\hatpi_1, \hatpi_2 ) \in \calH_K^c\right]\\
        & \leq & \exp\left(-c_1^\ast\dfrac{\beta^2}{\kappa(\bSig)}\cdot \dfrac{t}{\lambda_{\max}(\bSig)}\right)  \\
        && + \left(-c^\ast_{21}\beta^2\cdot \dfrac{t}{\lambda_{\max}(\bSig)}\right)\\
        && +  \exp \left( - c^\ast_{22} \, \snr \, \min \left\{\frac{\snr}{h_2}\left(t - \dfrac{h_2}{\snr}\right)^2, \, \left(t - \dfrac{h_2}{\snr}\right) \right\}\right)\\
        && + 6 \exp\left(-\frac{h_{12}}{10}\left[\log \frac{h_{12}}{t} + \dfrac{t}{h_{12}} -1\right]\right)\\
        & = & \exp\left(-c^\ast_1\snr \ t\right)\\
        && + \exp\left(-c^\ast_{21} \kappa(\bSig) \cdot \snr \ t\right)\\
        && +  \exp \left( - c^\ast_{22} \, \snr \, \min \left\{\frac{\snr}{h_2}\left(t - \dfrac{h_2}{\snr}\right)^2, \, \left(t - \dfrac{h_2}{\snr}\right) \right\}\right)\\
        && + 6 \exp\left(-\frac{h_{12}}{10}\left[\log \frac{h_{12}}{t} + \dfrac{t}{h_{12}} -1\right]\right).
    \end{eqnarray*}

A valid choice for $t$ for sufficiently large $B$ will be $t = k_{12}\dfrac{h_2 + \log \snr}{\snr}$ where $h_2 = k_2B$. For more details refer to Proposition \ref{prop:choice-t}. 

\textbf{Case 2 :} $(\hatpi_1, \hatpi_2 ) \in \calH^\ast_K$. Under this case, we have $\bE_1 = 0$. Thus, we have:
$$
\begin{aligned}
\left\{\Delta \left( (\hatPi_1, \hatPi_2 ), \left( \bPi_1, \bPi_2 \right) \right) \leq 0 \mid (\hatpi_1, \hatpi_2 ) \in \calH_K \right\} & = \{\bE_{22} - \bE_{21} >0 \}\\
& \subseteq \left\{\left|\bE_{22}\right| > \frac {t'}2\right\} \cup \left\{\left|\bE_{21}\right| > \frac {t'}2\right\}
\end{aligned}
$$
Now, if we see the proof of Proposition \ref{prop:E21-dist-given-X}, then we can see that for some $t'>0$, we have $ \Prob(|\bE_{21}|> \nicefrac{t'}{2})\leq c'\exp\left(-\frac{c_{21}}{2\sqrt{2}}t'\right)$ independent of the distribution of $\bX$. Similarly following the proof of Proposition \ref{prop:E22-dist-cond-X}, gives us the following unconditional concentration for some $t^\ast > \sqrt{2}h_2\kappa(\bSig)$:
\begin{equation*}
        \Prob(|\bE_{22}|> t^\ast)\leq \exp \left( - c_{22} \, \min \left\{ \frac{(t^\ast-\sqrt{2}h_2\kappa(\bSig))^2}{2h_2\kappa^2({\bSig})}, \, \frac{t^\ast-\sqrt{2}h_2\kappa(\bSig)}{\sqrt{2}\kappa(\bSig)} \right\} \right)
\end{equation*}

Let us choose $ t' = 4\sqrt{2}h_2 \kappa(\boldsymbol{\Sigma})$, which gives us $t'/2 - \sqrt{2}h_2\kappa(\bSig) = 2\sqrt{2}h_2 \kappa(\boldsymbol{\Sigma}) - \sqrt{2}h_2\kappa(\bSig) = \sqrt{2}h_2\kappa(\bSig) $. Then as $k_2 \geq 2$, the following concentration inequality holds:
\begin{equation}
    \mathbb{P}\left( |\bE_{22}| > \frac{t'}{2} \right)
    \leq  \exp\left( - c_{22}\min\left\{ h_2, h_2 \right\} \right) = \exp\left( -c_{22}k_2B \right).
\end{equation}

Similarly for this choice of $t'>0$, we have:
\begin{equation}
    \mathbb{P}\left( |\bE_{21}| > \frac{t'}{2} \right)
    \leq  \exp\left( -c^\ast_{21}k_2B \right).
\end{equation}
for a constant $c_{21}^\ast > 0$. Thus for the specific $t'$ combining the above two concentrations, we can say that:
\begin{eqnarray*}
    \Prob\left[\Delta \left( (\hatPi_1, \hatPi_2 ), \left( \bPi_1, \bPi_2 \right) \right) \leq 0 \mid (\hatpi_1, \hatpi_2 ) \in \calH_K \right] & \leq & \Prob\left(\left|\bE_{22}\right| > \frac {t'}2\right) + \Prob\left(\left|\bE_{21}\right| > \frac {t'}2\right)\\
    & \leq & \exp\left( -c^\ast k_2B \right).
\end{eqnarray*}
where $c^\ast$ is dependent on $c_{21}^\ast, c_{22}$. Also note that:
\begin{equation}\label{eq:hk-star-bound}
    \sum_{k_2=2}^K K^{k_2}  e^{-c^\ast k_2B} = \sum_{k_2=2}^K \left(Ke^{-c^\ast B}\right)^{k_2} \leq \frac{\big(K e^{-c^\ast B}\big)^{2}}{1-K e^{-c^\ast B}} = \bigo\left(K^2 e^{-2c^\ast B}\right)
\end{equation}
where the last equality follows given $B>\tfrac{\log k}{c^\ast}$ which will imply $1-K e^{-c^\ast B}\ge 1-K^{-\varepsilon}$.

Thus, using the equivalence between the sets $\calQ_K$ and $\calQ^\ast_K$, using Proposition \ref{prop:choice-t} and Equation \ref{eq:hk-star-bound}, we can say that:
\begin{eqnarray*}
    \mathbb{P} \left( \left( \hatPi_{1, \texttt{ML}}, \hatPi_{2, \texttt{ML}} \right) \neq \left( \bPi_1, \bPi_2 \right) \right) 
    & \leq & \sum_{(\hatpi_1, \hatpi_2 ) \in \calQ_K} \mathbb{P} \left( \Delta \left( (\hatPi_1, \hatPi_2 ), \left( \bPi_1, \bPi_2 \right) \right) \leq 0  \right) \notag \\
    & = & \sum_{(\hatpi_{12}, \hatpi_2 ) \in \calQ^\ast_K} \mathbb{P} \left( \Delta \left( (\hatPi_1, \hatPi_2 ), \left( \bPi_1, \bPi_2 \right) \right) \leq 0  \right) \notag \\
    & = & \sum_{(\hatpi_1, \hatpi_2 ) \in \calH^\ast_K} \mathbb{P} \left( \cdots \leq 0  \right) \\
    && + \sum_{(\hatpi_{12}, \hatpi_2 ) \in \calH^{*^c}_K} \mathbb{P} \left( \cdots  \leq 0 \right) \notag \\
    & = & \bigo\left(K^{2} e^{-2c^\ast B}\right) \\
    && + \bigo\left(K^4 \snr ^{-c_1}e^{-2c_1B}\right) + \bigo\left(K^2 e^{-c_0 B}\right)\\
    & = & \bigo\left(K^4 \snr ^{-c^\ast_1}e^{-2c^\ast_1B}\right) + \bigo\left(K^2 e^{-c^\ast_2 B}\right)
\end{eqnarray*}
for appropriate choices of $c_1^\ast,c_2^\ast>0$. This completes the proof of the theorem.
\end{proof}
\newpage

\subsection{Proof of Theorem \ref{thm:error-betahat}}\label{thm:error-betahat-proof}

\begin{proof}
    First recall the DGP that, $\widetilde{\bY}_{{\bPi}_2} = \widetilde{\bX}_{{\bPi}_1}\beta + \beps$, and observe that  $\widetilde{\bY}_{\hatPi_2}= \bM_2 \widetilde{\bY}_{{\bPi}_2}, \widetilde{\bX}_{\hatPi_1}= \bM_1 \widetilde{\bX}_{{\bPi}_1}$, where $\bM_2= \bSigma^{-{1}/{2}}\hatPi_2 \bPi_2^\top \bSigma^{{1}/{2}}$ and $\bM_1 = \bSigma^{-{1}/{2}}\hatPi_1 \bPi_1^\top \bSigma^{{1}/{2}}$. 
Let us look at the bias in estimating $\beta$:
{
\footnotesize
\begin{eqnarray}\label{eq:betahat-bias}
        \hat{\beta} & = &\left(\widetilde{\bX}_{\hatPi_1}^\top  \widetilde{\bX}_{\hatPi_1}\right)^{-1} \widetilde{\bX}_{\hatPi_1}^\top \widetilde{\bY}_{\hatPi_2} \notag\\
    & = &\left(\widetilde{\bX}_{\hatPi_1}^\top  \widetilde{\bX}_{\hatPi_1}\right)^{-1} \widetilde{\bX}_{\hatPi_1}^\top \bM_2 \widetilde{\bY}_{{\bPi}_2} \notag\\
    & = & \left(\widetilde{\bX}_{\hatPi_1}^\top  \widetilde{\bX}_{\hatPi_1}\right)^{-1} \widetilde{\bX}_{\hatPi_1}^\top \left(\bM_2 \widetilde{\bX}_{{\bPi}_1}\beta + \bM_2 \beps\right) \notag\\ 
    & = & \left(\widetilde{\bX}_{\hatPi_1}^\top  \widetilde{\bX}_{\hatPi_1}\right)^{-1} \widetilde{\bX}_{\hatPi_1}^\top \bM_1 \widetilde{\bX}_{{\bPi}_1} \beta + \left(\widetilde{\bX}_{\hatPi_1}^\top  \widetilde{\bX}_{\hatPi_1}\right)^{-1} \widetilde{\bX}_{\hatPi_1}^\top (\bM_2 -\bM_1)\widetilde{\bX}_{{\bPi}_1}\beta \notag\\
    && + \left(\widetilde{\bX}_{\hatPi_1}^\top  \widetilde{\bX}_{\hatPi_1}\right)^{-1} \widetilde{\bX}_{\hatPi_1}^\top \bM_2 \beps\notag\\
    & = & \beta + \beta \left(\widetilde{\bX}_{\hatPi_1}^\top  \widetilde{\bX}_{\hatPi_1}\right)^{-1} \widetilde{\bX}_{\hatPi_1}^\top (\bM_2 -\bM_1)\widetilde{\bX}_{{\bPi}_1} + \left(\widetilde{\bX}_{\hatPi_1}^\top  \widetilde{\bX}_{\hatPi_1}\right)^{-1}
    \widetilde{\bX}_{\hatPi_1}^\top \bM_2 \beps\notag\\
    \implies \hat{\beta} - \beta & = & \beta \left(\widetilde{\bX}_{\hatPi_1}^\top  \widetilde{\bX}_{\hatPi_1}\right)^{-1} \widetilde{\bX}_{\hatPi_1}^\top \left(\bM_2\bM_1^{-1} -\bI\right)\widetilde{\bX}_{\hatPi_1} + \left(\widetilde{\bX}_{\hatPi_1}^\top  \widetilde{\bX}_{\hatPi_1}\right)^{-1} \widetilde{\bX}_{\hatPi_1}^\top \bM_2 \beps\notag\\
& = & \beta \left(\widetilde{\bX}_{\hatPi_1}^\top  \widetilde{\bX}_{\hatPi_1}\right)^{-1} \widetilde{\bX}_{\hatPi_1}^\top \left(\bSigma^{-{1}/{2}}\hatPi_{12} \bSigma^{{1}/{2}} -\bI\right)\widetilde{\bX}_{\hatPi_1} + \left(\widetilde{\bX}_{\hatPi_1}^\top  \widetilde{\bX}_{\hatPi_1}\right)^{-1} \widetilde{\bX}_{\hatPi_1}^\top \bM_2 \beps\notag\\
& = & \beta \left(\bX^\top {\hatPi_1}^\top \bSigma^{-1} {\hatPi_1} \bX\right)^{-1} \bX^\top {\hatPi_1}^\top \bSigma^{-1} \left(\hatPi_{12}- \bI\right) {\hatPi_1} \bX + \left(\widetilde{\bX}_{\hatPi_1}^\top  \widetilde{\bX}_{\hatPi_1}\right)^{-1} \widetilde{\bX}_{\hatPi_1}^\top \bM_2 \beps\notag\\
\implies \norm{\hat{\beta} - \beta}_2 & \leq & \norm{\beta}_2\cdot \norm{\dfrac{\bX^\top {\hatPi_1}^\top \bSig^{-1} \left(\hatPi_{12} - \bI\right) {\hatPi_1} \bX}{\bX^\top {\hatPi_1}^\top \bSig^{-1} {\hatPi_1} \bX}}_2 + \norm{\dfrac{\widetilde{\bX}_{\hatPi_1}^\top \bM_2 \beps}{\widetilde{\bX}_{\hatPi_1}^\top  \widetilde{\bX}_{\hatPi_1}}}_2\notag\\
& \leq & \norm{\beta}_2\cdot \dfrac{\norm{\bSigma^{-1} \left(\hatPi_{12}- \bI\right)}_2}{\lambda_{\min}(\Siginv)} + \norm{\dfrac{\widetilde{\bX}_{\hatPi_1}^\top \bM_2 \beps}{\widetilde{\bX}_{\hatPi_1}^\top  \widetilde{\bX}_{\hatPi_1}}}_2\notag\\
& \leq & \norm{\beta}_2\cdot \kappa(\bSig)\norm{\hatPi_{12}- \bI}_2 + \dfrac{1}{\sqrt{n}}\norm{\dfrac{\widetilde{\bX}_{\hatPi_1}^\top \bM_2 \beps/\sqrt{n}}{\widetilde{\bX}_{\hatPi_1}^\top  \widetilde{\bX}_{\hatPi_1}/n}}_2
\end{eqnarray}}

Recall that $\hatPi_{12} = \hatPi_2 \bPi_2^\top \bPi_1 \hatPi_1^\top$. We are interested in the event $\left\{\hatPi_{12} \ne \bI\right\}$. Observe that the following is true:
\begin{equation}
\label{eq:selection}
\left\{\hatPi_{12} \ne \bI\right\} \subseteq \underset{(\hatPi_1, \hatPi_2) \in \calQ_K:\hatPi_{12}\neq \bI}{\bigcup}\left\{\left\| \bP_{\hatPi_1, X}^\perp \widetilde{\bY}_{{\hatPi_2}}\right\|_2^2 
<
\left\| \bP_{\hatPi_1, X}^\perp \tilY_{\hatPi_1 \bPi_1^\top \bPi_2} \right\|_2^2\right\},
\end{equation}
So, $$\Prob\left(\hatPi_{12} \ne \bI\right) = \sum_{(\hatPi_1, \hatPi_2) \in \calQ_K:\hatPi_{12}\neq \bI} \Prob\left(\left\| \bP_{\hatPi_1, X}^\perp \tilY_{{\hatPi_2}}\right\|_2^2 
< \left\| \bP_{\hatPi_1, X}^\perp \tilY_{\hatPi_1 \bPi_1^\top \bPi_2} \right\|_2^2\right).$$
Now define, $\bM_1 = \bSigma^{-1/2}\hatPi_1 \bPi_1^\top \bSigma^{1/2}$,  $\bM_2 = \bSigma^{-1/2}\hatPi_2 \bPi_2^\top \bSigma^{1/2}$, and observe that $\bM_1 =\bSigma^{-1/2} \hatPi_{12}^\top \bSigma^{1/2} \bM_2$ and $\tilY_{\hatPi_2} = \tilX_{\hatPi_2\bPi_2^\top \bPi_1} \beta + \bM_2 \beps$ with $\tilY_{\hatPi_1 \bPi_1^\top \bPi_2} = \tilX_{\hatPi_1} \beta + \bM_1 \beps.$

Then, 
\begin{eqnarray*}
\left\| \bP_{\hatPi_1, X}^\perp \tilY_{\hatPi_2}\right\|_2^2 -
\left\| \bP_{\hatPi_1, X}^\perp \tilY_{\hatPi_1 \bPi_1^\top \bPi_2} \right\|_2^2 & = &\norm{\bP_{\hatPi_1, X}^\perp \tilX_{\hatPi_2\bPi_2^\top \bPi_1} \beta}_2^2 + 2 \beta \left\langle \bP_{\hatPi_1, X}^\perp \tilX_{\hatPi_2\bPi_2^\top \bPi_1} , \bM_2 \beps\right\rangle \\
&& -\left(\norm{ \bP_{\hatPi_1, X}^\perp \bM_1 \beps}_2^2 - \norm{ \bP_{\hatPi_1, X}^\perp \bM_2 \beps}_2^2\right) \\
& = &\underbrace{\norm{\bP_{\hatPi_1, X}^\perp \tilX_{\hatPi_2\bPi_2^\top \bPi_1} \beta}_2^2 + 2 \beta \left\langle \bP_{\hatPi_1, X}^\perp \tilX_{\hatPi_2\bPi_2^\top \bPi_1} , \bM_2 \beps\right\rangle}_{\bE_1} \\
&& - \underbrace{\beps^\top \left(\bM_1^\top \bP_{\hatPi_1, X}^\perp  \bM_1 - \bM_2^\top \bP_{\hatPi_1, X}^\perp  \bM_2\right)\beps}_{\bE_4} 
\end{eqnarray*}

Let us define, $\bA_4 =\bM_1^\top \bP_{\hatPi_1, X}^\perp  \bM_1 - \bM_2^\top \bP_{\hatPi_1, X}^\perp  \bM_2 = \left(\bM_1^\top \bM_1 - \bM_2^\top \bM_2\right)- \left(\bM_1^\top \bu\bu^\top \bM_1 - \bM_2^\top \bu\bu^\top \bM_2\right)$, for some $\bu \in \R^n$ with $\norm{\bu}_2=1$. By Von-Neumann Lemma \ref{lem:von-neuman}, we can say:
\begin{align}
    &\left( \tr\left(\bM_1^\top \bM_1 \right) - \tr\left(\bM_2^\top \bM_2\right) \right) 
    - \left( \lambda_{\max}\left(\bM_1 \bM_1^\top\right) - \lambda_{\min}\left(\bM_2 \bM_2^\top\right) \right) 
    \leq \tr(\bA_4) \notag \\
    &\leq \left( \tr\left(\bM_1^\top \bM_1 \right) - \tr\left(\bM_2^\top \bM_2\right) \right) 
    - \left( \lambda_{\min}\left(\bM_1 \bM_1^\top\right) - \lambda_{\max}\left(\bM_2 \bM_2^\top\right) \right)
\end{align}
This means that
\begin{eqnarray*}
    \tr(\bA_4) \in \left( \tr\left(\bM_1^\top \bM_1 \right) - \tr\left(\bM_2^\top \bM_2\right) \right) 
    \pm \left( \kappa(\bSigma) - \frac{1}{\kappa(\bSigma)} \right)
\end{eqnarray*}

$\bA_4$ can be written as: 
\begin{align}
\bA_4 &= \bM_1^\top \bP_{\hatPi_1, X}^\perp \bM_1 - \bM_2^\top \bP_{\hatPi_1, X}^\perp \bM_2 \notag \\
&= \bM_2^\top \left( \bSig^{1/2} \hatPi_{12} \bSig^{-1/2} \bP_{\hatPi_1, X}^\perp \bSig^{-1/2} \hatPi_{12}^\top \bSig^{1/2} - \bP_{\hatPi_1, X}^\perp \right) \bM_2 \notag \\
&= \bM_2^\top \bSig^{1/2} \left( \hatPi_{12} \bSig^{-1/2} \bP_{\hatPi_1, X}^\perp \bSig^{-1/2} \hatPi_{12}^\top - \bSig^{-1/2} \bP_{\hatPi_1, X}^\perp \bSig^{-1/2} \right) \bSig^{1/2} \bM_2 \notag\\
&= \bM_2^\top \bSigma^{1/2} \left( \hatPi_{12} \widetilde{\bSigma^{-1}} \hatPi_{12}^\top - \widetilde{\bSigma^{-1}} \right)\bSigma^{1/2} \bM_2 \notag\\
& = \bM_2^\top \bSigma^{1/2} \underbrace{\left( \left(\hatPi_{12} - \bI\right) \widetilde{\bSig^{-1}}  \hatPi_{12}^\top + \widetilde{\bSig^{-1}} \left(\hatPi_{12}^\top - \bI\right) \right)}_{\eqqcolon \bA_6;\textrm{ rank}(\bA_6)\leq 2h_{12}} \bSig^{1/2} \bM_2 \notag
\end{align}
where we define:
\begin{equation}
\widetilde{\bSigma^{-1}} := \bSig^{-1/2} \bP_{\hatPi_1, X}^\perp \bSig^{-1/2}.
\end{equation}
Hence, $\textrm{rank}(\bA_6) \leq 2h_{12}$, $\norm{\bA_6}_2 \leq \frac{2}{\lambda_{\min}(\bSigma)}$. We also have:
$$\norm{\bA_6}_F \leq \norm{(\hatPi_{12} - I)}_F \norm{\widetilde{\bSigma^{-1}}  \hatPi_{12}^\top}_2 + \norm{\widetilde{\bSigma^{-1}}}_2 \norm{(\hatPi_{12}^\top - I)}_F= 2\frac{\sqrt{2h_{12}}}{\lambda_{\min}(\bSigma)}$$
Using these results, we have $\textrm{rank}(\bA_4) \leq 2h_{12}$ and $\norm{\bA_4}_2 \leq 2\kappa(\bSigma)$.  
Finally, again by Lemma \ref{lem:von-neuman} we have:
\begin{eqnarray*}
    \tr(\bA_4)& = &\tr\left(\bA_6 \bSig^{1/2} \bM_2 \bM_2^\top \bSig^{1/2}\right)\\
    \implies |\tr(\bA_4)| & \leq & \sigma_{\max}\left({\bSig^{1/2} \bM_2 \bM_2^\top \bSigma^{1/2}}\right)\sum_{i=1}^{\textrm{rank}(\bA_6)}\sigma_i(\bA_6) \\
    & \leq & \lambda_{\max}({\bSigma})\sqrt{2h_{12}} \norm{\bA_6}_F \\
    & \leq & 4h_{12} \kappa(\bSigma)
\end{eqnarray*}
\newpage

Following the proof of Proposition \ref{prop:E22-dist-cond-X}, we have:
\begin{eqnarray*}
    \mathbb{P} \left(  |\bE_4|  > t^\ast \right)
    & \leq &  \exp \left( - \dfrac{c_{4}}{2} \, \min \left\{ \frac{(t^\ast-4h_{12} \kappa(\bSig))^2}{8h_{12}\kappa^2({\bSig})}, \, \frac{t^\ast -4h_{12} \kappa(\bSigma)}{2\kappa({\bSig})} \right\} \right).
\end{eqnarray*}
Thus plugging in $t^\ast = \beta^2\delta_0$ with $\delta_0 \coloneq 4\left\|\bT_{\bPi_1, \bPi_2} \right\|^2 $, and conditioning on $\left\|\bT_{\bPi_1, \bPi_2} \right\|^2> \nicefrac{t}{\lambda_{\max}(\bSig)}$ for some $t>0$, we have:
\begin{eqnarray*}
    t^\ast - 4h_{12}\kappa(\bSig)
    & = & \beta^2 \cdot 4\left\|\bT_{\bPi_1, \bPi_2} \right\|^2 - 4h_{12}\kappa(\bSig)\\
    & > & 4\left[\beta^2 \dfrac{t}{\lambda_{\max}(\bSig)} - h_{12}\kappa(\bSig)\right]\\
    & = & 4\left[{\kappa(\bSig)\ \snr}\ t - h_{12}\kappa(\bSig)\right]\\
    & = & 4\kappa(\bSig)\ \snr\left[t - \dfrac{h_{12}}{\snr}\right]\\
\end{eqnarray*}
Hence for $t \geq \frac{h_{12}}{\snr}$, we have the following conditional probability:
\begin{eqnarray*}
    \Prob\left(|\bE_4| > \beta^2\delta_0\right)\leq  \exp \left( - c_{4} \, \snr \, \min \left\{\frac{\snr}{h_{12}}\left(t - \dfrac{h_{12}}{\snr}\right)^2, \, \left(t - \dfrac{h_{12}}{\snr}\right) \right\}\right)
\end{eqnarray*}
Thus, we have:
\begin{eqnarray*}
    \Prob\left(\left\| \bP_{\hatPi_1, X}^\perp \tilY_{{\hatPi_2}}\right\|_2^2 
< \left\| \bP_{\hatPi_1, X}^\perp \tilY_{\hatPi_1 \bPi_1^\top \bPi_2} \right\|_2^2\right) & = & \Prob\left(\left\| \bP_{\hatPi_1, X}^\perp \tilY_{{\hatPi_2}}\right\|_2^2 
- \left\| \bP_{\hatPi_1, X}^\perp \tilY_{\hatPi_1 \bPi_1^\top \bPi_2} \right\|_2^2\leq 0\right)\\
& = & \Prob\left(\bE_1 - \bE_4 \leq \beta^2\delta_0\right)\\
& = & \Prob\left(\bE_1 - \bE_4 < \beta^2\delta_0 , \left\|\bT_{\bPi_1, \bPi_2} \right\|^2 >  t_0\right) \\
        && + \Prob\left(\bE_1 - \bE_4 < \beta^2\delta_0 , \left\|\bT_{\bPi_1, \bPi_2} \right\|^2 <
         t_0\right)\\
        & \leq & \Prob\left(\bE_1 - \bE_4 < \beta^2\delta_0 \ {\Big|} \ \left\|\bT_{\bPi_1, \bPi_2} \right\|^2 > t_0\right) \\
        && + \Prob\left(\left\|\bT_{\bPi_1, \bPi_2} \right\|^2 <
        \ t_0\right)\\
        & \leq & \Prob\left(\bE_1 < 2\beta^2\delta_0\ {\Big|} \ \left\|\bT_{\bPi_1, \bPi_2} \right\|^2 > t_0\right)\\
        && + \Prob\left(\bE_4 > \beta^2\delta_0 \ {\Big|} \ \left\|\bT_{\bPi_1, \bPi_2} \right\|^2 > t_0\right)\\     
         && + \Prob\left(\left\|\bT_{\bPi_1, \bPi_2} \right\|^2 <
        \ t_0\right)\\
\end{eqnarray*}
where we choose $\delta_0 = 4\left\|\bT_{\bPi_1, \bPi_2} \right\|^2$ and $t_0 = \nicefrac{t}{\lambda_{\max}(\bSig)}$ like in proof of Theorem \ref{thm:error-pi1-pi2}. Now for a choice of $t \in \left[\dfrac{h_{12}}{\snr},h_{12}\right]$, from Propositions \ref{prop:E1-dist-given-X},\ref{prop:E22-dist-cond-X} and \ref{prop:prob-T-less-t}, with the above inequality we get:

\begin{eqnarray*}
    &&\Prob\left(\left\| \bP_{\hatPi_1, X}^\perp \tilY_{{\hatPi_2}}\right\|_2^2 
- \left\| \bP_{\hatPi_1, X}^\perp \tilY_{\hatPi_1 \bPi_1^\top \bPi_2} \right\|_2^2 \leq 0\right) \\
& \leq & \exp\left(-c_1\dfrac{\beta^2}{\kappa^{2}(\bSig)}\cdot \dfrac{t}{\lambda_{\max}(\bSig)}\right)  \\
&& + \exp \left( - c_{4} \, \snr \, \min \left\{\frac{\snr}{h_{12}}\left(t - \dfrac{h_{12}}{\snr}\right)^2, \, \left(t - \dfrac{h_{12}}{\snr}\right) \right\}\right)\\
&& + 6 \exp\left(-\frac{h_{12}}{10}\left[\log \frac{h_{12}}{t} + \dfrac{t}{h_{12}} -1\right]\right)\\
        & = & \exp\left(-c_1\snr \ t\right)\\
        && +  \exp \left( - c_{4} \, \snr \, \min \left\{\frac{\snr}{h_{12}}\left(t - \dfrac{h_{12}}{\snr}\right)^2, \, \left(t - \dfrac{h_{12}}{\snr}\right) \right\}\right)\\
        && + 6 \exp\left(-\frac{h_{12}}{10}\left[\log \frac{h_{12}}{t} + \dfrac{t}{h_{12}} -1\right]\right)\\
\end{eqnarray*}
For {$\log \snr > 1$}, a valid choice for $t$ will be $t = h_{12}\dfrac{\log \snr}{\snr}$ where $h_{12} = k_{12}B$ since this $t$ will lie in the interval $\left[\dfrac{h_{12}}{\snr},h_{12}\right]$. For this choice of $t$, we can say that:
\begin{eqnarray*}
    &&\Prob\left(\left\| \bP_{\hatPi_1, X}^\perp \tilY_{{\hatPi_2}}\right\|_2^2 
- \left\| \bP_{\hatPi_1, X}^\perp \tilY_{\hatPi_1 \bPi_1^\top \bPi_2} \right\|_2^2 \leq 0\right) \\
& \leq & \exp\left(-c_1h_{12}\log \snr\right)\\
        && +  \exp \left( - c_{4} \, h_{12} \, \min \left\{\left[\log\left(\dfrac{\snr}{e}\right)\right]^2, \, \log\left(\dfrac{\snr}{e}\right) \right\}\right)\\
        && 6 \exp\left(-\frac{h_{12}}{10}\left[\log \frac{\snr}{\log \snr} + \dfrac{\log \snr}{\snr} -1\right]\right)\\
& \overset{(\ast)}{\leq} &   \exp\left(-c_1h_{12}\log \snr\right) \\
&& +  \exp \left( - c_{4} \, h_{12} \, \min \left\{\left[\log\left(\dfrac{\snr}{e}\right)\right]^2, \, \log\left(\dfrac{\snr}{e}\right) \right\}\right) \\
        &&+ 6 \exp \left( - c_0 \,  h_{12}\ \log \snr \right)\\
         & \leq & \exp(-c^\ast h_{12} \log \snr)  + \exp \left( - c_{4} \, h_{12} \, \min \left\{\left[\log\left(\dfrac{\snr}{e}\right)\right]^2, \, \log\left(\dfrac{\snr}{e}\right) \right\}\right)
\end{eqnarray*}
It can be verified easily that for $\snr > 1$, $\log\left(\frac{\snr}{\log \snr}\right)+\frac{\log \snr}{\snr} - 1 > \frac{\log \snr}{4}$. Observing the fact that $\min\{x,x^2\} \geq \frac{x^2}{1+x}$, defining $\phi(\snr)\coloneq \frac{(\log\snr-1)^2}{\log \snr}$,we can further upper bound the above by:
\begin{eqnarray*}
    \Prob\left(\left\| \bP_{\hatPi_1, X}^\perp \tilY_{{\hatPi_2}}\right\|_2^2 
- \left\| \bP_{\hatPi_1, X}^\perp \tilY_{\hatPi_1 \bPi_1^\top \bPi_2} \right\|_2^2 \leq 0\right) \leq \exp(-c^\ast h_{12} \log \snr) + \exp(-c_4 h_{12} \phi(\snr))
\end{eqnarray*}
Now, notice that for $\log \snr > 1$, we have:
\begin{eqnarray*}
    \log \snr - \phi(\snr) & = & \log\snr - \left(\log \snr + \frac{1}{\log \snr}-2\right)\\
    & = & 2 - \frac{1}{\log \snr} > 0
\end{eqnarray*}
Hence, using the last two bonds, for $B > \frac{(1+\gamma)\log K}{c_0^\ast\phi(\snr)}$ where $\gamma > 0$, denoting $r = c_0^\ast B \phi(\snr) - \log K$, we can say for a suitable choice of constant $c_1>0$:
\begin{eqnarray*}
    \Prob\left(\hatPi_{12} \ne \bI\right) &=& \sum_{(\hatPi_1, \hatPi_2) \in \calQ_K:\hatPi_{12}\neq \bI} \Prob\left(\left\| \bP_{\hatPi_1, X}^\perp \tilY_{{\hatPi_2}}\right\|_2^2 
< \left\| \bP_{\hatPi_1, X}^\perp \tilY_{\hatPi_1 \bPi_1^\top \bPi_2} \right\|_2^2\right)\\
&\leq & \sum_{(\hatPi_1, \hatPi_2) \in \calQ_K:\hatPi_{12}\neq \bI} e^{-c_0^\ast h_{12}\phi( \snr)}\\
& \leq & \sum_{k_{12}=2}^K K^{k_{12}} e^{-c_0^\ast k_{12}B\phi(\snr)}\\
& = & \sum_{k_{12} = 2}^K e^{-rk_{12}}\leq \frac{\exp(-2r)}{1 - \exp(-r)} \leq  c_1 ^2 e^{-2B \phi(\snr)}
\end{eqnarray*}

Now let us talk about the noise term in the upper bound of the bias term \eqref{eq:betahat-bias}. Note that we can write $\norm{\frac{\widetilde{\bX}_{\hatPi_1}^\top \bM_2 \beps/\sqrt{n}}{\widetilde{\bX}_{\hatPi_1}^\top  \widetilde{\bX}_{\hatPi_1}/n}}_2 = \|\bA\beps\|_2$ where $\beps \sim \calN(\bzero,\bI)$ and $\bA \coloneq \left(\frac{\widetilde{\bX}_{\hatPi_1}^\top  \widetilde{\bX}_{\hatPi_1}}{n}\right)^{-1}\frac{\widetilde{\bX}_{\hatPi_1}^\top \bM_2}{\sqrt{n}}$ where $\widetilde{\bX}_{\hatPi_1} = {\bSig}^{- \nicefrac{1}{2}}\hatPi_1 \bX$. Next we define:
\begin{eqnarray*}
    \bGam \coloneq \bA^\top\bA = \left(\frac{\widetilde{\bX}_{\hatPi_1}^\top  \widetilde{\bX}_{\hatPi_1}}{n}\right)^{-2}\dfrac{\bM_2^\top\widetilde{\bX}_{\hatPi_1}\widetilde{\bX}_{\hatPi_1}^\top \bM_2}{n}
\end{eqnarray*}
and note that: 
\begin{eqnarray*}
    \tr(\bGam) = \sqrt{\tr(\bGam^2)} = \|\bGam\|_2 & = & n\cdot \dfrac{\widetilde{\bX}_{\hatPi_1}^\top \bM_2\bM_2^\top\widetilde{\bX}_{\hatPi_1}}{\left(\widetilde{\bX}_{\hatPi_1}^\top  \widetilde{\bX}_{\hatPi_1}\right)^{2}}\\
    & \leq & n\cdot\dfrac{\left\|\bM_2\bM_2^\top\right\|_2}{\bX^\top\hatPi_1^\top\bSig^{-1}\hatPi_1\bX}\\
    & \leq & \dfrac{\lambda_{\max}(\bSig)\cdot\kappa(\bSig)}{\|\bX/\sqrt{n}\|^2_2}
\end{eqnarray*}
Thus invoking Lemma \ref{lem:concen-gaus-qf} for a fixed value of $\bX$, and choosing $t = f(n) > 1$ we have:
    {\small
    \begin{eqnarray*}
        \Prob\left(\|\bA\beps\|_2> \dfrac{\sqrt{\lambda_{\max}(\bSig)\cdot\kappa(\bSig)}}{||\bX/\sqrt{n}||_2}\sqrt{5t}\ \Bigg|\ \bX\right)
         & \leq  & \pr\left(||\bA\beps||^2_2> \dfrac{\lambda_{\max}(\bSig)\cdot\kappa(\bSig)}{||\bX/\sqrt{n}||^2_2}5t\ \Bigg|\ \bX\right)\\
        & \leq & \pr\left(||\bA\beps||^2_2> \dfrac{\lambda_{\max}(\bSig)\cdot\kappa(\bSig)}{||\bX/\sqrt{n}||^2_2}(1 + 2\sqrt{t} + 2t)\, \Bigg|\, \bX\right)\\
        & \leq & \pr\left(||\bA\beps||^2_2> \tr(\bGam) + 2\sqrt{\tr(\bGam^2)t} + 2||\bGam||_2t\,\Bigg|\,\bX\right)\\
        & \leq & \exp(-t)
    \end{eqnarray*}}
    This implies that:
    $$
    \Prob\left(\norm{\frac{\widetilde{\bX}_{\hatPi_1}^\top \bM_2 \beps/\sqrt{n}}{\widetilde{\bX}_{\hatPi_1}^\top  \widetilde{\bX}_{\hatPi_1}/n}}_2\geq \dfrac{\sqrt{\lambda_{\max}(\bSig)\cdot\kappa(\bSig)}}{||\bX/\sqrt{n}||_2}\sqrt{5f(n)}\ \Bigg|\ \bX\right) \leq \exp(-f(n))
    $$
Next assuming the entries of $\bX$ are iid $\calN(0,1)$ random variables, using Lemma \ref{lem:concen-chisd}, we have:
$$
\pr\left(\left|\left|\bX/\sqrt
        n\right|\right|_2< 1 - \sqrt{\dfrac{g(n)}{n}}\right) \leq \exp(-g(n))
$$
Thus, combining the results, we have:
{
\scriptsize
$$
\begin{aligned}
    & \pr\left(\left\{\left|\left|\bX/\sqrt
        n\right|\right|_2< 1 - \sqrt{\dfrac{g(n)}{n}}\right\}\cup\left\{\norm{\frac{\widetilde{\bX}_{\hatPi_1}^\top \bM_2 \beps/\sqrt{n}}{\widetilde{\bX}_{\hatPi_1}^\top  \widetilde{\bX}_{\hatPi_1}/n}}_2\geq \sqrt{\lambda_{\max}(\bSig)\cdot\kappa(\bSig)}\dfrac{\sqrt{5f(n)}}{1 - \sqrt{\dfrac{g(n)}{n}}}\right\}\right)\\
        \leq & \pr\left(\left\{\left|\left|\bX/\sqrt
        n\right|\right|_2< 1 - \sqrt{\dfrac{g(n)}{n}}\right\}\cup\left\{\norm{\frac{\widetilde{\bX}_{\hatPi_1}^\top \bM_2 \beps/\sqrt{n}}{\widetilde{\bX}_{\hatPi_1}^\top  \widetilde{\bX}_{\hatPi_1}/n}}_2\geq \sqrt{\lambda_{\max}(\bSig)\cdot\kappa(\bSig)}\dfrac{\sqrt{5f(n)}}{\left|\left|\bX/\sqrt
        n\right|\right|_2},\left|\left|\bX/\sqrt
        n\right|\right|_2\geq  1 - \sqrt{\dfrac{g(n)}{n}}\right\}\right)\\
        \leq & \pr\left(\left|\left|\bX/\sqrt
        n\right|\right|_2< 1 - \sqrt{\dfrac{g(n)}{n}}\right) + \pr\left(\norm{\frac{\widetilde{\bX}_{\hatPi_1}^\top \bM_2 \beps/\sqrt{n}}{\widetilde{\bX}_{\hatPi_1}^\top  \widetilde{\bX}_{\hatPi_1}/n}}_2\geq \sqrt{\lambda_{\max}(\bSig)\cdot\kappa(\bSig)}\dfrac{\sqrt{5f(n)}}{\left|\left|\bX/\sqrt
        n\right|\right|_2}\Bigg| \bX\right)\\
        \leq & \exp(-g(n)) + \exp(-f(n)) \\
        \leq & 2\exp(- f(n)\wedge g(n)) 
\end{aligned}
$$}
Thus, we have:
\begin{eqnarray*}
 \left|\hat{\beta} - \beta\right| & \leq & \norm{\beta}_2\cdot \kappa(\bSig)\norm{\hatPi_{12}- \bI}_2 + \dfrac{1}{\sqrt{n}}\norm{\dfrac{\widetilde{\bX}_{\hatPi_1}^\top \bM_2 \beps/\sqrt{n}}{\widetilde{\bX}_{\hatPi_1}^\top  \widetilde{\bX}_{\hatPi_1}/n}}_2\\
 & \leq & 0 + \dfrac{1}{\sqrt{n}}\sqrt{\lambda_{\max}(\bSig)\cdot\kappa(\bSig)}\dfrac{\sqrt{5f(n)}}{1 - \sqrt{\dfrac{g(n)}{n}}}\\
 & = & \sqrt{5\lambda_{\max}(\bSig)\cdot\kappa(\bSig)}\cdot \dfrac{\sqrt{\dfrac{f(n)}{n}}}{1 - \sqrt{\dfrac{g(n)}{n}}}
\end{eqnarray*}
with at least probability $p$, where:
\begin{eqnarray*}
    p \coloneq 1 - c_1 K^2 \exp({-B \phi(\snr)}) - 2\exp(- f(n)\wedge g(n)) 
\end{eqnarray*}
A valid choice of $f(n)$ and $g(n)$ can be $n^\delta$ for some $\delta < 1$, which gives us the concentration bound:
\begin{eqnarray*}
    \sqrt{5\lambda_{\max}(\bSig)\cdot\kappa(\bSig)}\cdot \dfrac{n^{-(1 - \delta)/2}}{1 - n^{-(1 - \delta)/2}}
\end{eqnarray*}
with probability $p = 1 - c_1 K^2 e^{-2B \phi(\snr)} - 2\exp(- n^\delta)$, where $n = KB$.
\end{proof}

\newpage
\subsection{Supporting Propositions and Lemmas}
\begin{lemma} \label{lem:specA}
The matrix $\bA$ in \eqref{eq:E22} for $h_2 = \textrm{rank }(\hatPi_2 {\bPi}_2^\top)$ has the following properties:
\begin{enumerate}
    \item[(a)] $-\kappa({\bSig}) \leq \lambda(\bA) \leq 1 \implies \norm{\bA}_2 \leq \kappa({\bSig})$
    \item[(b)] The following inequality holds for $\tr(\bA)$:
    \begin{equation*}
    \begin{split}
        &-\sqrt{2}h_2\kappa(\bSig)\\
        \leq & -h_2[\sqrt{2}\kappa(\bSig)-1] - \left(1 - \frac{1}{\kappa(\bSig)}\right) \leq \tr(\bA) \leq \min\left\{(n-1)\left( 1 - \sqrt[n-1]{ \nicefrac{1}{\kappa \left( \bSig\right)}} \right), \kappa(\bSig) - 1\right\}
    \end{split}
    \end{equation*} 
    \item[(c)] $|\tr(\bA)| \leq \sqrt{2}h_2\kappa(\bSig)$ 
\end{enumerate} 
\end{lemma}
\begin{proof}
Let us look at the parts separately:

{Part (a):}
Observe that, $\textrm{rank}(\bI_n - \bM) = \textrm{rank}({\bSig}^{-\frac12} (\bI_n- \hatPi_2 {\bPi}_2^\top ){\bSig}^{\frac12})=h_2$. Simplifying notations, let us write write $\bA= \bP^{\perp} - \bM^\top \bP^{\perp} \bM = (\bP^{\perp} - \bM^\top \bP^{\perp}) + (\bM^\top \bP^{\perp} - \bM^\top \bP^{\perp} \bM) = ((\bI_n- \bM^\top)\bP^{\perp} + \bM^\top \bP^{\perp}(\bI_n-\bM))$. Hence, $\textrm{rank}(\bA) \leq \textrm{rank}(\bI_n-\bM)^\top + \textrm{rank}(\bI_n-\bM) \leq 2h_2$.
$\bM^\top \bP\bM$ is a PSD of rank $(n-1)$, hence, $\lambda_{\min}(\bM^\top \bP\bM)=0$, moreover, $\lambda_{\max}(\bM^\top \bP\bM) = \norm{\bM^\top\bP\bM}_2 \leq \kappa({\bSig})$.

Observe that, by Weyl's inequality (\ref{lem:weyl-ineq}), 
\begin{equation*}
\begin{split}
\lambda_{\min}(\bP^{\perp}) - \lambda_{\max}(\bM^\top \bP\bM) & \leq \lambda(\bA) \leq \lambda_{\max}(\bP^{\perp}) - \lambda_{\min}(\bM^\top \bP\bM) \\
\implies - \lambda_{\max}(\bM^\top \bP\bM) & \leq \lambda(\bA) \leq 1 - \lambda_{\min}(\bM^\top \bP\bM) \\
\implies - \kappa({\bSig}) & \leq \lambda(\bA) \leq 1 
\end{split}
\end{equation*}
This completes the proof of part (a).

{Part (b):} We will start off by showing $\tr(\bA) \leq (n - 1) \left( 1 - \sqrt[n-1]{ \nicefrac{1}{\kappa \left( {\bSig} \right)}} \right)$. Trace of the matrix $\bA$ is given by:
    \begin{eqnarray*}
        \tr(\bA) & = & \tr \left(\bP^\perp \right) - \tr \left( \left({\bSig}^{-1/2} \hatPi_2 {\bPi}_2^\top {\bSig}^{1/2} \right)^\top \bP^\perp \left({\bSig}^{-1/2} \hatPi_2 {\bPi}_2^\top {\bSig}^{1/2} \right) \right).
    \end{eqnarray*}
    Here we use the spectral decomposition of ${\bSig} = \bQ \bLam \bQ^\top$ and define $\widetilde{{\bPi}} \coloneq \bQ^\top\hatPi_2 {\bPi}_2^\top \bQ$ which gives us: 
    \begin{eqnarray*}
       {\bSig}^{-1/2} \hatPi_2 {\bPi}_2^\top {\bSig}^{1/2} & = & (\bQ \bLam^{-1/2} \bQ^\top) \hatPi_2 {\bPi}_2^\top (\bQ \bLam^{1/2} \bQ^\top)\\
       & = & \bQ \bLam^{-1/2} (\bQ^\top \hatPi_2 {\bPi}_2^\top \bQ) \bLam^{1/2} \bQ^\top\\
       & = & \bQ\bLam^{-1/2} \widetilde{{\bPi}}\bLam^{1/2} \bQ^\top.
    \end{eqnarray*}
    Then we obtain:
    \begin{eqnarray*}
        \tr(\bA) & = & \tr \left(\bP^\perp \right) - \tr \left( \left( \bQ \bLam^{-1/2} \widetilde{{\bPi}} \bLam^{1/2} \bQ^\top \right)^\top \bP^\perp \left( \bQ \bLam^{-1/2} \widetilde{{\bPi}} \bLam^{1/2} \bQ^\top \right)\right).
    \end{eqnarray*}
    Here we also define $\bP^\ast \coloneqq \bQ^\top \bP^\perp \bQ$. Note that $\bP^\ast$ is also a projection matrix with rank $n-1$ as $\bP^\perp$ is a projection with rank $n-1$. Therefore, we get
    \begin{eqnarray}
        \tr(\bA) & = & \tr \left(\bP^\perp \right) - \tr \left( \left( \bLam^{-1/2} \widetilde{{\bPi}} \bLam^{1/2} \right)^\top \bP^\ast \left( \bLam^{-1/2} \widetilde{{\bPi}} \bLam^{1/2} \right)\right) \notag \\
        & = & \tr \left(\bP^\perp \right) - \tr \left( \left( \widetilde{{\bPi}} \bLam \widetilde{{\bPi}}^\top \right) \left( \bLam^{-1/2} \bP^\ast \bLam^{-1/2} \right) \right) \notag \\
        & = & (n - 1) - \tr \left( \left( \widetilde{{\bPi}} \bLam \widetilde{{\bPi}}^\top \right) \left( \bLam^{-1/2} \bP^\ast \bLam^{-1/2} \right) \right). \label{eq:upper_trA_eq1}
    \end{eqnarray}
    Here we will use Lemma \ref{lem:von-neuman} to get 
    \begin{eqnarray}
        \tr \left( \left( \widetilde{{\bPi}} \bLam \widetilde{{\bPi}}^\top \right) \left( \bLam^{-1/2} \bP^\ast \bLam^{-1/2} \right) \right) & \overset{\eqref{eq:trAB-lb}}{\geq} & \sum_{i = 1}^n \lambda_i \left( \widetilde{{\bPi}} \bLam \widetilde{{\bPi}}^\top \right) \lambda_{n - i + 1} \left( \bLam^{-1/2} \bP^\ast \bLam^{1/2} \right) \notag \\
        & = & \sum_{i = 1}^n \lambda_i \left( \bLam \right) \lambda_{n - i + 1} \left( \bLam^{-1/2} P \bLam^{1/2} \right). \label{eq:upper_trA_eq2}
    \end{eqnarray}
    The last line follows from the fact that $\bLam$ and $\widetilde{{\bPi}} \bLam \widetilde{{\bPi}}^\top$ have same eigenvalues as $\widetilde{{\bPi}}$ is orthogonal. As mentioned above, $\bP^\ast$ is a projection matrix with rank $n-1$. Thus we can write $\bP^\ast = \bI_n - \bu\bu^T$ for some $\| \bu \| = 1$. Thus, using the Lemma \ref{lem:interlacing-lemma}, for $i = 1, 2, \cdots, n-1$, we obtain
    \begin{eqnarray}
        \lambda_{n-i} \left( \bLam^{-1/2} \bP^\ast \bLam^{-1/2} \right) & = & \lambda_{n - i} \left( \bLam^{-1/2} \left( I - \bu\bu^\top \right) \bLam^{-1/2} \right) \notag \\
        & = & \lambda_{n- i} \left( \bLam^{-1} - \bLam^{-1/2} uu^\top \bLam^{-1/2} \right)  \notag \\
        & \overset{\eqref{eq:interlacing}}{\geq} & \lambda_{n - i + 1} \left( \bLam^{-1} - \bLam^{-1/2} \bu\bu^\top \bLam^{-1/2} + \bLam^{-1/2} \bu\bu^\top \bLam^{-1/2} \right) \notag \\
        & = & \lambda_{n - i + 1} \left( \bLam^{-1}  \right). \label{eq:upper_trA_eq3}
    \end{eqnarray}
    Moreover, as $\bP^\ast$ is a non-singular matrix, we have
    \begin{equation}\label{eq:upper_trA_eq4}
        \lambda_{n} \left( \bLam^{-1/2} \bP^\ast \bLam^{-1/2} \right) = 0.
    \end{equation}
    Therefore, from \eqref{eq:upper_trA_eq2} we get
    \begin{eqnarray}
        \tr \left( \left( \widetilde{{\bPi}} \bLam \widetilde{{\bPi}}^\top \right) \left( \bLam^{-1/2} P \bLam^{-1/2} \right) \right) & \geq & \sum_{i = 1}^n \lambda_i \left( \bLam \right) \lambda_{n - i + 1} \left( \bLam^{-1/2} P \bLam^{-1/2} \right) \notag \\
        & \overset{\eqref{eq:upper_trA_eq4}}{=} & \sum_{i = 2}^n \lambda_i \left( \bLam \right) \lambda_{n - i + 1} \left( \bLam^{-1/2} P \bLam^{-1/2} \right) \notag \\
        & = & \sum_{i = 1}^n \lambda_{i+1} \left( \bLam \right) \lambda_{n - i} \left( \bLam^{-1/2} P \bLam^{-1/2} \right) \notag \\
        & \overset{\eqref{eq:upper_trA_eq3}}{=} & \sum_{i = 1}^{n - 1} \lambda_{i+1} \left( \bLam \right) \lambda_{n - i} \left( \bLam^{-1} \right) \notag \\
        & = & \sum_{i = 1}^{n - 1} \frac{\lambda_{i+1}\left( \bLam \right)}{\lambda_i \left( \bLam \right)} \notag \\
        & \overset{\eqref{eq:AMGM}}{\geq} & (n - 1) \left( \prod_{i = 1}^{n - 1} \frac{\lambda_{i+1}\left( \bLam \right)}{\lambda_i \left( \bLam \right)} \right)^{\frac{1}{n-1}} \notag \\
         & = & (n - 1) \left( \frac{\lambda_{n}\left( \bLam \right)}{\lambda_1 \left( \bLam \right)} \right)^{\frac{1}{n-1}}. \label{eq:upper_trA_eq5}
    \end{eqnarray}
    Then combining \eqref{eq:upper_trA_eq1} and \eqref{eq:upper_trA_eq5}, we get
    \begin{eqnarray*}
        \tr(\bA) & \leq & (n - 1) \left( 1 - \left( \frac{\lambda_{n}\left( \bLam \right)}{\lambda_1 \left( \bLam \right)} \right)^{\frac{1}{n-1}} \right) \\
        & = & (n - 1) \left( 1 - \left( \frac{1}{\kappa \left( {\bSig} \right)} \right)^{\frac{1}{n-1}} \right).
    \end{eqnarray*}
    This completes the proof of this result.

Now, as defined, $\widetilde{\bX}_{\bPi_1} = {\bSig}^{-\frac{1}{2}}\bPi_1X$ with $\bP_{\bPi_1, X} = Proj\left(\widetilde{\bX}_{\bPi_1}\right) = \dfrac{\widetilde{\bX}_{\bPi_1}\widetilde{\bX}_{\bPi_1}^\top}{\widetilde{\bX}_{\bPi_1}^\top \widetilde{\bX}_{\bPi_1}}$. Here want to bound the trace of the matrix $\bA = \left(\bP^\perp - \bM^\top \bP^\perp \bM \right)$ where $\bM = {\bSig}^{-1/2} \hatPi_2 {\bPi}_2^\top {\bSig}^{1/2}$. Let us denote $\bX^\ast = \bPi_1\bX\sim N(\bzero,\bI_n)$, which means $\widetilde{\bX}_{\bPi_1} = {\bSig}^{-1/2}\bX^\ast \sim N(0, \Siginv)$. Thus, we have:
\begin{equation*}
    \begin{split}
        tr(\bA) = & \tr\left(\bP^\perp - \bM^\top \bP^\perp \bM \right)\\
        = & (n-1) - \tr(\bM^\top\bM) + \tr(\bM^\top\bP\bM)\\
        = & n-1 - \tr(\bM^\top\bM) +\dfrac{\tr\left(\bM^\top\widetilde{\bX}_{\bPi_1}\widetilde{\bX}_{\bPi_1}^\top \bM\right)}{\widetilde{\bX}_{\bPi_1}^\top \widetilde{\bX}_{\bPi_1}}\\
        = & n-1 - \tr(\bM^\top\bM)+ \dfrac{\widetilde{\bX}_{\bPi_1}^\top \bM \bM^\top\widetilde{\bX}_{\bPi_1}}{\widetilde{\bX}_{\bPi_1}^\top \widetilde{\bX}_{\bPi_1}}\\
        = & (n-1) - \tr(\bM\bM^\top)+ \dfrac{{\bX^\ast}^\top{\bSig}^{-1/2}\bM\bM^\top{\bSig}^{-1/2}\bX^\ast}{{\bX^\ast}^\top \Siginv\bX^\ast}\\
        = & \tr(\bI_n-\bM\bM^\top)- \dfrac{{\bX^\ast}^\top{\bSig}^{-1/2}(\bI_n -\bM\bM^\top){\bSig}^{-1/2}\bX^\ast}{{\bX^\ast}^\top \Siginv\bX^\ast}\\ 
        = & \tr(\bI_n - \bM\bM^\top) - \dfrac{\widetilde{\bX}_{\bPi_1}^\top(\bI_n - \bM\bM^\top)\widetilde{\bX}_{\bPi_1}}{\widetilde{\bX}_{\bPi_1}^\top \widetilde{\bX}_{\bPi_1}}\\
    \end{split}
\end{equation*}
Let use denote the matrix $\bB = \bI_n - \bM\bM^\top$, then w.p. 1, we have the following bounds:
\begin{itemize}
    \item $|\tr(\bA)| \geq \left|tr(\bB) - \dfrac{\widetilde{\bX}_{\bPi_1}^\top \bB\widetilde{\bX}_{\bPi_1}}{\widetilde{\bX}_{\bPi_1}^\top \widetilde{\bX}_{\bPi_1}}\right| \geq |\tr(\bB)| - \left|\dfrac{\widetilde{\bX}_{\bPi_1}^\top \bB\widetilde{\bX}_{\bPi_1}}{\widetilde{\bX}_{\bPi_1}^\top \widetilde{\bX}_{\bPi_1}}\right|\geq |\tr(\bB)| - |\lambda_{\max}(\bB)|$ 
    \item $\tr(\bB) - \lambda_{\max}(\bB)\leq \tr(\bA)\leq \tr(\bB) - \lambda_{\min}(\bB)$
\end{itemize}

Let us simplify $\bB$ in terms of $\bK = \bI_n - \hatPi_2{\bPi}_2^\top = \bI_n - \text{block-diag}(\bI_{n-h_2},{\bPi}_{h_2}) = \text{block-diag}(\bzero_{n-h_2},\bI_{h_2}-{\bPi}_{h_2})$ as:
$$
\begin{aligned}
    \bB = & \bI_n - \bM\bM^\top\\
    = & \bI_n - {\bSig}^{-1/2} \hatPi_2{\bPi}_2^\top  {\bSig}^{1/2}{\bSig}^{1/2} {\bPi}_2\hatPi^\top_2{\bSig}^{-1/2}\\
    = & {\bSig}^{-1/2}\left({\bSig} - \hatPi_2 {\bPi}_2^\top {\bSig} {\bPi}_2\hatPi^\top_2\right){\bSig}^{-1/2}\\
    = & {\bSig}^{-1/2}\left[{\bSig} - \begin{bmatrix}
        \bI_{n-h_2} & \bzero\\
        \bzero^\top & {\bPi}_{h_2}
    \end{bmatrix}\begin{bmatrix}
        {\bSig}_{n-h_2} & {\bSig}_{n-h_2,h_2}\\
        {\bSig}_{n-h_2,h_2}^{\top} & {\bSig}_{h_2}
    \end{bmatrix} \begin{bmatrix}
        \bI_{n-h_2} & \bzero^\top\\
        \bzero & {\bPi}^\top_{h_2}
    \end{bmatrix}\right]{\bSig}^{-1/2} \\
    = & {\bSig}^{-1/2}\underbrace{\begin{bmatrix}
        \bzero_{n-h_2} & {\bSig}_{n-h_2,h_2}^{*}(\bI_{h_2} - {\bPi}_{h_2})^{\top}\\
        (\bI_{h_2} - {\bPi}_{h_2}){\bSig}_{n-h_2,h_2}^{*^\top} & {\bSig}_{h_2} - {\bPi}_{h_2}{\bSig}_{h_2}{\bPi}_{h_2}^\top
    \end{bmatrix}}_{{\bSig}_{n,h_2}}{\bSig}^{-1/2}
\end{aligned}
$$
Let $\bB = \bI_n - \bM\bM^\top$ with $\bM = {\bSig}^{-1/2}\hatPi_2{\bPi}_2^\top{\bSig}^{1/2}.$

\textbf{Step 1: Analytical expression for $\bM\bM^\top$}: Set ${\bPi}_2^{\ast} = \hatPi_2{\bPi}_2^\top$ and diagonalize ${\bSig}$ as ${\bSig} = \bQ\bLam \bQ^\top$ with $\bQ$ orthogonal and $\bLam = \operatorname{diag}(\lambda_1,\dots,\lambda_n)$, $0<\lambda_n\le \dots \le \lambda_1$. Then
\begin{equation}
\bM\bM^\top = {\bSig}^{-1/2}{\bPi}_2^{\ast}{\bSig} {\bPi}_2^{\ast\top}{\bSig}^{-1/2}
= \bQ\bLam^{-1/2}\bQ^\top {\bPi}_2^{\ast} \bQ\bLam \bQ^\top {\bPi}_2^{\ast\top} \bQ \bLam^{-1/2}\bQ^\top.
\end{equation}
Define the orthogonal matrix $\widetilde{{\bPi}}=\bQ^\top{\bPi}_2^{\ast} \bQ$. Using $\bQ^\top \bQ=\bI_n$ we obtain $\bM\bM^\top = \bQ\bLam^{-1/2}\widetilde{{\bPi}}\bLam
\widetilde{{\bPi}}^\top \bLam^{-1/2}\bQ^\top$. Hence, by the cyclic property of trace, $\tr(\bM\bM^\top)=\tr\bigl(\widetilde{{\bPi}}\bLam \widetilde{{\bPi}}^\top \bLam^{-1}\bigr).$

\textbf{Step 2: Upper bound for $\tr (\bB)$}: Observe
\begin{equation*}
\tr (\bB) = n - \tr(\bM\bM^\top)
= n - \tr\bigl(\widetilde{{\bPi}}\bLam \widetilde{{\bPi}}^\top \bLam^{-1}\bigr).
\end{equation*}
Let $\bA_1=\widetilde{{\bPi}}\bLam \widetilde{{\bPi}}^\top$ and $\bA_2=\bLam^{-1}$; both are positive definite. By Von Neumann’s trace inequality (\ref{eq:trAB-lb}), $\tr(\bA_1\bA_2)\geq\sum_{i=1}^n {\bSig}_i(\bA_1){\bSig}_{n-i+1}(bA_2)$. Since $\bA_1$ has ordered (decreasing) eigenvalues $(\lambda_1,\ldots,\lambda_n)$ and $\bA_2$ has ordered eigenvalues $(1/\lambda_n,\ldots,1/\lambda_1)$, the right‑hand side equals $n$. Therefore $\tr(\bM\bM^\top)\geq n \Longrightarrow\quad \tr \bB\leq 0$. 

\textbf{Step 3: Lower bound for $\tr(\bB)$:} Write:
\begin{equation*}
\tr (\bB)
= \tr\bigl(\bI_n - \widetilde{{\bPi}}\bLam \widetilde{{\bPi}}^\top \bLam^{-1}\bigr)
= \tr\bigl(\bI_n -\bLam^{-1}\widetilde{{\bPi}}\bLam\bigr)
+ \tr\bigl(\bLam^{-1}\widetilde{{\bPi}}\bLam - \widetilde{{\bPi}}\bLam \widetilde{{\bPi}}^\top \bLam^{-1}\bigr).
\end{equation*}
The first term equals $\tr(\bI_n-\widetilde{{\bPi}})=h_2$, the number of points permuted by ${\bPi}_2^{\ast}$. For the second term, set $\widetilde{{\bPi}}= (\bI_n + \bDel)$. Using Von Neumann's trace inequality (\ref{lem:von-neuman}) again,
{
\small
\begin{equation*}
\bigl|\tr\bigl(-{\bDel}^\top\bLam^{-1}\widetilde{{\bPi}}\bLam\bigr)\bigr|
\leq {\sigma}_{\max}\bigl(\bLam^{-1}\widetilde{{\bPi}}\bLam\bigr)\sum_{i=1}^{h_2}{\sigma}_i(\bDel)
\leq {\sigma}_{\max}\bigl(\bLam^{-1}\widetilde{{\bPi}}\bLam\bigr) \sqrt{h_2}\left\|\bDel\right\|_F
=h_2\sqrt{2} \kappa({\bSig}),
\end{equation*}}
where $\kappa({\bSig})=\lambda_1/\lambda_n$ is the condition number of ${\bSig}$ and $\norm{\bDel}_F=\sqrt{2h_2}$ for a permutation of $h_2$ indices.
Combining, we have
\begin{equation}\label{eq:B-trace-bound}
-h[\sqrt{2}\kappa({\bSig})-1]=h-\sqrt{2}h\kappa({\bSig})
\leq \tr (\bB)
\leq 0
\end{equation}
\textbf{Step 4: Eigenvalue structure of $\bB$:} We know $\bM\bM^\top = \bQ\bLam^{-1/2}\widetilde{{\bPi}}\bLam \widetilde{{\bPi}}^\top \bLam^{-1/2} \bQ^\top$. Hence the eigenvalues of $\bM\bM^\top$ match those of $\bLam^{-1/2}\widetilde{{\bPi}}\bLam \widetilde{{\bPi}}^\top \bLam^{-1/2} \sim \widetilde{{\bPi}}\bLam \widetilde{{\bPi}}^\top \bLam^{-1}$.

Then, 
\begin{equation}\label{eq:B-eigen-bound}
\begin{split}
\lambda_{\min}(\widetilde{{\bPi}}\bLam \widetilde{{\bPi}}^\top)\lambda_{\min}(\bLam^{-1}) & \leq \lambda(\widetilde{{\bPi}}\bLam \widetilde{{\bPi}}^\top \bLam^{-1})  \leq \lambda_{\max}(\widetilde{{\bPi}}\bLam \widetilde{{\bPi}}^\top)\lambda_{\max}(\bLam^{-1}) \\
\dfrac{1}{\kappa({\bSig})} & \leq \lambda(\widetilde{{\bPi}}\bLam \widetilde{{\bPi}}^\top \bLam^{-1}) \leq \kappa({\bSig})\\
1- \kappa({\bSig}) &\leq \lambda(\bB) \leq 1 - \dfrac{1}{\kappa({\bSig})} 
\end{split}
\end{equation}

Combining the two bounds (\ref{eq:B-trace-bound}) and (\ref{eq:B-eigen-bound}), we have:
$$
\begin{aligned}
    \tr(\bB) - \lambda_{\max}(\bB)\leq & \tr(\bA)\leq \tr(\bB) - \lambda_{\min}(\bB)\\
    \implies -h[\sqrt{2}\kappa({\bSig})-1] - \left(1 - \frac{1}{\kappa({\bSig})}\right) \leq & \tr(\bA) \leq -(1 - \kappa({\bSig})) = \kappa({\bSig}
    ) - 1\\
   \implies  -h_2[\sqrt{2}\kappa({\bSig})-1] - h_2\left(1 - \frac{1}{\sqrt{2}\kappa({\bSig})}\right) \leq & \tr(\bA) \leq\kappa({\bSig}
    ) - 1 \\
    \implies h_2(1 - \sqrt{2}\kappa({\bSig})) + h_2\left(  \frac{1 - \sqrt{2}\kappa({\bSig})}{\sqrt{2}\kappa({\bSig})}\right) \leq & \tr(\bA) \leq -\kappa({\bSig}
    ) - 1\\
    \implies -\sqrt{2}h_2\kappa(\bSig)\leq  -h_2\frac{2\kappa^2({\bSig}) - 1}{\sqrt{2}\kappa({\bSig})}\leq & \tr(\bA) \leq \kappa({\bSig}
    ) - 1\\
\end{aligned}
$$
Combining the above inequality completes the proof.

{Part (c):}
Let us define some short forms $\kappa \coloneqq \kappa(\bSig)$, $A := \sqrt{2}h_2 \kappa$, with $B \coloneqq \min \left\{ (n-1)\left(1-\kappa^{-1/(n-1)}\right), \ \kappa-1 \right\}.$ We will show that $ A>B$.

First, we prove that $(n-1)\left(1-\kappa^{-1/(n-1)}\right) \leq \kappa-1$.  Set $x \coloneq \kappa^{1/(n-1)} \ge 1$. Then
\begin{equation}
(n-1)\left(1-\kappa^{-1/(n-1)}\right)
=
(n-1)\left(1-\frac{1}{x}\right)
=
\frac{(n-1)(x-1)}{x}
\le
(n-1)(x-1).
\end{equation}
On the other hand, by Bernoulli's inequality (\ref{lem:bernoulli-ineq}), $x^{n-1}-1 \ge (n-1)(x-1).$ Since $x^{n-1}=\kappa$, it follows that $
(n-1)(x-1) \le \kappa-1.$
Therefore, $(n-1)\left(1-\kappa^{-1/(n-1)}\right) \leq \kappa-1$. Hence the minimum is attained by the first term, and so $B = (n-1)\left(1-\kappa^{-1/(n-1)}\right)$. Now, since $\kappa \ge 1$ and $h_2 \ge 1$, we have $A = \sqrt{2}\, h_2 \kappa \ge \sqrt{2}\,\kappa > \kappa-1$. Also, from the previous step, $B \leq \kappa-1$. Combining these two inequalities yields $B \leq \kappa-1 < \sqrt{2}\, h_2 \kappa = A$. Thus,
\begin{equation*}
\sqrt{2}\, h_2 \kappa(\bSig)
>
\min \left\{ (n-1)\left(1-\kappa(\bSig)^{-1/(n-1)}\right), \ \kappa(\bSig)-1 \right\}.
\end{equation*}
This completes the proof.
\end{proof}

\newpage
\subsection{Technical Lemmas proved}

\begin{lemma}\label{lem:matrix-trace-bound}
Let $B\in\mathbb R^{n\times n}$ be symmetric positive–definite, and
let $X\in\mathbb R^{m\times n}$ be arbitrary.  Then
$$
  \lambda_{\min}(B)\lVert X\rVert_F^{2} \leq \tr(XBX^{\top}) \leq
  \lambda_{\max}(B)\,\lVert X\rVert_F^{2}
$$
\end{lemma}

\begin{proof}
Begin by noting that we can Write the Frobenius norm as a trace $\lVert X\rVert_F^{2}\;=\;\tr(XX^{\top})$. Because $B$ is symmetric positive–definite, it admits an orthogonal eigen decomposition $B = \bQ\bLam \bQ^{\top}$, where $\bLam=\operatorname{diag}(\lambda_1,\dots,\lambda_n),\;
0<\lambda_{\min}=\lambda_1\le\cdots\le\lambda_n=\lambda_{\max}.$ Insert this into the trace and use cyclicity $\tr(XBX^{\top})
  =\tr\bigl(X\bQ\bLam \bQ^{\top}X^{\top}\bigr)
  =\tr\bigl(\bLam\,Y^{\top}Y\bigr),$ for $Y\coloneqq X\bQ.$ Because $\bQ$ is orthogonal $(\bQ^{\top}\bQ=\bI)$ we have $YY^{\top}=X\bQ \bQ^{\top}X^{\top}=XX^{\top}$. Thus we have $\lVert Y\rVert_F^{2}=\tr(YY^{\top})=\tr(XX^{\top})=\lVert X\rVert_F^{2}$. Expanding the remaining trace column-wise, $\tr(XBX^{\top})=\sum_{k=1}^{n}\lambda_k\,\lVert Y_{*k}\rVert_2^{2}$,
where $Y_{*k}$ is the $k$-th column of $Y$.  Since every
$\lVert Y_{*k}\rVert_2^{2}\ge 0$,
\[
\lambda_{\min}\sum_{k}\lVert Y_{*k}\rVert_2^{2}
  \;\le\;
  \sum_{k}\lambda_k\,\lVert Y_{*k}\rVert_2^{2}
  \;\le\;
  \lambda_{\max}\sum_{k}\lVert Y_{*k}\rVert_2^{2}.
\]
Using $\sum_{k}\lVert Y_{*k}\rVert_2^{2}=\lVert Y\rVert_F^{2}=\lVert X\rVert_F^{2}$
gives the desired inequality.
\end{proof}

\begin{lemma}\label{lem:raleigh-bound}
Let $A \in \R^{n \times n}$  be any matrix (not necessarily symmetric), and let $B \in \R^{n \times n}$ be symmetric positive definite. We consider the generalized Rayleigh quotient $ R(\bx) \coloneq \frac{\bx^\top A \bx}{\bx^\top B \bx}$ for $\bx \neq 0$. Then we have the following inequality:
\begin{eqnarray*}
    \left| R(\bx) \right| \leq \norm{ B^{-1/2} A B^{-1/2}}_2 \leq \dfrac{\norm{A}_2}{\lambda_{\min}(B)}
\end{eqnarray*}
\end{lemma}

\begin{proof}
Since $B \succ 0$, it has a unique symmetric positive definite square root $B^{1/2}$, with inverse $B^{-1/2}$. Define the change of variable $\by = B^{1/2} \bx \iff \quad \bx = B^{-1/2} \by$
Substituting into the Rayleigh quotient, we get:
\begin{align*}
\frac{\bx^\top A \bx}{\bx^\top B \bx} 
= \frac{(B^{-1/2} \by)^\top A (B^{-1/2} \by)}{(B^{-1/2} \by)^\top B (B^{-1/2} \by)} = \frac{\by^\top B^{-1/2} A B^{-1/2} \by}{\by^\top \by}
\end{align*}
Let $ \bM \coloneqq B^{-1/2} A B^{-1/2}$, then $\frac{\bx^\top A \bx}{\bx^\top B \bx} = \frac{\by^\top \bM \by}{\by^\top \by}$. This is the standard Rayleigh quotient of the matrix $\bM$ with respect to the nonzero vector $\by$. To bound it, we apply the Cauchy–Schwarz inequality:
$\left|\by^\top \bM \by\right| = \left|\langle \by \bM \by \rangle\right| \leq \|\by\| \cdot \|\bM \by\|$. Divide both sides by $\norm{\by}^2$ (since  $\by \neq 0$):
$$
\left| \frac{\by^\top \bM \by}{\by^\top \by} \right| \leq \frac{\|\bM \by\|}{\|\by\|} \leq \sup_{\mathbf{z} \neq 0} \frac{\|\bM \mathbf{z}\|}{\|\mathbf{z}\|} = \|\bM\|_2.
$$
Therefore,
$$
\left| \frac{\bx^\top A \bx}{\bx^\top B \bx} \right| 
= \left| \frac{\by^\top \bM \by}{\by^\top \by} \right| 
\leq \|M\|_2 = \|B^{-1/2} A B^{-1/2}\|_2 \leq \dfrac{\norm{A}_2}{\lambda_{\min}(B)}
$$
This completes the proof. 
\end{proof}

\begin{lemma}\label{lem:T-bound}
Let $\bSig \succ 0$ and define the weighted inner product $\langle a,b\rangle_{\Sigma^{-1}} \coloneq a^\top \bSig^{-1} b$  and norm $\|v\|_{\Sigma^{-1}}^2 \coloneqq v^\top \bSig^{-1} v$. Let $X_1\in\R^n$ with $X_1\neq 0$, let $P$ be a permutation matrix, and set $X_{12}:=PX_1$.
Define
$\|T_{\Pi_1,\Pi_2}\|^2
\coloneqq \|X_{12}\|_{\Sigma^{-1}}^2
   - \frac{\langle X_{12},X_1\rangle_{\Sigma^{-1}}^2}{\|X_1\|_{\Sigma^{-1}}^2}$. Then, we have the following property:
\begin{enumerate}
\item[(i)] {Projection (residual) form.} $\|T_{\Pi_1,\Pi_2}\|^2
= \big\|X_{12}-\operatorname{proj}^{(\Sigma^{-1})}_{X_{1}}\!\left(X_{12}\right)\big\|_{\Sigma^{-1}}^2$,
with
$\operatorname{proj}^{(\Sigma^{-1})}_{X_{1}}\!\left(X_{12}\right)
= \frac{\langle X_{12},X_{1}\rangle_{\Sigma^{-1}}}{\|X_{1}\|_{\Sigma^{-1}}^2}\,X_{1}.$
\item[(ii)] {Whitened (Euclidean) form.} With $u\coloneqq\bSig^{-1/2}X_{12}$ and $v\coloneqq\bSig^{-1/2}X_{1}$,
$$
\|T_{\Pi_1,\Pi_2}\|^2
= \min_{\alpha\in\R}\|\bSig^{-1/2}(PX_1-\alpha X_1)\|_2^2
= \|u\|_2^2 - \frac{(u^\top v)^2}{\|v\|_2^2}.
$$
\item[(iii)] {Spectral sandwich.}  
Let $\lambda_{\min}(\bSig^{-1})$ and $\lambda_{\max}(\bSig^{-1})$ be the extremal eigenvalues of $\bSig^{-1}$. Then
$$
\lambda_{\min}(\bSig^{-1})\,
\min_{\alpha\in\R}\|PX_1-\alpha X_1\|_2^2
\;\le\;
\|T_{\Pi_1,\Pi_2}\|^2
\;\le\;
\lambda_{\max}(\bSig^{-1})\,
\min_{\alpha\in\R}\|PX_1-\alpha X_1\|_2^2.
$$
Moreover, the inner minimization admits the explicit form
$$
\min_{\alpha\in\R}\|PX_1-\alpha X_1\|_2^2
= \|PX_1\|_2^2 - \frac{(X_1^\top P X_1)^2}{\|X_1\|_2^2}.
$$
\end{enumerate}
\end{lemma}

\begin{proof}
Define for $\alpha\in\R$ the residual $r(\alpha) \coloneqq PX_1 - \alpha X_1$.cThen by definition
\begin{eqnarray*}
f(\alpha) 
&\coloneqq& \|r(\alpha)\|_{\Sigma^{-1}}^2 \\
&=& (PX_1-\alpha X_1)^\top \bSigma^{-1}(PX_1-\alpha X_1) \\
&=& \|X_{12}\|_{\Sigma^{-1}}^2 
   - 2\alpha \langle X_{12},X_1\rangle_{\Sigma^{-1}}
   + \alpha^2 \|X_1\|_{\Sigma^{-1}}^2.
\end{eqnarray*}
This is a convex quadratic in $\alpha$, minimized at
$\alpha^\star = \frac{\langle X_{12},X_1\rangle_{\Sigma^{-1}}}{\|X_1\|_{\Sigma^{-1}}^2}$. Substituting $\alpha^\star$ gives
$$
\min_{\alpha\in\R} f(\alpha)
= \|X_{12}\|_{\Sigma^{-1}}^2
- \frac{\langle X_{12},X_1\rangle_{\Sigma^{-1}}^2}{\|X_1\|_{\Sigma^{-1}}^2},
$$
which proves the identity in the statement and shows that 
$\|T_{\Pi_1,\Pi_2}\|^2$ is exactly the squared residual after projecting
$X_{12}$ onto the span of $X_1$ under the $\bSigma^{-1}$ inner product. 
This establishes part (i).

For (ii), note that $\|v\|_{\Sigma^{-1}} = \|\bSigma^{-1/2} v\|_2$. 
Setting $u = \bSigma^{-1/2}X_{12}$ and $v = \bSigma^{-1/2}X_1$,
we can rewrite $\|T_{\Pi_1,\Pi_2}\|^2
= \min_{\alpha\in\R}\|u - \alpha v\|_2^2
= \|u\|_2^2 - \frac{(u^\top v)^2}{\|v\|_2^2}$, which is the whitened (Euclidean) projection form.

Finally, for (iii), recall the Rayleigh--Ritz inequality: for all $z\in\R^n$,
$$
\lambda_{\min}(\bSigma^{-1})\|z\|_2^2
\leq z^\top \bSigma^{-1} z
\leq \lambda_{\max}(\bSigma^{-1})\|z\|_2^2.
$$
Applying this to $z=PX_1-\alpha X_1$ and then minimizing over $\alpha$ yields
$$
\lambda_{\min}(\bSigma^{-1}) \min_{\alpha}\|PX_1-\alpha X_1\|_2^2
\leq \|T_{\Pi_1,\Pi_2}\|^2
\leq \lambda_{\max}(\bSigma^{-1}) \min_{\alpha}\|PX_1-\alpha X_1\|_2^2.
$$
The inner minimization admits the explicit form
$$
\min_{\alpha\in\R}\|PX_1-\alpha X_1\|_2^2
= \|PX_1\|_2^2 - \frac{(X_1^\top P X_1)^2}{\|X_1\|_2^2},
$$
which completes the proof.
\end{proof}

\newpage
\subsection{Known Technical Lemmas}

\begin{lemma}[Bernoulli's inequality]\label{lem:bernoulli-ineq}
Let $r \ge 1$ and $x \ge -1$. Then $(1+x)^r \ge 1+rx$.
\end{lemma}
\begin{lemma}[Hanson--Wright inequality]\label{lem:hanson-wright}
Let $X=(X_1,\dots,X_n)^\top$ be a random vector with independent, mean-zero, sub-Gaussian entries satisfying $\|X_i\|_{\psi_2} \leq K$ for $i=1,\dots,n.$ Let \(A\in\R^{n\times n}\) be a fixed matrix. Then for every $t>0$,
$$
\Pr\!\left(\left|X^\top A X-\mathbb{E}[X^\top A X]\right|\ge t\right)
\le
2\exp\left[
-c\min\left(
\frac{t^2}{K^4\|A\|_F^2},
\frac{t}{K^2\|A\|_2}
\right)
\right],
$$
where $c>0$ is an absolute constant, $\|A\|_F$ is the Frobenius norm, and $\|A\|_2$ is the operator norm.
\end{lemma}

\begin{lemma}[Interlacing Theorem]\label{lem:interlacing-lemma}
Let $A \in \R^{n \times n}$ be square symmetric matrix and $y \in \R^n$, then $\forall i = 1,\cdots, n-1$ and $a \in \R$, we have:
\begin{equation}\label{eq:interlacing}
    \lambda_{n-i+1}(A +ayy^\top) \leq \lambda_{n-i}(A)
\end{equation}
where $\lambda_1(B) > \lambda_2(B) > \cdots > \lambda_{n}(B)$ are the ordered eigen values of a matrix $B$.
\end{lemma}

\begin{lemma}[Von Neumann's trace inequality]\label{lem:von-neuman}
 for any $n \times n$ complex matrices $\bA$ and $B$ with singular values
$\alpha_1 \geq \alpha_2 \geq \cdots \geq \alpha_n$ and $\beta_1 \geq \beta_2 \geq \cdots \geq \beta_n$ respectively, we have
\begin{equation*}
 |\tr(AB)| \leq \sum_{i=1}^{n} \alpha_i \beta_i
\end{equation*}
with equality if and only if $\bA$ and $B^{\dagger}$ share singular vectors.
\end{lemma}
A simple corollary to this is the following result: 

\begin{lemma}
For Hermitian $n \times n$ positive semi-definite complex matrices 
$A$ and $B$, where the eigenvalues are sorted decreasingly
$
\alpha_1 \geq \alpha_2 \geq \cdots \geq \alpha_n$  and $\beta_1 \geq \beta_2 \geq \cdots \geq \beta_n$
we have
\begin{equation}\label{eq:trAB-lb}
    \sum_{i=1}^{n} \alpha_i \beta_{n - i + 1} \leq \tr(AB) \leq \sum_{i=1}^{n} \alpha_i \beta_i.   
\end{equation}
\end{lemma}

\begin{lemma}[Weyl's inequality]\label{lem:weyl-ineq}
Let $A,B \in \R^{n\times n}$ be symmetric matrices, and let
\begin{equation*}
\lambda_1(M) \ge \lambda_2(M) \ge \cdots \ge \lambda_n(M)
\end{equation*}
denote the ordered eigenvalues of a symmetric matrix $M$. Then for each $i=1,\dots,n$,
\begin{equation*}
\lambda_i(A) + \lambda_n(B) \leq 
\lambda_i(A+B) \leq
\lambda_i(A) + \lambda_1(B).
\end{equation*}
Equivalently,
\begin{equation*}
\left| \lambda_i(A+B)-\lambda_i(A) \right|
\le
\|B\|_2,
\qquad i=1,\dots,n.
\end{equation*}
\end{lemma}

\begin{lemma}[AM--GM Inequality]
    Let $\alpha_1, \alpha_2, \dots, \alpha_n $ be positive real numbers. Then
    \begin{equation}\label{eq:AMGM}
        \frac{\alpha_1 + \alpha_2 + \dots + \alpha_n}{n} \geq \left(\alpha_1 \alpha_2 \dotsm \alpha_n \right)^{\tfrac1n}
    \end{equation}
    with equality if and only if 
    \(\alpha_1 = \alpha_2 = \dots = \alpha_n.\)
\end{lemma}

\begin{lemma}[\cite{vershynin2020high}]\label{lem:concen-gaus-qf}
    Let $A$ be an $m \times n$ matrix, $\Gamma = A^\top A$, and $g \sim N(\bzero, \bI_n)$. Then for all $t > 0$,
    $$
    \pr(||Ag||^2_2 > \tr(\Gamma) + 2\sqrt{\tr(\Gamma^2)t} + 2||\Gamma||_2t)\leq \exp(-t)
    $$
\end{lemma}

\begin{lemma}[\cite{vershynin2020high}]\label{lem:concen-chisd}
    Let $g \sim N(\bzero, \bI_n)$. Then for all $t > 0$,
    $$
    \pr\left(||g/\sqrt{n}||_2 > 1 -\sqrt{2t/n}\right)\leq \exp(-t)
    $$
\end{lemma}

\newpage
\section{Covariance Parameters Estimates}\label{suppsec:variance-estimates}
\begin{figure}[!htbp]
    \centering
    \begin{subfigure}{0.9\textwidth}
        \centering
        \includegraphics[width=\textwidth]{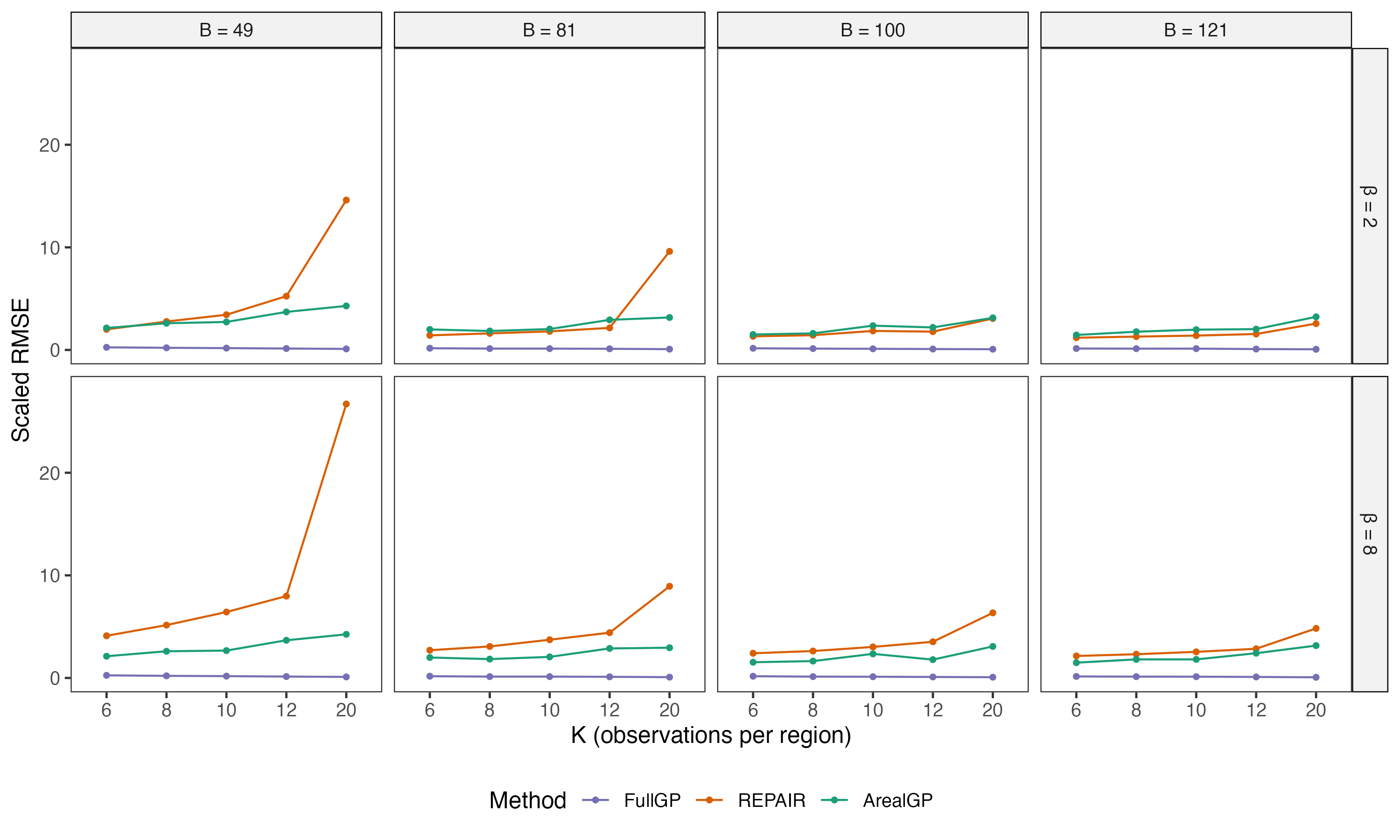}
        \caption{Scaled RMSE of $\hat{\tau}^2$.}
        \label{fig:tausq_rmse_lines}
    \end{subfigure}
    
    \vspace{0.5em}
    
    \begin{subfigure}{0.9\textwidth}
        \centering
        \includegraphics[width=\textwidth]{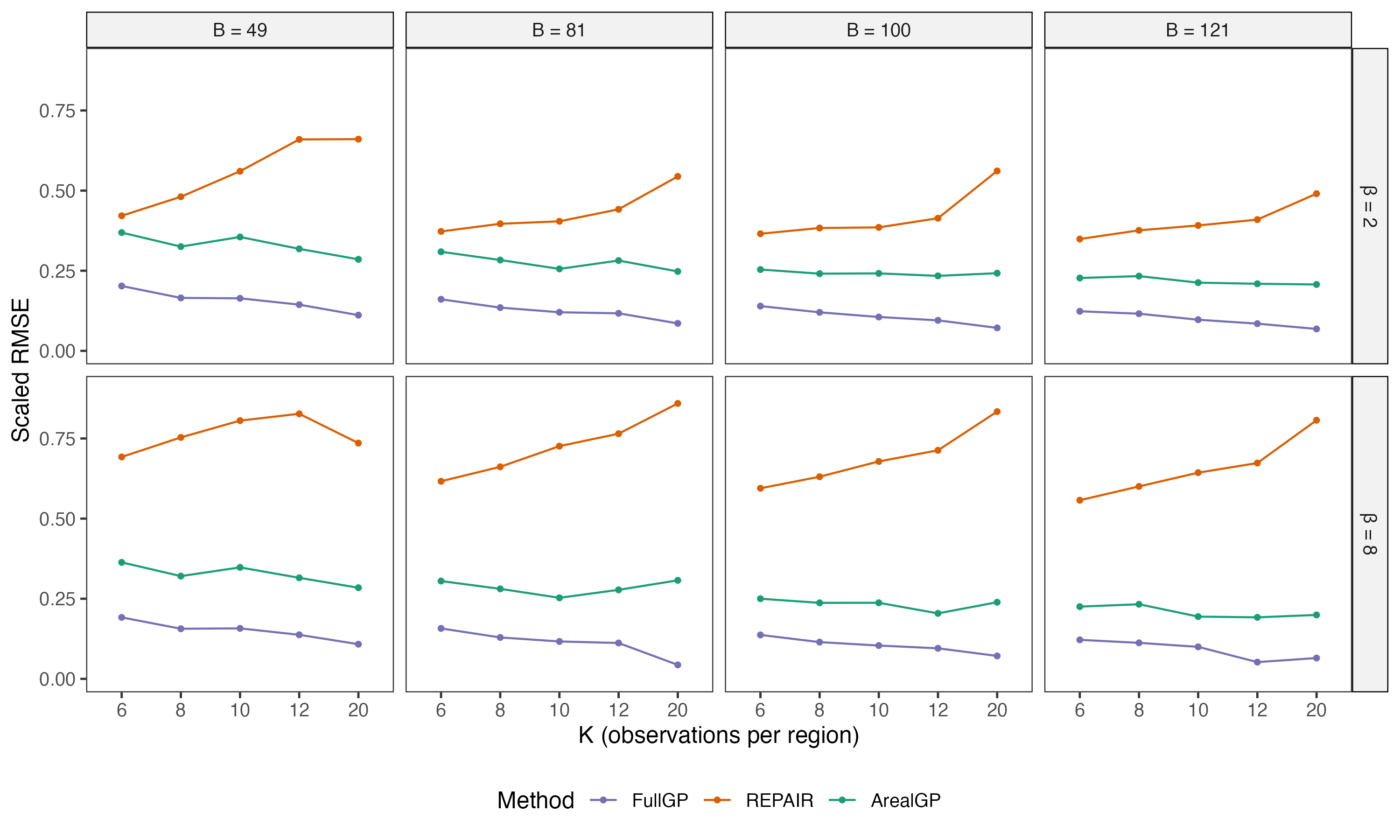}
        \caption{Scaled RMSE of $\widehat{\sigma^2 \phi}$.}
        \label{fig:sigmasqphi_rmse_lines}
    \end{subfigure}
    
    \caption{Scaled RMSE for covariance-related parameters across simulation settings.}
    \label{fig:covparam_rmse_lines}
\end{figure}

Estimation of covariance parameters such as $\sigma^2$ and $\phi$ is often challenging in finite samples and this difficulty is reflected in our simulations through the comparatively higher RMSE values for variance-related parameters relative to the regression coefficient. This is consistent with the broader literature: \citet{zhang2004inconsistent} showed that under fixed-domain asymptotics, the variance and range parameters of the Matérn class cannot be estimated consistently in isolation, and only certain microergodic combinations remain well-identified. While this non-identifiability is specific to the fixed-domain regime, it underscores a more general difficulty in disentangling covariance parameters that persists even at moderate sample sizes.

In our setting, the finite-sample imprecision in estimating $\sigma^2$ and $\phi$ is likely further amplified by the latent permutation structure and by the use of a mean-field variational approximation. The mean-field factorization substantially simplifies computation, but it can also reduce accuracy for nuisance covariance parameters, which is consistent with the larger RMSE values observed in our experiments.

\end{document}